%% file: main.tex
\documentclass[12pt,reqno]{amsart}
\usepackage[utf8]{inputenc}

\input prepictex
\input pictexwd
\input postpictex

\usepackage[pdfencoding=auto, psdextra]{hyperref}
\pdfoutput=1

\usepackage{ytableau}
\ytableausetup{centertableaux, boxsize=12pt}

\usepackage{longtable,array}
\usepackage{graphicx}
\usepackage{kbordermatrix}
\usepackage{amsthm,amsmath,amssymb,bbm,fullpage,mathtools,stmaryrd}
\usepackage{url,hyperref}
\usepackage[all]{xy}
\usepackage{caption}
\usepackage{subcaption}
\usepackage{comment}
\usepackage{tkz-graph}
\usepackage{cellspace}

\newtheorem{theorem}{Theorem}[section]
\newtheorem{lemma}[theorem]{Lemma}
\newtheorem{proposition}[theorem]{Proposition}
\newtheorem{corollary}[theorem]{Corollary}

\newtheorem{conjecture}[theorem]{Conjecture}
\theoremstyle{definition}
  \newtheorem{remark}[theorem]{Remark}
  \newtheorem{definition}[theorem]{Definition}
  \newtheorem{example}[theorem]{Example}

\newcommand\bbC{\mathbb{C}}
\newcommand\bbZ{\mathbb{Z}}
\newcommand\ZZ{\mathbb{Z}}
\newcommand\calB{\mathcal{B}}
\newcommand\calD{\mathcal{D}}
\newcommand\calL{\mathcal{L}}
\newcommand\calM{\mathcal{M}}
\newcommand\calS{\mathcal{S}}
\newcommand\calT{\mathcal{T}}
\newcommand\fS{\mathfrak{S}}
\newcommand\ch{\mathrm{ch}}
\newcommand\Num{\mathrm{Num}}
\newcommand\SSYT{\mathrm{SSYT}}
\newcommand\SYT{\mathrm{SYT}}
\newcommand\Tab{\mathrm{Tab}}
\newcommand\n{\!\!\!\!-1} 

\renewcommand{\L}[1]{\mathcal{L}_{#1}} 
\renewcommand{\S}[1]{\mathcal{S}_{(#1)}} 
\newcommand{\M}[1]{\mathcal{M}_{(#1)}} 
\DeclareMathOperator{\Hom}{Hom}
\DeclareMathOperator{\im}{im}
\DeclareMathOperator{\Ind}{Ind}
\DeclareMathOperator{\Res}{Res}

\DeclareMathOperator{\Frob}{Frob}
\newcommand{\largewedge}{\mbox{\Large $\wedge$}}

\definecolor{applegreen}{rgb}{0.55,0.71,0.0}

\usepackage{float}
\usepackage{tikz}
\usetikzlibrary{cd}
\tikzstyle{vertex}=[circle, draw, inner sep=0pt, minimum size=4pt]
\newcommand{\vertex}{\node[vertex]}

\input{savebox.tex}

\title{On the Strength of Chromatic Symmetric Homology for graphs}

\begin{document}
\author[Chandler]{Alex Chandler}
\address[A.\ Chandler]{University of Vienna}
\email{alex.chandler@univie.ac.at}
\urladdr{\url{https://alexchandler.wordpress.ncsu.edu/}}

\author[Sazdanovic]{Radmila Sazdanovic}
\address[R.\ Sazdanovic]{}
\email{rsazdanovic@math.ncsu.edu}
\urladdr{\url{https://sazdanovic.wordpress.ncsu.edu/}}

\author[Stella]{Salvatore Stella}
\address[S.\ Stella]{University of Leicester}
\email{s.stella@leicester.ac.uk}
\urladdr{\url{http://math.haifa.ac.il/~stella/index.html}}

\author[Yip]{Martha Yip}
\address[M.\ Yip]{University of Kentucky}
\email{martha.yip@uky.edu}
\urladdr{\url{http://www.ms.uky.edu/~myip/}}

\begin{abstract}
  In this paper, we investigate the strength of chromatic symmetric homology as a graph invariant.
  Chromatic symmetric homology is a lift of the chromatic symmetric function for graphs to a homological setting, and its Frobenius characteristic is a $q,t$ generalization of the chromatic symmetric function.
  We exhibit three pairs of graphs where each pair has the same chromatic symmetric function but distinct homology.
  We also show that integral chromatic symmetric homology contains torsion, and based on computations, conjecture that $\bbZ_2$-torsion in bigrading $(1,0)$ detects nonplanarity in the graph.
\end{abstract}

\thanks{
  We would like to thank ICERM for the opportunity to work intensively on this project at a Collaborate{@}ICERM meeting in February 2019.
  RS was partially supported by the Simons Collaboration Grant 318086.
  SS was partially supported by ISF grant 1144/16.
  MY is partially supported by the Simons Collaboration Grant 429920.
}

\maketitle
\parskip=5pt

\section{Introduction}

In 1995, Stanley \cite{Sta95} defined a symmetric function generalization of the chromatic polynomial of graphs.
This graph invariant has been categorified to the chromatic symmetric homology by the second and fourth authors~\cite{SY18}, in the spirit of Khovanov and related link and graph homology theories.
Given a graph $G$ with $n$ vertices, the chromatic symmetric homology $H_{i,j}(G;\bbC)$ is a bigraded  $\bbC[\fS_n]$-module, whose bigraded Frobenius characteristic $\Frob_G(q,t)$ is equal to Stanley's chromatic symmetric function $X_G$ when specialized at $q=t=1$.
This construction is not restricted to $\bbC$; in particular it works over $\bbZ$ and $\bbZ_p$ for prime numbers $p \geq 2$.
The main results of this paper offer answers and insights to the following questions:
\begin{enumerate}
    \item
      Does the categorified invariant detect more information than the original invariant?
    \item
      In the case that the categorified invariant takes the form of integral homology, does it contain torsion?
      If so, does torsion detect any additional information about the graph?
\end{enumerate}

\noindent{\bf Theorem~\ref{thm.stronger}.} \textit{Chromatic symmetric homology over $\bbC$ is strictly stronger than $X_G$.}

This result is based on finding explicit examples of pairs of graphs which are distinguished by the chromatic symmetric homology but not by the chromatic symmetric function.
A similar result holds for Khovanov homology and the Jones polynomial~\cite{Bar1}.

To fully explore the strength of the chromatic symmetric homology we also work over $\bbZ$.
As the chromatic symmetric function is recovered by specializing the Frobenius characteristic of the homology, this implies that any torsion in chromatic symmetric homology is not directly captured in this decategorification process.
Thus, torsion in chromatic symmetric homology has the potential to carry extra information about the graph which is undetected by the chromatic symmetric function.
Analyzing torsion in related homology theories has turned out to be fruitful, see for example~\cite{CLSS,HGR,HGRP,Kan18,LS,MS,MPSWY,PPS, PS,S1,S2,Sum19}.

\noindent{\bf Theorems~\ref{thm.K5} and~\ref{thm.K33}.}
\textit{Over the integral domain $\mathbb{Z}$, the chromatic symmetric homology of the nonplanar graphs $K_5$ and $K_{3,3}$ have $\mathbb{Z}_2$-torsion.}

With Kuratwoski's theorem in mind, the above theorems indicate that chromatic symmetric homology may provide an algebraic obstruction for the planarity of graphs.
There are $14$ nonplanar graphs in the set of $143$ connected graphs with up to $6$ vertices.
Our computations (see Appendix~\ref{computations}) show that only these $14$ nonplanar graphs contains $\ZZ_2$-torsion in bidegree $(1,0)$ homology.
Thus we make the following conjecture.

\noindent{\bf Conjecture~\ref{tor.conj}} \textit{A graph $G$ is nonplanar if and only if the chromatic symmetric homology in bidegree $(1,0)$ contains $\bbZ_2$-torsion.}

We have found that other types of torsion can appear in chromatic symmetric homology.
Our computations (see Section~\ref{z3torsion}) show that the homology of the star graph on $7$ vertices contains $\bbZ_3$-torsion in bidegree $(1,0)$.  This raises the question of whether every type of torsion can appear in chromatic symmetric homology, and whether it characterizes other graph properties.

Due to the computational complexity of the problem, we restrict our attention to computing chromatic symmetric homology only in $q$-degree zero, and so we write $H_i(G)=H_{i,0}(G)$ to simplify the notation.
It turns out that for the purposes of answering the two questions outlined above, this restriction suffices.

It is worth noting that the following theoretical results played an essential role in cutting down the computation time.
Given a graph with $n$ vertices and $m$ edges, the homological width of the chromatic symmetric chain complex is $m+1$.
In~\cite{Chan19, CS19} the first and second authors construct the broken circuit model for chromatic symmetric homology: a chain complex of homological width $n$ that is quasi-isomorphic to the chromatic symmetric chain complex.
The broken circuit model can be viewed as a categorification of Whitney's broken circuit theorem for the chromatic symmetric polynomial.
For graphs with many edges, this gives a significant reduction in homological width.
For example, for complete graphs, the homological width of the broken circuit model grows linearly in $n$ as opposed to the one of the chromatic symmetric chain complex which grows quadratically in $n$. Furthermore, an algorithm proposed by Lampret \cite{lampret2019chain} uses algebraic Morse theory, via the so-called steepness matchings, to cancel large portions of chain complexes before computing their homology.
Both of these reductions are incorporated into our program to compute chromatic symmetric homology.
Our toy implementation is only capable of computing the $\bbZ$-module structure of the homology and it is available at \cite{code}.
Due to the inefficient way in which differentials are built, it is only usable on graphs with no more than 7 vertices and relatively few edges.
We list the result of computations on all connected graphs with at most 6 vertices in Appendix~\ref{computations}.

This paper is organized as follows.
In Section \ref{sec-two}, we recall the construction of chromatic symmetric homology following~\cite{SY18}, and point out how the computation simplifies when we restrict to zeroth quantum grading.
We also lay the algebraic foundation necessary for computing chromatic symmetric homology in homological degree $1$, which involves developing straightening laws for numberings which are parallel to a dual version for oriented column tabloids.
In Section \ref{sec-three}, we prove Theorem \ref{thm.stronger} by providing examples of pairs of graphs with equal chromatic symmetric functions but different chromatic symmetric homology.
We explicitly compute the first chromatic symmetric homology $H_1$ for these graphs, restricted to specific Specht modules.
In Section \ref{sec-four}, we prove Theorems~\ref{thm.K5} and~\ref{thm.K33} by explicitly finding a generator which has order two in homology.
Last but not least, assuming that Conjecture~\ref{tor.conj} holds, we provide an infinite family of pairs of graphs with the same chromatic symmetric function, but distinct chromatic symmetric homology.

\section{Computing q-degree zero homology}
\label{sec-two}
We briefly recall the construction of chromatic symmetric homology for a graph $G$, before moving on to describe the specific details required for computing the homology in $q$-degree zero.
A complete description of the $q$-graded homology is in the paper~\cite{SY18}.

\subsection{Defining chromatic symmetric homology}
Let $G$ be a graph with vertex set $[n]:=\{1,\ldots, n\}$.
An edge incident to the vertices $i$ and $j$ will be denoted by $e_{ij}$, if $i<j$.  We order the set of edges $E(G)$ lexicographically.

If $G$ has $m$ edges, then the $2^m$ subsets $F\subseteq E(G)$ are the {\em spanning subgraphs} of $G$.
These have the structure of a Boolean lattice $\calB(G)$, ordered by the reverse inclusion.
In the Hasse diagram of the lattice $\calB(G)$, we direct an edge $\varepsilon(F,F')$ from a spanning subgraph $F$ to a spanning subgraph $F'$ if and only if $F'$ can be obtained by removing an edge $e$ from $F$. The {\em sign} of $\varepsilon(F,F')$ is $-1$ to the number of edges in $F$ less than $e$ with respect to the lexicographic ordering on $E(G)$, and we denote this by $\mathrm{sgn}(\varepsilon(F,F'))$. See Figure~\ref{fig.k3} for an example.

\subsubsection{Modules}
We refer the reader to Fulton~\cite[Chapter 7]{Ful97} for background on representations of the symmetric group $\fS_n$.
For any subset $B\subseteq [n]$, let $\fS_B$ denote the subgroup of permutations of $B$.
A {\em set partition} $\beta=\{B_1,\ldots, B_r\}$ of $[n]$ with $r$ parts is a set of $r$ nonempty subsets of $[n]$ which are pairwise disjoint and whose union is $[n]$.

Let $\calS_\lambda$ denote the irreducible $\fS_n$-module indexed by the partition $\lambda\vdash n$.
For $b \in \bbZ_{\geq1}$, let $\L{b}$ denote the $q$-graded $\fS_b$-module
\[
  \L{b} = \largewedge^* \S{b-1,1} = \bigoplus_{i=0}^{b-1} \S{b-i,1^{i}} ,
\]
such that $\S{b-i, 1^{i}}$ is in $q$-degree $i$.

We now assign a $q$-graded $\fS_n$-module to each spanning subgraph $F \subseteq E(G)$.
If $F$ has $r$ connected components $B_1,\ldots, B_r$ of sizes $b_1,\ldots, b_r$, then $\beta = \beta(F) = \{B_1,\ldots, B_r\}$ is the set partition associated to $F$.
The {\em Young subgroup} associated to $F$ is $\fS_{\beta(F)} = \fS_{B_1}\times \cdots \times\fS_{B_r}$.
We assign the $q$-graded $\fS_n$-module
\[
  \calM_F = \Ind_{\fS_{\beta(F)}}^{\fS_n} \left( \L{b_1} \otimes \cdots \otimes \L{b_r} \right)
\]
to the spanning subgraph $F$.

Let $(\calM_F)_j$ denote the $j$-th graded piece of the module $\calM_F$.
Then for $i, j\in \bbZ_{\geq0}$, the {\em $i$-th chain module} for $G$ in $q$-degree $j$ is
\[
  C_{i,j}(G) = \bigoplus_{|F|=i} (\calM_F)_j,
\]
which is a sum over the spanning subgraphs of $G$ with $i$ edges.

\subsubsection{Differentials}
Given spanning subgraphs $F$ and $F'$ of $G$ where $F'=F-e$ is obtained by removing the edge $e$ from $F$, we define an {\em edge map} $d_{\varepsilon(F,F')}: \calM_F \rightarrow \calM_{F'}$.
We will quote some results and refer the reader to~\cite[Section 2]{SY18} for more details.

In the special case that removing $e$ from $F$ does not disconnect a component, then $d_{\varepsilon(F,F')}$ is simply the identity map.

Otherwise, we first consider the simplest case where $F$ is a connected spanning subgraph of $G$ and removing $e$ disconnects $F$ into two components $A$ and $B$.
Suppose $|A|=a$ and $|B|=b$.  Consider the $(\fS_a \times \fS_b)$-module
\[
  \calT = \left(\S{a-1,1} \otimes \S{b} \right) \oplus \left(\S{a} \otimes \S{b-1,1} \right).
\]
If $\calM$ is a $q$-graded $\fS_n$-module, let $\calM\{k\}$ denote the forward $q$-degree shift of $\calM$ by $k$.
Since $\calM_{F'}=\Ind(\calL_a\otimes \calL_b)$, $\calL_a\otimes \calL_b\cong  \largewedge^*\calT$, and $\Res \calM_F \cong (\largewedge^*\calT) \oplus (\largewedge^* \calT) \{1\}$~\cite[Lemma 2.6]{SY18}, then applying Frobenius reciprocity (and noting that $\Ind V$ and $\mathrm{coInd}\, V$ are naturally isomorphic under the Nakayama isomorphism), we have
\begin{equation}
  \label{eqn.frobrep}
  \Hom_{\fS_n}(\calM_F, \calM_{F'})
  \cong \Hom_{\fS_a \times\fS_b} \left((\largewedge^*\calT) \oplus (\largewedge^* \calT) \{1\},  \largewedge^*\calT \right).
\end{equation}
We choose $d_{\varepsilon(F,F')} \in \Hom_{\fS_n}(\calM_F, \calM_{F'})$ to be the map that corresponds to the $(\fS_a\times \fS_b)$-module map that is identity on $(\largewedge^*\calT)$, and zero on $(\largewedge^*\calT)\{1\}$.

In the more general case where $F$ has $r$ connected components $B_1,\ldots, B_r$ of sizes $b_1,\ldots, b_r$ and the removal of $e$ decomposes $B_r$ into two components $A$ and $B$ of sizes $a$ and $b$ respectively, let $d_\zeta: \calL_{a+b} \rightarrow \Ind_{\fS_A\times \fS_B}^{\fS_{B_r}}(\calL_a\otimes \calL_b)$ be the map defined previously, and let $\mathcal{N} = \calL_{b_1}\otimes \cdots \otimes \calL_{b_{r-1}}$.
The map $d_{\varepsilon(F,F')}: \calM_F \rightarrow \calM_{F'}$ is then defined by
\[
  d_{\varepsilon(F,F')} = \Ind_{\fS_{B_1}\times \cdots \times\fS_{B_{r-1}}\times \fS_{B_r}}^{\fS_n} (\mathrm{id}_{\mathcal{N}} \otimes d_\zeta).
\]

Finally, for $i\in \bbZ_{\geq0}$, we define the {\em $i$-th chain map} $d_i: C_i(G) \rightarrow C_{i-1}(G)$ by letting
\[
  d_i = \sum_{\varepsilon} \mathrm{sgn}(\varepsilon) d_\varepsilon,
\]
where the sum is over all edges $\varepsilon$ in $\calB(G)$ which join a subgraph of $G$ with $i$ edges to a subgraph with $i-1$ edges.
See Appendix~\ref{app.K3} for an example on the complete graph $K_3$.

It was shown in~\cite[Proposition 2.10]{SY18} that $d$ is indeed a differential, thus for $i,j \in \bbZ_{\geq0}$, the {\em $(i,j)$-th homology} of $G$ is
\[
  H_{i,j}(G;\bbC) = \ker d_{i,j}/\im d_{i+1,j},
\]
and the associated {\em bigraded Frobenius series} is
\[
  \Frob_G(q,t) = \sum_{i,j\geq0} (-1)^{i+j}t^iq^j \ch(H_{i,j}(G;\bbC)),
\]
where $\ch$ is the Frobenius characteristic map from the ring of representations of the symmetric groups to the ring of symmetric functions.
We have $\Frob_G(1,1)=X_G$, which shows that chromatic symmetric homology is a lift of the chromatic symmetric function~\cite[Theorem 2.13]{SY18}.

\subsection{Computing homology in q-degree zero}
The computation of homology in $q$-degree zero is easier to carry out, due to our choice of the edge maps arising from Frobenius reciprocity (see Equation~\ref{eqn.frobrep}).
The computations performed in Section~\ref{sec.pairs} and Section~\ref{sectorsion} are in $q$-degree zero only, and we see that even with this restriction, the homology still retains rich information about the graph.
From this point on, we will drop the second index in the notation for the chain module $C_{i,j}$ and the homology $H_{i,j}$ as $j$ is always $0$.

First, in $q$-degree zero, the module associated to each spanning subgraph $F\subseteq E(G)$ with connected components $B_1,\ldots, B_r$ is simply the {\em permutation module}
\[
  \calM_F = \Ind_{\fS_{B_1} \times \cdots \times \fS_{B_r}}^{\fS_n} \left( \S{b_1} \otimes \cdots \otimes \S{b_r} \right),
\]
so the $i$-th chain module $C_{i}(G)$ of the graph is a direct sum of $\binom{m}{i}$ permutation modules of $\fS_n$.
If $\lambda$ is the partition whose parts are the sizes of the connected components of $F$, then $\calM_F \cong \calM_\lambda = \bbC[\fS_n] \otimes_{\bbC[\fS_\lambda]} \S{n}$.

As for the edge maps, we first revisit the case where $F$ is a connected spanning subgraph of $G$ and removing $e$ disconnects $F$ into two components $A$ and $B$ which forms $F'$.
Suppose $|A|=a$ and $|B|=b$.
Then we have $\calM_F = \S{a+b}$ and
\[
  \calM_{F'}
  =
  \Ind_{\fS_a\times \fS_b}^{\fS_{a+b}} (\S{a}\otimes \S{b})
  \cong
  \S{a+b} \oplus \bigoplus_{\lambda \rhd (a,b)} \calS_\lambda^{\oplus K_{\lambda, (a,b)} },
\]
where the direct sum is over partitions $\lambda$ which dominate $(a,b)$, and the {\em Kostka number} $K_{\lambda, (a,b)}$ is the number of semistandard Young tableaux of shape $\lambda$ and weight $(a,b)$.
By Frobenius reciprocity, the edge map
\[
  d_\zeta : \S{a+b} \rightarrow \Ind_{\fS_a\times \fS_b}^{\fS_{a+b}} (\S{a}\otimes \S{b})
\]
is simply the inclusion map, and so is the induced edge map $d_{\varepsilon(F,F')}:\calM_F \rightarrow \calM_{F'}$ in the general case when $F$ has more than one connected component.
Since each module $\calM_F$ is cyclically generated, then $d_{\varepsilon(F,F')}$ is completely determined by specifying the image of a cyclic generator for $\calM_F$.

\subsection{Restriction to Specht modules}

We would like to describe the chromatic homology in terms of Specht modules.
Schur's Lemma states that the only nontrivial morphisms between irreducible modules are homothetys, so it suffices to restrict the computation of the homology to Specht modules corresponding to a fixed partition inside each permutation module.
Furthermore, since each Specht module is cyclically generated, then our inclusion maps are completely determined by specifying the image of a cyclic generator for each Specht module.
We next describe how we will achieve these computations systematically.

\subsubsection{Young symmetrizers}
In this section, we will primarily follow the exposition in Fulton~\cite[Section 7]{Ful97}.
In order to compute the homology, it is important for us keep track of the Young subgroup from which we are inducing (not just up to isomorphism), thus, where we depart from Fulton, details will be provided.

Let $\lambda=(\lambda_1,\ldots, \lambda_r) \vdash n$ be a partition with $r$ parts.
A {\em Ferrers diagram} of shape $\lambda$ consists of $r$ rows of boxes justified to the left such that the $i$-th row from the top consists of $\lambda_i$ boxes.

A {\em numbering} of shape $\lambda \vdash n$ is a filling of the Ferrers diagram of shape $\lambda$ with the numbers $1,2,\ldots, n$ each appearing exactly once.
A {\em standard Young tableau} of shape $\lambda$ is a numbering of shape $\lambda$ whose rows increase strictly from left to right, and whose columns increase strictly from top to bottom.
A {\em semistandard Young tableau} of shape $\lambda$ is a filling of shape $\lambda$ such that the numbers $1,2,\ldots,n$ may repeat, and whose rows increase weakly while columns increase strictly.
We let $\Num(\lambda)$ denote the set of $n!$ numberings of shape $\lambda \vdash n$.
The symmetric group $\fS_n$ acts on $\Num(\lambda)$ by permuting the entries in a numbering.

\begin{definition}
  \label{def.totalorder}
  Given a numbering $T$, let $T(i,j)$ denote the entry in the $i$-th row and $j$-th column of $T$ (that is, we use the same indexing convention as that for a matrix).
  We define a total order $\preccurlyeq$ on the set of numberings of a fixed shape as follows.
  If $T$ and $S$ are numberings of shape $\lambda$ such that the $i$-th row is the lowest row in which the numberings are different, the $j$-th column is the rightmost column in that row in which the numberings are different, and $T(i,j) > S(i,j)$, then we say that $T \succcurlyeq S$.
\end{definition}

\begin{example}
  The order on the numberings of shape $(2,1)$ is
  \[
    \ytableaushort{12,3} \succcurlyeq
    \ytableaushort{21,3} \succcurlyeq
    \ytableaushort{13,2} \succcurlyeq
    \ytableaushort{31,2} \succcurlyeq
    \ytableaushort{23,1} \succcurlyeq
    \ytableaushort{32,1}.
  \]
\end{example}

The {\em row group} $R(T)\subseteq \fS_n$ is the subset of permutations that permute elements within each row of $T$, and the {\em column group} $C(T)\subseteq \fS_n$ is the subset of permutations that permute elements within each column of $T$.
The corresponding {\em Young symmetrizers} in $\bbC[\fS_n]$ are
\[
  a_T = \sum_{\rho\in R(T)} \rho,
  \qquad
  b_T = \sum_{\zeta \in C(T)} \mathrm{sgn}(\zeta) \zeta,
  \qquad
  c_T = b_Ta_T.
\]
For any $\pi \in \fS_n$, we have
\[
  R(\pi \cdot T) = \pi R(T) \pi^{-1}
  \qquad \hbox{and} \qquad
  C(\pi \cdot T) = \pi C(T) \pi^{-1},
\]
so that
\begin{equation}
  \label{eqn.abcproperties}
  a_{\pi \cdot T} = \pi a_T\pi^{-1},
  \qquad
  b_{\pi \cdot T} = \pi b_T\pi^{-1},
  \qquad \hbox{and}\qquad
  c_{\pi \cdot T} = \pi c_T\pi^{-1}.
\end{equation}

\begin{definition}
  For any numberings $S,T\in \Num(\lambda)$, we define the permutation $\sigma_{T,S}\in \fS_n$ by
  \[
    \sigma_{T,S} \cdot T = S.
  \]
\end{definition}
Note that for any $\pi\in \fS_n$,
\begin{equation}
  \label{eqn.sigmaproperties}
  \sigma_{T,S}^{-1} = \sigma_{S,T},
  \qquad\hbox{and}\qquad
  \sigma_{\pi\cdot T,S}^{-1} = \pi \sigma_{T,S}^{-1}.
\end{equation}

\begin{definition}
  Let $F$ be a spanning subgraph of $G$, and let $\lambda \vdash n$ be the partition whose parts are the sizes of the connected components of $F$.
  The {\em numbering $T(F)$ associated to $F$} is the numbering of shape $\lambda$ such that each row consists of the elements in a connected component of $F$ arranged in increasing order, and rows of $T(F)$ having the same size are ordered so that the minimum element in each row is increasing down the first column.

  Let $T=T(F)$.
  The {\em $q$-degree zero permutation module $\calM_T$ associated to the numbering $T$} is cyclically generated by the Young symmetrizer $a_T$:
  \[
    \calM_F= \calM_T = \bbC[\fS_n] \cdot a_T.
  \]
\end{definition}

Note that for each $S\in \Num(\lambda)$, we have $\sigma_{T,S}a_T = a_S \sigma_{T,S}$.
Furthermore, if $\rho\in R(T)$, then $\rho a_T = a_{\rho \cdot T} = a_T$.
This leads to the following description of a basis for $\calM_T$.
\begin{proposition}
  \label{prop.Mbasis}
  Let $F$ be a spanning subgraph of $G$ with associated numbering $T$ of shape $\lambda$.
  Then
  \[
    \calM_T = \mathrm{span}\{a_S\sigma_{T,S} \mid S\in \Tab(\lambda)\},
  \]
  where $\Tab(\lambda)$ is the set of numberings of shape $\lambda$ whose rows are increasing.
\end{proposition}

\begin{definition}
  For any numberings $S$ and $T$ of shape $\lambda$, let
  \[
    v_T^S = \sigma_{T,S}b_Ta_T = b_Sa_S\sigma_{T,S}\in \calM_T.
  \]
\end{definition}

\begin{lemma}
  \label{lem.STU}
  Let $S,T,U$ be numberings of shape $\lambda$, and suppose $U=\pi\cdot T$ where $\pi\in\fS_n$.
  Then $v_U^S = v_T^S \pi^{-1}$.
\end{lemma}
\begin{proof}
  We have
  \begin{align*}
    v_U^S &= v_{\pi\cdot T}^S = \sigma_{\pi\cdot T, S} \,c_{\pi \cdot T}\\
    &= (\sigma_{\pi\cdot T,S}^{-1})^{-1}\,c_{\pi\cdot T} = (\pi\sigma_{T,S}^{-1})^{-1}\,c_{\pi\cdot T} &\hbox{by Equation~\eqref{eqn.sigmaproperties}}\\
    &= \sigma_{T,S}\, \pi^{-1}c_{\pi\cdot T} = \sigma_{T,S}\, c_{T}\pi^{-1} &\hbox{by Equation~\eqref{eqn.abcproperties}}\\
    &= \sigma_{T,S}b_{T} a_{T}\pi^{-1}\\
    &= v_T^S\pi^{-1}.
  \end{align*}
\end{proof}

\begin{proposition}
  \label{prop.STU}
  Let $S,T,U$ be numberings of shape $\lambda$.
  \begin{enumerate}
  \item[(a)]
    If $U=\pi\cdot T$ where $\pi\in R(T)$, then $v_U^S = v_T^S$.
  \item[(b)]
    If $U$ can be obtained from $T$ by permuting two rows of length $\ell$, then $v_U^S= (-1)^\ell v_T^S$.
  \end{enumerate}
\end{proposition}
\begin{proof}
  In both cases, we make use of Lemma~\ref{lem.STU}.
  In the first case, $v_U^S = v_T^S \pi^{-1} = v_T^S$, since $\pi^{-1}\in R(T)$.
  In the second case, suppose $U= \tau_1 \cdots \tau_\ell T$, where $\tau_j\in \fS_n$ is the transposition that exchanges the entries in the two rows of interest in the $j$-th column of $T$.
  For each $j$, we have $\tau_j\in C(T)$, so $v_U^S = v_T^S \tau_\ell \cdots \tau_1 = (-1)^\ell v_T^S$.
\end{proof}

\begin{definition}
  The {\em $q$-degree zero Specht module $\calS_T$ associated to the numbering $T$} is cyclically generated by the Young symmetrizer $c_T = b_Ta_T$:
  \[
    \calS_T = \bbC[\fS_n]\cdot c_T,
  \]
  and $\calS_T \cong \calS_\lambda$.
\end{definition}

\begin{remark}
  By Proposition~\ref{prop.STU} we may assume without loss of generality that the Specht modules $\calS_T$ which will appear in our computations are indexed by numberings $T$ which are increasing along each row, and furthermore, if two rows of $T$ have the same length, then we may assume that the rows are ordered so that the row with a smaller element in the first column lies above the row with a larger element in the first column.
  See Appendix~\ref{app.K3} for an example.
\end{remark}

\begin{proposition}[{\cite[Section 7.2 Proposition 2]{Ful97}}]
  \label{prop.Sbasis}
  Let $F$ be a spanning subgraph of $G$ with associated numbering $T$ of shape $\lambda$.
  Then
  \[
    \calS_T
    =
    \mathrm{span} \left\{ b_Sa_S \sigma_{T,S} \mid S \in \SYT(\lambda)\right\}
    =
    \mathrm{span} \left\{ v_T^S \mid S \in \SYT(\lambda)\right\},
  \]
  where $\SYT(\lambda)$ is the set of standard Young tableaux of shape $\lambda$.
\end{proposition}

\begin{remark}
  It will be helpful in the following sections to remember that the subscript $T$ in $v_T^S$ indexes the copy of the module $\calS_T$ that $v_T^S$ belongs to.
  The superscript $S$ indexes an element in the basis we have chosen for $\calS_T$.
\end{remark}

\begin{example}
  \label{eg.forappA}
  Let $n=3$ and for $i=1,2$, let $s_i\in \fS_3$ denote the transposition that exchanges $i$ and $i+1$.
  Then $\fS_3 = \langle s_1, s_2 \mid s_1^2=e, s_2^2=e, s_1s_2s_1=s_2s_1s_2\rangle$.
  Let
  \[
    \ytableausetup{notabloids}
    X_1 = \ytableaushort{12,3},\qquad
    X_2 = \ytableaushort{13,2},\qquad
    X_3 = \ytableaushort{23,1}.
  \]
  For example, the permutation module indexed by the numbering $X_3$ is generated by $a_{X_3}= e+s_2$ and has basis
  \[
    \calM_{X_3} = \bbC[\fS_3]  \cdot (e+s_2) = \mathrm{span}\{e+s_2, s_1 + s_1s_2, s_2s_1 + s_2s_1s_2 \},
  \]
  and the Specht submodule generated by $c_{X_3} = (e-s_1)(e+s_2)$ has basis
  \[
    \calS_{X_3} = \bbC[\fS_3] \cdot c_{X_3} = \mathrm{span}\left\{v_{X_3}^{X_1}, v_{X_3}^{X_2} \right\},
  \]
  where
  \begin{align*}
    v_{X_3}^{X_1} &= \sigma_{X_3,X_1} c_{X_3} = s_2s_1 (e-s_1)(e+s_2) = (e-s_2s_1s_2)(e+s_1)s_2s_1 = c_{X_1} s_2s_1,\\
    v_{X_3}^{X_2} &= \sigma_{X_3,X_2} c_{X_3} = s_1 (e-s_1)(e+s_2) = (e-s_1)(e+s_2s_1s_2)s_1 = c_{X_2}s_1.\\
  \end{align*}
\end{example}

\subsection{Computation of \texorpdfstring{$H_1(G;\bbC)$}{H1(G;C)}}
\label{sec.restrict}
In this section, we give an overview of the kinds of computations that will appear in the later sections.
Specifically in Section~\ref{sec.pairs}, we will compute the ($q$-degree zero) homology of certain graphs on $n=6$ vertices in homological degree one, restricted to the Specht module of isomorphism type $\lambda = (2^2)$ or $(2^2,1^2)$.
Since this is the case across all three examples exhibited in Section~\ref{sec.pairs}, it will be worthwhile to first discuss how to make these computations in slightly greater generality and for all $n$.

\subsubsection{Restriction to Specht modules}
We show how to compute $H_{1}(G;\bbC)$ restricted to Specht modules of type $\lambda = (2^k,1^{n-2k})$ for $k\geq1$, by choosing appropriate cyclic generators.

We will be computing
\[
  \xymatrix{C_2(G)\big|_{\calS_\lambda} \ar[r]^{d_{2}} & C_{1}(G)\big|_{\calS_\lambda} \ar[r]^{d_{1}} & C_{0}(G)\big|_{\calS_\lambda} \ar[r]^-{} & 0}.
\]

In homological degree zero, there is only one spanning subgraph without edges.
The chain group $C_0(G) = \calM_{F_{\emptyset}} \cong \M{1^n}$ is the regular represention of $\fS_n$,
where $F_\emptyset$ is the edgeless subgraph.
The multiplicity of $\calS_\lambda$ in $\bbC[\fS_n]$ is the number $f^\lambda = K_{\lambda, (1^n)}$ of standard Young tableaux of shape $\lambda$.
We list the tableaux $Y_1,\ldots, Y_{f^\lambda}$ in $\SYT(\lambda)$ with respect to the total order defined in Definition~\ref{def.totalorder}, and
\[
  C_{0}(G)\big|_{\calS_\lambda} = \bigoplus_{i=1}^{f^\lambda} \bbC[\fS_n]\cdot v_{Y_i}^{Y_1}.
\]

In homological degree one, there are $m=|E(G)|$ spanning subgraphs with exactly one edge, so $C_1(G) =\bigoplus_{i=1}^m \calM_{F_{e_i}}$.
If $F_{e_i}$ is the spanning subgraph containing the edge $e_i=(p,q)$, then the permutation module $\calM_{F_{e_i}}$ has the associated numbering $T(F_{e_i})$ of shape $\mu = (2,1^{n-2})$, and $\calM_{F_{e_i}} = \bbC[\fS_n]\cdot (e + (p\,q))\cong \calM_\mu$.

The multiplicity of $\calS_\lambda \in \calM_{F_{e_i}}$ is the number $K_{\lambda,\mu}$ of semistandard Young tableaux of shape $\lambda$ and weight $\mu$.
We next obtain numberings of shape $\lambda$ that will index these $K_{\lambda,\mu}$ Specht modules $\calS_\lambda$, by standardizing the set $\SSYT(\lambda,\mu)$ of semistandard Young tableaux of shape $\lambda$ and weight $\mu$ with respect to $T(F_{e_i})$ in the following way.
For any numbering $S$, the {\em word $w(S)$ of $S$} is obtained by reading the entries of the rows of $S$ from left to right, and from the top row to the bottom row (note that this is not the usual definition of a reading word for tableaux).
So, given $Y\in \SSYT(\lambda,\mu)$, let $w(Y) = y_1,\ldots, y_n$ be the word of $Y$, and let $w(T) = t_1,\ldots, t_n$ be the word of $T=T(F_{e_i})$.
From this we obtain a numbering $X$ of shape $\lambda$ by replacing the entry in $Y$ that corresponds to $y_k$ by $t_k$.
We list the numberings $X_{i}^1,\ldots, X_{i}^{K_{\lambda,\mu}}$ obtained by standardizing $\SSYT(\lambda,\mu)$ with respect to $T(F_{e_i})$, in the total order of Definition~\ref{def.totalorder}.

Observe that since $\mu = (2,1^{n-2})$ and $\lambda= (2^k,1^{n-2k})$ where $k\geq1$, then this standardization procedure guarantees that the first row of each numbering $X_{i}^{j}$ is $\ytableaushort{pq}$.
So $v_{X_i^j}^{Y_1}\in \calM_{F_{e_i}}$ and $\bbC[\fS_n]\cdot v_{X_i^j}^{Y_1} \cong \calS_\lambda$ for $j=1,\ldots, K_{\lambda,\mu}$.
Thus
\[
  C_1(G)\big|_{\calS_\lambda} =  \bigoplus_{i=1}^m \bigoplus_{j=1}^{K_{\lambda,\mu}} \bbC[\fS_n]\cdot v_{X_i^j}^{Y_1}.
\]

Lastly, we consider the chain module in homological degree two.
The subgraphs of $G$ with exactly two edges have connected components of partition type $(2^2,1^{n-4})$ or $(3,1^{n-3})$.
We are only concerned with Specht modules of type $\lambda = (2^k,1^{n-2k})$ with $k\geq2$ necessarily, and since $\lambda \not\!\rhd (3,1^{n-3})$, then $\calS_\lambda$ does not appear as a summand in a permutation module isomorphic to $\calM_{(3,1^{n-3})}$.
Hence, we only need to consider the spanning subgraphs with connected components of partition type $(2^2,1^{n-4})$.

So suppose $G$ has $h$ spanning subgraphs whose connected components has partition type $\nu=(2^2,1^{n-4})$.
List these spanning subgraphs with respect to the lexicographic order of their edge sets.
Suppose $F_{e_i,e_j}$ is the $k$-th such spanning subgraph and it contains the edges $e_i=(p,q)$ and $e_j=(r,s)$ with $p<r$.
The permutation module $\calM_{F_{e_i,e_j}}$ has the associated numbering $T(F_{e_i,e_j})$ of shape $\nu$, and $\calM_{F_{e_i,e_j}}= \bbC[\fS_n]\cdot (e+(p\,q))(e+(r\,s))\cong \calM_\nu$.

Similar to the previous case for $C_1(G)$, the multiplicity of $\calS_\lambda\in \calM_{F_{e_i,e_j}}$ is $K_{\lambda,\nu}$.
We list the numberings $W_k^1,\ldots, W_k^{K_{\lambda,\nu}}$ obtained by standardizing $\SSYT(\lambda,\nu)$ with respect to $T(F_{e_i,e_j})$, in the total order of Definition~\ref{def.totalorder}.
The standardization procedure guarantees that the top two rows of each numbering $W$ are $\ytableaushort{pq,rs}$, so $v_{W}^{Y_1}\in \calM_{F_{e_i,e_j}}$, and $\bbC[\fS_n]\cdot  v_{W_k^\ell}^{Y_1}\cong \calS_\lambda$ for $\ell=1,\ldots, K_{\lambda,\nu}$.
Thus
\[
  C_2(G)\big|_{\calS_\lambda} =  \bigoplus_{k=1}^h \bigoplus_{\ell=1}^{K_{\lambda,\nu}} \bbC[\fS_n]\cdot v_{W_{k}^{\ell}}^{Y_1}.
\]

\begin{example}
  \label{eg.210}
  Consider the following spanning subgraphs $F$, $F'$, and $F''$.
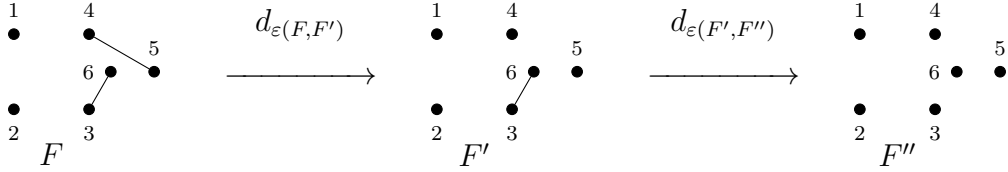
  \begin{center}
    \begin{tikzpicture}[scale=.5]
      \begin{scope}[xshift=0]
        \node at (-1,-1.2){$F$};
        \vertex[fill, label=above:\tiny{$1$}](v1) at (-2,2){};
        \vertex[fill, label=below:\tiny{$2$}](v2) at (-2,0){};
        \vertex[fill, label=below:\tiny{$3$}](v3) at (0,0){};
        \vertex[fill, label=above:\tiny{$4$}](v4) at (0,2){};
        \vertex[fill, label=above:\tiny{$5$}](v5) at (1.73,1){};
        \vertex[fill, label=left:\tiny{$6$}](v6) at (0.577,1){};
        \draw (v4)--(v5);
        \draw (v6)--(v3);
      \end{scope}
      \begin{scope}[xshift=160, yshift=27]
        \node[label=above:{$d_{\varepsilon(F,F')}$}] at (0,0) {$\xrightarrow{\hspace*{1.8cm}}$};
      \end{scope}
      \begin{scope}[xshift=320]
        \node at (-1,-1.2){$F'$};
        \vertex[fill, label=above:\tiny{$1$}](v1) at (-2,2){};
        \vertex[fill, label=below:\tiny{$2$}](v2) at (-2,0){};
        \vertex[fill, label=below:\tiny{$3$}](v3) at (0,0){};
        \vertex[fill, label=above:\tiny{$4$}](v4) at (0,2){};
        \vertex[fill, label=above:\tiny{$5$}](v5) at (1.73,1){};
        \vertex[fill, label=left:\tiny{$6$}](v6) at (0.577,1){};
        \draw (v6)--(v3);
      \end{scope}
      \begin{scope}[xshift=480, yshift=27]
        \node[label=above:{$d_{\varepsilon(F',F'')}$}] at (0,0) {$\xrightarrow{\hspace*{1.8cm}}$};
      \end{scope}
      \begin{scope}[xshift=640]
        \node at (-1,-1.2){$F''$};
        \vertex[fill, label=above:\tiny{$1$}](v1) at (-2,2){};
        \vertex[fill, label=below:\tiny{$2$}](v2) at (-2,0){};
        \vertex[fill, label=below:\tiny{$3$}](v3) at (0,0){};
        \vertex[fill, label=above:\tiny{$4$}](v4) at (0,2){};
        \vertex[fill, label=above:\tiny{$5$}](v5) at (1.73,1){};
        \vertex[fill, label=left:\tiny{$6$}](v6) at (0.577,1){};
      \end{scope}
    \end{tikzpicture}
  \end{center}
  The numberings associated to the subgraphs are
  \[
    T(F) = \ytableaushort{36,45,1,2}\,,
    \qquad
    T(F') = \ytableaushort{36,1,2,4,5}\,,
    \qquad\hbox{and}\qquad
    T(F'') = \ytableaushort{1,2,3,4,5,6},
  \]
  and the permutation modules associated to the subgraphs are
  \begin{align*}
    \calM_F &= \bbC[\fS_n]\cdot (e+(36))(e+(45)) \cong \M{2^2,1^2},\\
    \calM_{F'} &= \bbC[\fS_n]\cdot (e+(36)) \cong \M{2,1^4},\\
    \calM_{F''} &= \bbC[\fS_n] \cong \M{1^6}.
  \end{align*}
  Let $\lambda=(2^3)$.
  When $\mu = (2^2,1^2)$, $K_{\lambda,\mu}=1$ and the unique semistandard Young tableau of shape $\lambda$ and weight $\mu$ is
  \[
    \ytableaushort{11,22,34}\,,
    \hbox{ which is standardized to }
    \ytableaushort{36,45,12}\,.
  \]
  When $\mu=(2,1^4)$, $K_{\lambda,\mu}=2$ and the two semistandard Young tableaux of shape $\lambda$ and weight $\mu$ are
  \[
    \ytableaushort{11,23,45}\ \hbox{ and }\ \ytableaushort{11,24,35}\,,
    \hbox{ which are standardized to }
    \ytableaushort{36,12,45}\ \hbox{ and }\ \ytableaushort{36,14,25}\,.
  \]
  And when $\mu = (1^6)$, there are $K_{\lambda,\mu} = 5$ standard Young tableaux of shape $\lambda$
  \[
    Y_1=\ytableaushort{12,34,56}\,,\
    Y_2=\ytableaushort{13,24,56}\,,\
    Y_3=\ytableaushort{12,34,46}\,,\
    Y_4=\ytableaushort{13,25,46}\,,\
    Y_5=\ytableaushort{14,25,36}\,.
  \]
  The Specht submodules of type $\lambda$ in these permutation modules are generated by
  \[
    \ytableausetup{smalltableaux}
    \calM_F  \supseteq \bbC[\fS_6] \cdot v_{\,\scalebox{.75}{\ytableaushort{36,45,12}}}^{Y_1}\,,
    \qquad
    \calM_{F'} \supseteq \bbC[\fS_6] \cdot v_{\,\scalebox{.75}{\ytableaushort{36,12,45}}}^{Y_1} \oplus \bbC[\fS_6] \cdot v_{\,\scalebox{.75}{\ytableaushort{36,14,25}}}^{Y_1}\,,
    \qquad
    \calM_{F''} \supseteq \bigoplus_{i=1}^5 \bbC[\fS_6]\cdot v_{Y_i}^{Y_1}.
  \]
\end{example}

To compute the edge maps $d_{\varepsilon(F,F')}$ and $d_{\varepsilon(F',F'')}$, we need to describe some straightening laws for numberings.

\begin{remark}
  For some values of $\mu$ and $\lambda\vdash n$, it is less obvious how one should uniformly choose generators for Specht modules.
  For instance, consider the following example:
  
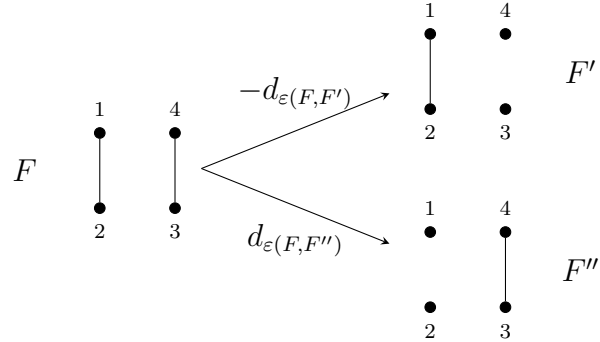
\begin{center}
    \begin{tikzpicture}[scale=.5]
      \begin{scope}[xshift=0]
        \node at (-4,1){$F$};
        \vertex[fill, label=above:\tiny{$1$}](v1) at (-2,2){};
        \vertex[fill, label=below:\tiny{$2$}](v2) at (-2,0){};
        \vertex[fill, label=below:\tiny{$3$}](v3) at (0,0){};
        \vertex[fill, label=above:\tiny{$4$}](v4) at (0,2){};
        \draw (v1)--(v2);
        \draw (v3)--(v4);
      \end{scope}
      \begin{scope}[xshift=20, yshift=30]
        \draw [->, >=stealth] (0,0)--(5,2);
        \draw [->, >=stealth] (0,0)--(5,-2);
        \node at (2.5,2){$-d_{\varepsilon(F,F')}$};
        \node at (2.5,-1.9){$d_{\varepsilon(F,F'')}$};
      \end{scope}
      \begin{scope}[xshift=250, yshift=75]
        \node at (2,1){$F'$};
        \vertex[fill, label=above:\tiny{$1$}](v1) at (-2,2){};
        \vertex[fill, label=below:\tiny{$2$}](v2) at (-2,0){};
        \vertex[fill, label=below:\tiny{$3$}](v3) at (0,0){};
        \vertex[fill, label=above:\tiny{$4$}](v4) at (0,2){};
        \draw (v1)--(v2);
      \end{scope}
      \begin{scope}[xshift=250, yshift=-75]
        \node at (2,1){$F''$};
        \vertex[fill, label=above:\tiny{$1$}](v1) at (-2,2){};
        \vertex[fill, label=below:\tiny{$2$}](v2) at (-2,0){};
        \vertex[fill, label=below:\tiny{$3$}](v3) at (0,0){};
        \vertex[fill, label=above:\tiny{$4$}](v4) at (0,2){};
        \draw (v3)--(v4);
      \end{scope}
    \end{tikzpicture}
  \end{center}
  The numberings associated to the subgraphs are
  \[
    \ytableausetup{nosmalltableaux}
    T(F) = \ytableaushort{12,34}\,,
    \qquad
    T(F') = \ytableaushort{12,3,4}\,,
    \qquad
    T(F'') = \ytableaushort{34,1,2}\,,
  \]
  and the permutation modules are
  \[
    \calM_F = \bbC[\fS_4] (e+(1\,2))(e+(3\,4)),
    \quad
    \calM_{F'} = \bbC[\fS_4] (e+(1\,2)),
    \quad
    \calM_{F''} = \bbC[\fS_4] (e+(3\,4)).
  \]
  The edge maps $d_{\varepsilon(F,F')}$ and $d_{\varepsilon(F,F'')}$ are simply (signed) inclusion maps, but it is not obvious how one should construct a generator for the Specht module of type $(3,1)$ inside $\calM_F$.
  More precisely, we will show that the generator cannot be constructed simply by considering a single $v_T^S$ for some numberings $T$ (and $S$) of shape $(3,1)$.

  The Specht submodules of type $\lambda = (3,1)$ in $\calM_{F'}$ and $\calM_{F''}$ are
  \[
    \ytableausetup{smalltableaux}
    \calM_{F'} \supseteq \bbC[\fS_4] \cdot v_{\scalebox{.75}{\ytableaushort{123,4}}}^{Y_1}\,
    \oplus
    \bbC[\fS_4] \cdot v_{\scalebox{.75}{\ytableaushort{124,3}}}^{Y_1},
    \qquad
    \calM_{F''} \supseteq \bbC[\fS_4] \cdot v_{\scalebox{.75}{\ytableaushort{134,2}}}^{Y_1}\,
    \oplus
    \bbC[\fS_4] \cdot v_{\scalebox{.75}{\ytableaushort{234,1}}}^{Y_1}.
  \]
  The four generators appearing in the previous equation are essentially the only generators for Specht modules of type $(3,1)$ constructed by considering group elements of the form $v_T^S$ for some numberings $T$ and $S$ of shape $(3,1)$, and none of them are elements in $\calM_F$.
  However, via the straightening laws, we see that $v_{\scalebox{.75}{\ytableaushort{123,4}}}^{Y_1} + v_{\scalebox{.75}{\ytableaushort{124,3}}}^{Y_1} + v_{\scalebox{.75}{\ytableaushort{134,2}}}^{Y_1} + v_{\scalebox{.75}{\ytableaushort{234,1}}}^{Y_1} = 0$, and
  \[
    (e - (1\,4))(e + (2\,3)+ (1\,2\,3))(e+(1\,2))(e+(3\,4))
    =
    v_{\scalebox{.75}{\ytableaushort{123,4}}}^{Y_1} + v_{\scalebox{.75}{\ytableaushort{124,3}}}^{Y_1}
    =
    -\left(v_{\scalebox{.75}{\ytableaushort{134,2}}}^{Y_1} + v_{\scalebox{.75}{\ytableaushort{234,1}}}^{Y_1}\right)
  \]
  is a generator for the Specht module indexed by $(3,1)$ inside $\calM_F$.
\end{remark}

\subsection{Straightening laws}
We describe straightening laws which are parallel to a dual version of those given in Fulton~\cite[Section 7.4]{Ful97} for oriented column tabloids.
In this section, all $\fS_n$-modules are {\em right} $\fS_n$-modules.

Let $\widetilde\calM^T$ denote the right $\fS_n$-submodule of $\bbC[\fS_n]$ generated by $a_T$.
For each numbering $S\in \Num(\lambda)$, $a_T \sigma_{S,T} = \sigma_{S,T}a_S$.  Furthermore, if $\rho\in R(T)$, then $a_T\rho = a_T$, so
\[
  \widetilde\calM^T = \mathrm{span}\{\sigma_{S,T}a_S \mid S\in \Tab(\lambda) \},
\]
where $\Tab(\lambda)$ is the set of row tabloids of shape $\lambda$.

Define an $\fS_n$-homomorphism $\Psi^T: \widetilde\calM^T \rightarrow \bbC[\fS_n]$ by $a_T \mapsto b_Ta_T$, so that
\[
  \Psi^T(\sigma_{S,T} a_S)
  =
  \Psi^T(a_T\sigma_{S,T} )
  =
  \Psi^T(a_T)\sigma_{S,T}
  =
  b_Ta_T\sigma_{S,T}
  =
  \sigma_{S,T}b_Sa_S
  =
  v_S^T.
\]

We will need the following lemma.
\begin{lemma}
  \label{lem.right_straighten_ba}
  For any numberings $S$ and $T$ of shape $\lambda$, $i=1,\ldots, \ell(\lambda)-1$ and $j=1,\ldots, \lambda_{i+1}$, let $Q$ be a nonempty subset of elements in the $(i+1)$-st row of $S$.
  Define
  \[
    \gamma_Q^T(S) = \sum_{U\in \Xi_{i,Q}(S)} \sigma_{U,T}a_U,
  \]
  where $\Xi_{i,Q}(S)$ is the set of all numberings $U$ obtained from $S$ by exchanging (possibly empty) subsets of $Q$ with subsets of the same size in the $i$-th row of $S$, preserving the order of each subset of elements.
  Then
  \[
    \gamma_Q^T(S) \in \ker\Psi^T.
  \]
\end{lemma}
\begin{proof}
  Let $P$ be the set of entries in the $i$-th row of $S$, and define subgroups $H \subseteq G \subseteq \fS_n$, where $G=\fS_{P\cup Q}$ is the subgroup of permutations that fixes elements outside $P\cup Q$, and $H=\fS_P\times\fS_Q$ is the subgroup of $G$ that permutes elements within each of $P$ and $Q$ respectively.

  It follows from the definition that $\gamma_Q^T(S) = \sum_{\nu}a_T \cdot \nu$, where the sum is over a transversal of the cosets $H\backslash G$.
  And since $H\subseteq R(T)$, then $a_T\cdot \eta = a_T$ for every $\eta \in H$, thus
  \[
    \sum_{\gamma \in G} a_T\cdot \gamma = |H|\, \gamma_Q^T(S).
  \]

  Under the map $\Psi^T$, the left hand side of the above equation is
  \[
    \sum_{\gamma \in G} b_Ta_T\gamma
    =
    \sum_{\gamma \in G} b_T \sum_{\rho\in R(T)}\rho \gamma
    =
    \sum_{\rho \in R(T)} \sum_{\gamma \in G} \rho b_{\rho^{-1}\cdot T} \gamma.
  \]
  For a fixed $\rho \in R(T)$, there exists two elements in $P\cup Q$ that are in the same column of $\rho^{-1}\cdot T$; concretely, if $j$ is the leftmost column of $\rho^{-1}\cdot T$ which contains an element of $Q$, then the $(i,j)$-th entry of $\rho^{-1}\cdot T$ is in $P$.
  Let $\tau_\rho\in G$ be the transposition which exchanges the $(i,j)$-th and $(i+1,j)$-th entries of $\rho^{-1}\cdot T$.
  Then we may factor
  \[
    \sum_{\gamma \in G} \rho b_{\rho^{-1}\cdot T} \gamma =\sum_{\xi} \rho b_{\rho^{-1}\cdot T} (e+\tau_\rho)\xi
  \]
  as a sum over a transversal of $\langle \tau_\rho\rangle\backslash G$. Moreover, $b_{\rho^{-1}\cdot T} (e+\tau_\rho)= 0$ since $\tau_\rho\in C(\rho^{-1}\cdot T)$.
  Therefore, $\sum_{\gamma \in G}a_T\cdot \gamma$, and hence $\gamma_Q^T(S)$, is in $\ker\Psi^T$.
\end{proof}

\begin{proposition}
  \label{prop.straighten}
  For any numberings $S$ and $T$ of shape $\lambda$, $i=1,\ldots, \ell(\lambda)-1$, and $j=1,\ldots, \lambda_{i+1}$, define
  \[
    \pi_{i,j}^T(S) = \sum_{U\in \Xi_{i,j}(S)} \sigma_{U,T}a_U,
  \]
  where $\Xi_{i,j}(S)$ is the set of all numberings $U$ obtained from $S$ by exchanging the first $j$ entries in the $(i+1)$-th row of $S$ with $j$ entries in the $i$-th row of $S$, preserving the order of each subset of elements.
  Then
  \[
    (-1)^j\pi_{i,j}^T(S)-\sigma_{S,T}a_S\in \ker \Psi^T.
  \]
\end{proposition}
\begin{proof}
  Let $J$ be the first $j$ elements in the $(i+1)$-st row of $S$.
  For $K\subseteq J$, let $f(K) = \sum_{U} \sigma_{U,T}a_U$ be the sum over numberings $U$ obtained from $S$ by exchanging the elements in $K$ with subsets of size $|K|$ in the $i$-th row of $S$, preserving the order of elements in each subset, and let $g(K)=\sum_{L \subseteq K} f(L)$.
  By the Principle of Inclusion-Exclusion,
  \[
    \pi_{i,j}^T(S) = f(J) = \sum_{K \subseteq J} (-1)^{|J\backslash K|} g(K)
    =
    (-1)^j \sum_{K \subseteq J} (-1)^k \gamma_K^T(S).
  \]
  Since $g(\emptyset) = \gamma_\emptyset^T(S) = \sigma_{S,T}a_S$, then
  \[
    (-1)^j\pi_{i,j}^T(S)-\sigma_{S,T}a_S =  \sum_{\emptyset \neq K \subseteq J} (-1)^k \gamma_K^T(S).
  \]
  The result now follows from Lemma~\ref{lem.right_straighten_ba}.
\end{proof}

Proposition~\ref{prop.straighten} provides an algorithm for expressing $v_S^T$ for any numbering $S$ of shape $\lambda$ as an integral linear combination of $v_U^T$ indexed by the standard Young tableaux of shape $\lambda$.
This is summarized in the following corollary.

\begin{corollary}
  \label{cor.straighten}
  For any numbering $S$ of shape $\lambda$,
  \[
    v_S^T = (-1)^j\sum_{U\in \Xi_{i,j}(S)} v_U^T,
  \]
  where $\Xi_{i,j}(S)$ is the set of all numberings $U$ obtained from $S$ by exchanging the first $j$ entries in the $(i+1)$-th row of $S$ with $j$ entries in the $i$-th row of $S$, preserving the order of each subset of elements.
\end{corollary}
\begin{proof}
  Under the map $\Psi^T$, the element
  \[
    (-1)^j\pi_{i,j}^T(S)-\sigma_{S,T}a_S \mapsto  \left((-1)^j\sum_{U\in \Xi_{i,j}(S)} v_U^T \right)-v_S^T = 0.
  \]
\end{proof}

We say that the numbering $S$ has a {\em decrease} in position $(i,j)$ if $S(i,j)>S(i+1,j)$.
If $S$ is not a standard Young tableau, then it has at least one decrease.
By applying the relation in Corollary~\ref{cor.straighten} iteratively to decreases from the bottom to the top, and from right to left, then $v_S^T$ can be expressed solely in terms of $v_U^T$ where each $U$ is a standard Young tableau.

\begin{example}
  By a small abuse of notation, we let $\pi_{i,j}$ denote the operator on numberings where $\pi_{i,j}(S) = (-1)^j \sum_{U\in \Xi_{i,j}(S)} U$ in this example.

  Let $T=\ytableaushort{123,456}$, and let $S=\ytableaushort{136,245}$.
  We first apply $\pi_{1,3}$, then followed by $\pi_{1,2}$.
  \[
    \pi_{1,2}\pi_{1,3}\ \ytableaushort{13{\,6_{\textcolor{red}\cdot}},245}
    =
    -\,\pi_{1,2}\ \ytableaushort{2{\,4_{\textcolor{red}\cdot}}5,136}
    =
    -\,\ytableaushort{135,246}-\,\ytableaushort{143,256}-\,\ytableaushort{213,456}
    \,.
  \]
  Then the straightening laws give
  \[
    \ytableausetup{smalltableaux}
    v_{\,\scalebox{.75}{\ytableaushort{136,245}}}^T
    =
    -v_{\,\scalebox{.75}{\ytableaushort{245,136}}}^T
    =
    -v_{\,\scalebox{.75}{\ytableaushort{135,246}}}^T
    -v_{\,\scalebox{.75}{\ytableaushort{134,256}}}^T
    -v_{\,\scalebox{.75}{\ytableaushort{123,456}}}^T.
  \]
\end{example}

\begin{example}
  We now finish Example~\ref{eg.210}.
  We see that when the map $d_{\varepsilon(F,F')}$ is restricted to Specht modules of type $\lambda$, it is simply the inclusion of $\bbC[\fS_6] \cdot v_{\,\scalebox{.75}{\ytableaushort{12,36,45}}}^{Y_1}$ into the first summand of the codomain.

  As for the map $d_{\varepsilon(F',F'')}$, it clearly sends $v_{\,\scalebox{.75}{\ytableaushort{14,25,36}}}^{Y_1} \mapsto v_{Y_5}^{Y_1}$.
  Lastly, by applying straightening laws, we see that $v_{\,\scalebox{.75}{\ytableaushort{12,36,45}}}^{Y_1} \mapsto - v_{Y_1}^{Y_1} - v_{Y_3}^{Y_1}$.
\end{example}

\section{Three pairs of graphs}
\label{sec.pairs}
\label{sec-three}
Amongst the set of connected graphs on $n=6$ vertices, there are nine pairs of non-isomorphic graphs whose chromatic symmetric functions are equal.
In this section, we exhibit three examples of pairs of non-isomorphic graphs, in which each pair of graphs has the same chromatic symmetric function, but has different chromatic symmetric homology (in $q$-degree zero).
Our computations follow the procedure and conventions that were outlined in Section~\ref{sec.restrict}.

\subsection{A summary of results}
We summarize the results of this section here.
The computations for each pair of graphs appearing in the next theorem are similar but lengthy, so we provide the details for performing these computations for the first pair of graphs, and defer the computations for the remaining pairs to Appendix~\ref{app.bigmatrices}.
There, the reader can also find our conjectures on the full $q$-degree zero homology of these graphs.
Please also see Appendix~\ref{app.K3} for a complete example of our methods.

\begin{theorem}
  \label{thm.stronger}
  Chromatic symmetric homology in $q$-degree zero distinguishes each of the pairs of graphs in Figures~\ref{fig.pairs_one}, \ref{fig.pairs_two} and \ref{fig.pairs_three}.
\end{theorem}

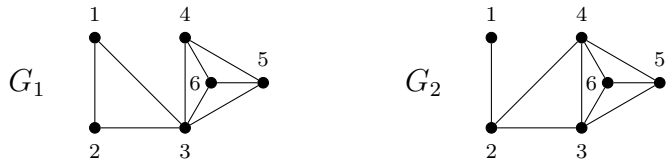
\begin{figure}[ht!]
  \centering
  \begin{tikzpicture}[scale=.6]
    \begin{scope}[xshift=0]
      \node at (-3.5,1){$G_1$};
      \vertex[fill, label=above:\tiny{$1$}](v1) at (-2,2){};
      \vertex[fill, label=below:\tiny{$2$}](v2) at (-2,0){};
      \vertex[fill, label=below:\tiny{$3$}](v3) at (0,0){};
      \vertex[fill, label=above:\tiny{$4$}](v4) at (0,2){};
      \vertex[fill, label=above:\tiny{$5$}](v5) at (1.73,1){};
      \vertex[fill, label={[xshift=0.2em]left:\tiny{$6$}}](v6) at (0.577,1){};
      \draw (v1)--(v2)--(v3)--(v4);
      \draw (v4)--(v5)--(v3);
      \draw (v6)--(v3); \draw (v6)--(v4); \draw (v6)--(v5);
      \draw (v1)--(v3);
    \end{scope}
    \begin{scope}[xshift=250]
      \node at (-3.5,1){$G_2$};
      \vertex[fill, label=above:\tiny{$1$}](v1) at (-2,2){};
      \vertex[fill, label=below:\tiny{$2$}](v2) at (-2,0){};
      \vertex[fill, label=below:\tiny{$3$}](v3) at (0,0){};
      \vertex[fill, label=above:\tiny{$4$}](v4) at (0,2){};
      \vertex[fill, label=above:\tiny{$5$}](v5) at (1.73,1){};
      \vertex[fill, label={[xshift=0.2em]left:\tiny{$6$}}](v6) at (0.577,1){};
      \draw (v1)--(v2)--(v3)--(v4);
      \draw (v4)--(v5)--(v3);
      \draw (v6)--(v3); \draw (v6)--(v4); \draw (v6)--(v5);
      \draw (v2)--(v4);
    \end{scope}
  \end{tikzpicture}
  \caption{$G_1$ and $G_2$ have the same chromatic symmetric function \\$X_{G_1} = X_{G_2} = 144s_{(1^6)}+72s_{(2,1^4)}+24s_{(2^2,1^2)}$.}
  \label{fig.pairs_one}
\end{figure}

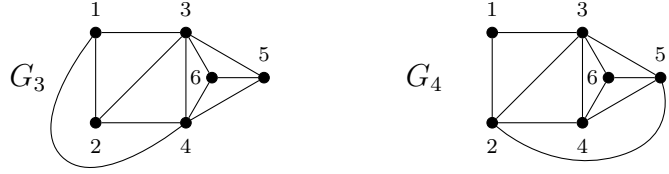
\begin{figure}[ht!]
  \centering
  \begin{tikzpicture}[scale=.6]
    \begin{scope}[xshift=0]
      \node at (-1.5,1){$G_3$};
      \vertex[fill, label=below:\tiny{$2$}](v2) at (0,0){};
      \vertex[fill, label=above:\tiny{$1$}](v1) at (0,2){};
      \vertex[fill, label=above:\tiny{$3$}](v3) at (2,2){};
      \vertex[fill, label={[xshift=0.2em]left:\tiny{$6$}}](v6) at (2.577,1){};
      \vertex[fill, label=above:\tiny{$5$}](v5) at (3.73,1){};
      \vertex[fill, label=below:\tiny{$4$}](v4) at (2,0){};
      \draw (v2)--(v1)--(v3)--(v6)--(v5)--(v4)--(v2);
      \draw (v3)--(v5);
      \draw (v3)--(v4);
      \draw (v2)--(v3);
      \draw (v4)--(v6);
     \draw [black] plot [smooth, tension=1.5] coordinates { (0,2) (-.75,-.75) (2,0)};
    \end{scope}
    \begin{scope}[xshift=250]
      \node at (-1.5,1){$G_4$};
      \vertex[fill, label=below:\tiny{$2$}](v2) at (0,0){};
      \vertex[fill, label=above:\tiny{$1$}](v1) at (0,2){};
      \vertex[fill, label=above:\tiny{$3$}](v3) at (2,2){};
      \vertex[fill, label={[xshift=0.2em]left:\tiny{$6$}}](v6) at (2.577,1){};
      \vertex[fill, label=above:\tiny{$5$}](v5) at (3.73,1){};
      \vertex[fill, label=below:\tiny{$4$}](v4) at (2,0){};
      \draw (v2)--(v1)--(v3)--(v6)--(v5)--(v4)--(v2);
      \draw (v3)--(v5);
      \draw (v3)--(v4);
      \draw (v2)--(v3);
      \draw (v4)--(v6);
      \draw [black] plot [smooth, tension=1.5] coordinates { (0,0) (2.75,-.75) (3.73,1)};
    \end{scope}
  \end{tikzpicture}
  \caption{$G_3$ and $G_4$ have the same chromatic symmetric function \\$X_{G_3} = X_{G_4} = 288s_{(1^6)}+72s_{(2,1^4)}+8s_{(2^2,1^2)}$.}
  \label{fig.pairs_two}
\end{figure}

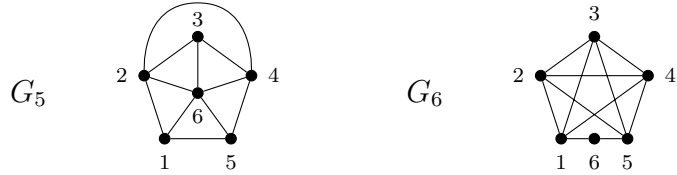
\begin{figure}[ht!]
  \centering
  \begin{tikzpicture}[scale=.75]
    \begin{scope}[xshift=-20]
      \node at (-3,0){$G_5$};
      \vertex[fill, label=below:\tiny{$1$}](v1) at (-0.588,-0.809){};
      \vertex[fill, label=left:\tiny{$2$}](v2) at (-0.951,0.309){};
      \vertex[fill, label={[yshift=-0.2em]above:\tiny{$3$}}](v3) at (0,1){};
      \vertex[fill, label=right:\tiny{$4$}](v4) at (0.951,0.309){};
      \vertex[fill, label=below:\tiny{$5$}](v5) at (0.588,-0.809){};
      \vertex[fill, label=below:\tiny{$6$}](v6) at (0,0){};
      \draw (v1)--(v2)--(v3)--(v4)--(v5)--(v1);
      \draw (v6)--(v1); \draw (v6)--(v2); \draw (v6)--(v3);
      \draw (v6)--(v4); \draw (v6)--(v5);
      \draw [black] plot [smooth, tension=2] coordinates { (-0.951,0.309) (0, 1.6)(0.951,0.309) };
    \end{scope}
    \begin{scope}[xshift=180]
     \node at (-3,0){$G_6$};
      \vertex[fill, label=below:\tiny{$1$}](v1) at (-0.588,-0.809){};
      \vertex[fill, label=left:\tiny{$2$}](v2) at (-0.951,0.309){};
      \vertex[fill, label=above:\tiny{$3$}](v3) at (0,1){};
      \vertex[fill, label=right:\tiny{$4$}](v4) at (0.951,0.309){};
      \vertex[fill, label=below:\tiny{$5$}](v5) at (0.588,-0.809){};
      \vertex[fill, label=below:\tiny{$6$}](v6) at (0,-0.809){};
      \draw (v1)--(v2)--(v3)--(v4)--(v5)--(v6)--(v1);
      \draw (v3)--(v5)--(v2)--(v4)--(v1)--(v3);
    \end{scope}
  \end{tikzpicture}
  \caption{$G_5$ and $G_6$ have the same chromatic symmetric function \\$X_{G_5} = X_{G_6} = 312s_{(1^6)}+60s_{(2,1^4)}+12s_{(2^2,1^2)}$.}
  \label{fig.pairs_three}
\end{figure}

\subsection{The pair \texorpdfstring{$G_1$}{G1} and \texorpdfstring{$G_2$}{G2}}
Let $G_1$ and $G_2$ denote the graphs shown in Figure~\ref{fig.pairs_one}.
This pair of graphs is easily seen to be non-isomorphic as $G_2$ has a degree one vertex whereas $G_1$ does not.
Nonetheless, a result of Orellana and Scott~\cite[Theorem 4.2]{OS14} states that this pair of graphs has the same chromatic symmetric function, which in this case is $X_{G_1} = X_{G_2} = 144s_{(1^6)}+72s_{(2,1^4)}+24s_{(2^2,1^2)}$.

We will show that the graphs $G_1$ and $G_2$ have distinct chromatic symmetric homology by showing that the multiplicities of the $\fS_6$-module $\S{2^3}$ in $H_{1}(G_1;\bbC)$ and $H_{1}(G_2;\bbC)$ are different.

\begin{lemma}
  \label{lemma.G1}
  The multiplicity of the $\fS_6$-module $\S{2^3}$ in $H_{1}(G_1;\bbC)$ is $2$.
\end{lemma}
\begin{lemma}
  \label{lemma.G2}
  The multiplicity of the $\fS_6$-module $\S{2^3}$ in $H_{1}(G_2;\bbC)$ is $1$.
\end{lemma}

\begin{proof}{[Lemma~\ref{lemma.G1}]}
  We will compute
  \[
    \xymatrix{C_2(G_1)\big|_{\S{2^3}} \ar[r]^{d_{2}} & C_1(G_1)\big|_{\S{2^3}} \ar[r]^{d_{1}} & C_0(G_1)\big|_{\S{2^3}}},
  \]
  restricted to the $\S{2^3}$ modules.
  We order the edges of $G_1$ in lexicographic order; that is, $$(1,2), (1,3), (2,3), (3,4), (3,5), (3,6), (4,5), (4,6), (5,6),$$ and label these as $e_1,\ldots, e_9$.
  If a spanning subgraph of $G_1$ has only one edge $e_i$, we denote that spanning subgraph by $F_{e_i}$, and if a spanning subgraph of $G_1$ has two edges $e_i$ and $e_j$, we denote that spanning subgraph by $F_{e_i,e_j}$.

  In homological degree zero, let $Y_1,\ldots, Y_5$ denote the standard Young tableaux of shape $(2^3)$ listed with respect to the $\preccurlyeq$ order of Definition~\ref{def.totalorder}.
  Then
  \[
    C_0(G_1)\big|_{\S{2^3}} = \bigoplus_{i=1}^5 \left(\bbC[\fS_6]\cdot v_{Y_i}^{Y_1}\right) \cong \S{2^3}^{\oplus5}.
  \]

  In homological degree one, there are $m=9$ spanning subgraphs $F_{e_i}$.
  The multiplicity of $\S{2^3}$ in $\calM_{F_{e_i}} \cong \M{2,1^4}$ is $2$.
  Let $X_{i}^{1}$ and $X_{i}^{2}$ denote the numberings which index the two copies of $\S{2^3}$ in each $\calM_{F_{e_i}}$, again listed with respect to the $\preccurlyeq$ ordering.
  So
  \[
    C_1(G_1)\big|_{\S{2^3}}
    =
    \bigoplus_{i=1}^9 \left(\bbC[\fS_6] \cdot v_{X_i^1}^{Y_1} \oplus \bbC[\fS_6] \cdot v_{X_i^2}^{Y_1} \right)
    \cong
    \S{2^3}^{\oplus 18}.
  \]

  Lastly in homological degree two, there are $15$ spanning subgraphs of $G_1$ whose connected components have partition type $(2^2,1^2)$.
  The multiplicity of $\S{2^3}$ in each $\calM_{F_{e_i,e_j}} \cong \M{2^2,1^2}$ is $1$, and we let $W_{ij}$ denote the numbering which indexes the copy of $\S{2^3}$ in $\calM_{F_{e_i,e_j}}$.
  So
  \[
    C_2(G_1)\big|_{\S{2^3}}
    =
    \bigoplus_{}^{} \left(\bbC[\fS_6] \cdot v_{W_{ij}}^{Y_1} \right)
    \cong
    \S{2^3}^{\oplus 15}.
  \]

  It suffices to describe the differentials $d_2(G_1)$ and $d_1(G_1)$ by their action on the chosen generators of the modules.
  With respect to the choices of generators made above, $d_2(G_1)$ and $d_1(G_1)$ are described by the following matrices:
  \[
    \setlength{\arraycolsep}{2.0pt}
    \renewcommand{\arraystretch}{1.0}
    \tiny
    d_2(G_1) =
    \kbordermatrix{
            & W_{14} & W_{15} & W_{16} & W_{17} & W_{18} & W_{19} & W_{27} & W_{28} & W_{29} & W_{37} & W_{38} & W_{39} & W_{49} & W_{58} & W_{67} \\
      X_1^1 &     \n &      0 &      1 &      1 &      0 &     \n &      0 &      0 &      0 &      0 &      0 &      0 &      0 &      0 &      0 \\
      X_1^2 &      0 &     \n &      1 &      1 &     \n &      0 &      0 &      0 &      0 &      0 &      0 &      0 &      0 &      0 &      0 \\
      X_2^1 &      0 &      0 &      0 &      0 &      0 &      0 &      1 &      0 &     \n &      0 &      0 &      0 &      0 &      0 &      0 \\
      X_2^2 &      0 &      0 &      0 &      0 &      0 &      0 &      1 &     \n &      0 &      0 &      0 &      0 &      0 &      0 &      0 \\
      X_3^1 &      0 &      0 &      0 &      0 &      0 &      0 &      0 &      0 &      0 &      1 &      0 &     \n &      0 &      0 &      0 \\
      X_3^2 &      0 &      0 &      0 &      0 &      0 &      0 &      0 &      0 &      0 &      1 &     \n &      0 &      0 &      0 &      0 \\
      X_4^1 &      1 &      0 &      0 &      0 &      0 &      0 &      0 &      0 &      0 &      0 &      0 &      0 &     \n &      0 &      0 \\
      X_4^2 &      0 &      0 &      0 &      0 &      0 &      0 &      0 &      0 &      0 &      0 &      0 &      0 &      0 &      0 &      0 \\
      X_5^1 &      0 &      1 &      0 &      0 &      0 &      0 &      0 &      0 &      0 &      0 &      0 &      0 &      0 &     \n &      0 \\
      X_5^2 &      0 &      0 &      0 &      0 &      0 &      0 &      0 &      0 &      0 &      0 &      0 &      0 &      0 &      0 &      0 \\
      X_6^1 &      0 &      0 &      1 &      0 &      0 &      0 &      0 &      0 &      0 &      0 &      0 &      0 &      0 &      0 &     \n \\
      X_6^2 &      0 &      0 &      0 &      0 &      0 &      0 &      0 &      0 &      0 &      0 &      0 &      0 &      0 &      0 &      0 \\
      X_7^1 &      0 &      0 &      0 &      1 &      0 &      0 &      0 &      0 &      0 &     \n &      0 &      0 &      0 &      0 &      1 \\
      X_7^2 &      0 &      0 &      0 &      0 &      0 &      0 &      1 &      0 &      0 &     \n &      0 &      0 &      0 &      0 &      0 \\
      X_8^1 &      0 &      0 &      0 &      0 &      1 &      0 &      0 &      0 &      0 &      0 &     \n &      0 &      0 &      1 &      0 \\
      X_8^2 &      0 &      0 &      0 &      0 &      0 &      0 &      0 &      1 &      0 &      0 &     \n &      0 &      0 &      0 &      0 \\
      X_9^1 &      0 &      0 &      0 &      0 &      0 &      1 &      0 &      0 &      0 &      0 &      0 &     \n &      1 &      0 &      0 \\
      X_9^2 &      0 &      0 &      0 &      0 &      0 &      0 &      0 &      0 &      1 &      0 &      0 &     \n &      0 &      0 &      0
    }
  \]
  \[
    \setlength{\arraycolsep}{2.0pt}
    \renewcommand{\arraystretch}{1.0}
    \tiny
    d_1(G_1) =
    \kbordermatrix{
          & X_1^1 & X_1^2 & X_2^1 & X_2^2 & X_3^1 & X_3^2 & X_4^1 & X_4^2 & X_5^1 & X_5^2 & X_6^1 & X_6^2 & X_7^1 & X_7^2 & X_8^1 & X_8^2 & X_9^1 & X_9^2 \\
      Y_1 &     1 &     0 &     0 &     0 &    \n &     0 &     1 &    \n &     0 &     1 &    \n &     0 &    \n &     0 &     0 &     0 &     1 &     0 \\
      Y_2 &     0 &     0 &     1 &     0 &    \n &     0 &     0 &     0 &     0 &     1 &     0 &     0 &     0 &    \n &     0 &     0 &     0 &     1 \\
      Y_3 &     0 &     1 &     0 &     0 &     0 &    \n &     0 &     0 &     1 &     0 &    \n &     0 &    \n &     0 &     1 &     0 &     0 &     0 \\
      Y_4 &     0 &     0 &     0 &     1 &     0 &    \n &     0 &     1 &     0 &     0 &     0 &     0 &     0 &    \n &     0 &     1 &     0 &     0 \\
      Y_5 &     0 &     0 &     0 &     0 &     0 &     0 &     0 &     1 &     0 &    \n &     0 &     1 &     0 &     0 &     0 &     0 &     0 &     0 \\
    }
  \]

  It is easily verified that $d_1\circ d_2 =0$, $\dim \ker d_1 = 13$, and $\mathrm{rank}\, d_2 = 11$, so the multiplicity of the Specht module $\S{2^3}$ in $H_{1}(G_1;\bbC)$ is two.
  We note that generators of the homology $H_{1}(G_1;\bbC)\big|_{\S{2^3}}$ are $X_4^2-X_6^2-X_8^2+X_9^1$ and $X_5^2+X_6^2-X_9^1-X_9^2$.
\end{proof}

\begin{proof}{[Lemma~\ref{lemma.G2}]}
  We follow the exact same procedure as we applied for the computation of $H_{1}(G_1;\bbC)\big|_{\S{2^3}}$.
  We similarly have $C_0(G_2)\big|_{\S{2^3}} \cong \S{2^3}^{\oplus 5}$, $C_1(G_2)\big|_{\S{2^3}}\cong \S{2^3}^{\oplus 18}$, and $C_2(G_2)\big|_{\S{2^3}} \cong \S{2^3}^{\oplus 15}$.
  The differentials $d_2(G_2)$ and $d_1(G_2)$ are described by the following matrices:
  \[
    \setlength{\arraycolsep}{2pt}
    \renewcommand{\arraystretch}{1.0}
    \tiny
    d_2(G_2) =
    \kbordermatrix{
            & W_{14} & W_{15} & W_{16} & W_{17} & W_{18} & W_{19} & W_{27} & W_{28} & W_{29} & W_{37} & W_{38} & W_{39} & W_{49} & W_{58} & W_{67} \\
      X_1^1 &     \n &      0 &      1 &      1 &      0 &     \n &      0 &      0 &      0 &      0 &      0 &      0 &      0 &      0 &      0 \\
      X_1^2 &      0 &     \n &      1 &      1 &     \n &      0 &      0 &      0 &      0 &      0 &      0 &      0 &      0 &      0 &      0 \\
      X_2^1 &      0 &      0 &      0 &      0 &      0 &      0 &      1 &      0 &     \n &      0 &      0 &      0 &      0 &      0 &      0 \\
      X_2^2 &      0 &      0 &      0 &      0 &      0 &      0 &      1 &     \n &      0 &      0 &      0 &      0 &      0 &      0 &      0 \\
      X_3^1 &      0 &      0 &      0 &      0 &      0 &      0 &      0 &      0 &      0 &      1 &      0 &     \n &      0 &      0 &      0 \\
      X_3^2 &      0 &      0 &      0 &      0 &      0 &      0 &      0 &      0 &      0 &      1 &     \n &      0 &      0 &      0 &      0 \\
      X_4^1 &      1 &      0 &      0 &      0 &      0 &      0 &      0 &      0 &      0 &      0 &      0 &      0 &     \n &      0 &      0 \\
      X_4^2 &      0 &      0 &      0 &      0 &      0 &      0 &      0 &      0 &      0 &      0 &      0 &      0 &      0 &      0 &      0 \\
      X_5^1 &      0 &      1 &      0 &      0 &      0 &      0 &      0 &      0 &      0 &     \n &      0 &      0 &      0 &     \n &      0 \\
      X_5^2 &      0 &      0 &      0 &      0 &      0 &      0 &      0 &      0 &      0 &     \n &      0 &      0 &      0 &      0 &      0 \\
      X_6^1 &      0 &      0 &      1 &      0 &      0 &      0 &      0 &      0 &      0 &      0 &     \n &      0 &      0 &      0 &     \n \\
      X_6^2 &      0 &      0 &      0 &      0 &      0 &      0 &      0 &      0 &      0 &      0 &     \n &      0 &      0 &      0 &      0 \\
      X_7^1 &      0 &      0 &      0 &      1 &      0 &      0 &     \n &      0 &      0 &      0 &      0 &      0 &      0 &      0 &      1 \\
      X_7^2 &      0 &      0 &      0 &      0 &      0 &      0 &     \n &      0 &      0 &      0 &      0 &      0 &      0 &      0 &      0 \\
      X_8^1 &      0 &      0 &      0 &      0 &      1 &      0 &      0 &     \n &      0 &      0 &      0 &      0 &      0 &      1 &      0 \\
      X_8^2 &      0 &      0 &      0 &      0 &      0 &      0 &      0 &     \n &      0 &      0 &      0 &      0 &      0 &      0 &      0 \\
      X_9^1 &      0 &      0 &      0 &      0 &      0 &      1 &      0 &      0 &     \n &      0 &      0 &      0 &      1 &      0 &      0 \\
      X_9^2 &      0 &      0 &      0 &      0 &      0 &      0 &      0 &      0 &     \n &      0 &      0 &      1 &      0 &      0 &      0
    }
  \]
  \[
    \setlength{\arraycolsep}{2pt}
    \renewcommand{\arraystretch}{1.0}
    \tiny
    d_1(G_2) =
    \kbordermatrix{
          & X_1^1 & X_1^2 & X_2^1 & X_2^2 & X_3^1 & X_3^2 & X_4^1 & X_4^2 & X_5^1 & X_5^2 & X_6^1 & X_6^2 & X_7^1 & X_7^2 & X_8^1 & X_8^2 & X_9^1 & X_9^2 \\
      Y_1 &     1 &     0 &    \n &     0 &     0 &     1 &     1 &    \n &     0 &     1 &    \n &     0 &    \n &     0 &     0 &     0 &     1 &     0 \\
      Y_2 &     0 &     0 &    \n &     0 &     1 &     0 &     0 &     0 &     0 &     1 &     0 &     0 &     0 &    \n &     0 &     0 &     0 &     1 \\
      Y_3 &     0 &     1 &     0 &    \n &     0 &     1 &     0 &     0 &     1 &     0 &    \n &     0 &    \n &     0 &     1 &     0 &     0 &     0 \\
      Y_4 &     0 &     0 &     0 &    \n &     0 &     0 &     0 &     1 &     0 &     0 &     0 &     0 &     0 &    \n &     0 &     1 &     0 &     0 \\
      Y_5 &     0 &     0 &     0 &     0 &     0 &    \n &     0 &     1 &     0 &    \n &     0 &     1 &     0 &     0 &     0 &     0 &     0 &     0 \\
    }
  \]

  It is easily verified that $d_1\circ d_2 =0$, $\dim \ker d_1(G_2) = 13$, and $\mathrm{rank}\, d_2(G_2) = 12$, so the multiplicity of the Specht module $\S{2^3}$ in $H_{1}(G_2;\bbC)$ is one.
  We note that a generator of the homology $H_{1}(G_2;\bbC)\big|_{\S{2^3}}$ is $X_4^2-X_6^2-X_8^2+X_9^1$.
\end{proof}

\begin{remark}
  Conjecture~\ref{conj.one} suggests that we could have chosen to show that the multiplicities of the $\fS_6$-module $\S{2^2,1^2}$ in $H_{1}(G_1;\bbC)$ and $H_{1}(G_2;\bbC)$ are different, and in fact, the procedure outlined in Section~\ref{sec.restrict} allows us to perform this computation.
  However, we chose to do the computations for $\S{2^3}$ because they are are simpler, due to the fact that the multiplicity of $\S{2^3}$ in $\M{2,1^4}$ is two, whereas the multiplicity of $\S{2^2,1^2}$ in $\M{2,1^4}$ is three.
\end{remark}

\subsection{The pair \texorpdfstring{$G_3$}{G3} and \texorpdfstring{$G_4$}{G4}}
Let $G_3$ and $G_4$ denote the graphs shown in Figure~\ref{fig.pairs_two}.
This pair of graphs is easily seen to be non-isomorphic as $G_4$ has a degree two vertex whereas $G_3$ does not.
Nonetheless, they have the same chromatic symmetric function, which in this case is $X_{G_3} = X_{G_4} = 288s_{(1^6)}+72s_{(2,1^4)}+8s_{(2^21^2)}$.
In contrast with the first example, the result of Orellana and Scott~\cite[Theorem 4.2]{OS14} does not apply to this pair of graphs.

\begin{lemma}
  \label{lemma.G3}
  The multiplicity of the $\fS_6$-module $\S{2^3}$ in $H_{1}(G_3;\bbC)$ is $1$.
\end{lemma}

\begin{lemma}
  \label{lemma.G4}
  The multiplicity of the $\fS_6$-module $\S{2^3}$ in $H_{1}(G_4;\bbC)$ is $0$.
\end{lemma}

Please see Appendix~\ref{app.bigmatrices} for these computations.

\subsection{The pair \texorpdfstring{$G_5$}{G5} and \texorpdfstring{$G_6$}{G6}}
Let $G_5$ and $G_6$ denote the graphs shown in Figure~\ref{fig.pairs_three}.
This pair of graphs is easily seen to be non-isomorphic since $G_5$ is planar whereas $G_6$ is not, as it is an edge-subdivision of the nonplanar complete graph $K_5$.
Nonetheless, they have the same chromatic symmetric function $X_{G_5} = X_{G_6} = 312s_{(1^6)}+60s_{(2,1^4)}+12s_{(2^2,1^2)}$.


\begin{lemma}
  \label{lemma.G5}
  The multiplicity of the $\fS_6$-module $\S{2^2,1^2}$ in $H_{1}(G_5,\bbC)$ is $1$.
\end{lemma}

\begin{lemma}
  \label{lemma.G6}
  The multiplicity of the $\fS_6$-module $\S{2^2,1^2}$ in $H_{1}(G_6,\bbC)$ is $0$.
\end{lemma}

Please see Appendix~\ref{app.bigmatrices} for these computations.

\section{Torsion in chromatic symmetric homology}
\label{sec-four}

In this section, we move from considering linear representations of $\fS_n$ over the field $\bbC$, to working over the ring $\bbZ$.

\subsection{The graphs \texorpdfstring{$K_5$}{K5} and \texorpdfstring{$K_{3,3}$}{K33}}
\label{sectorsion}
We summarize the results of this section here.
Working over the integers, the $q$-degree zero homology $H_1(G;\bbZ)$ exhibits torsion in the case when $G$ is the complete graph $K_5$ or the complete bipartite graph $K_{3,3}$.

We provide the details for the computation for the graph $K_5$ here, and defer the similar computations for the graph $K_{3,3}$ to Appendix~\ref{app.bigmatrices}.
\begin{theorem}
  \label{thm.K5}
  The chromatic symmetric homology $H_1(K_5;\bbZ)$ contains $\bbZ_2$-torsion.
\end{theorem}

\begin{theorem}
  \label{thm.K33}
  The chromatic symmetric homology $H_1(K_{3,3};\bbZ)$ contains $\bbZ_2$-torsion.
\end{theorem}

\begin{remark}
  Over the finite field $GF(p)$ of prime characteristic $p$, the $p$-modular irreducible representations $\calD_\lambda= \calS_\lambda/ ( \calS_\lambda \cap \calS_\lambda^\perp)$ are indexed by the $p$-regular partitions $\lambda$ of $n$, which are partitions whose parts appear with multiplicity less than $p$.

  In the case $p=2$, the set of $2$-regular partitions is the set of partitions with distinct parts.
  The symmetric group $\fS_5$ has three $2$-modular irreducible representations $\calD_{(5)}$, $\calD_{(3,2)}$ and $\calD_{(4,1)}$, of dimensions $1,4$ and $4$, and the symmetric group $\fS_6$ has four $2$-modular irreducible representations $\calD_{(6)}$, $\calD_{(5,1)}$, $\calD_{(4,2)}$ and $\calD_{(3,2,1)}$, of dimensions $1,4,4$ and $16$.

  For more background on modular representations of the symmetric group, see James and Kerber~\cite[Chapter 7]{JK}.
\end{remark}

We show that $\bbZ_2$-torsion in $H_1(K_5;\bbZ)$ is generated by a linear combination of polytabloids $v_T^Y$ of shape $(2^2,1)$, while $\bbZ_2$-torsion in $H_1(K_{3,3};\bbZ)$ is generated by a linear combination of polytabloids of shape $(2^2,1^2)$.
Computational data that we obtained (Appendix~\ref{computations}, graphs numbered $31$ and $131$) suggests that as $\bbZ$-modules, $H_1(K_5;\bbZ) = \bbZ^{24} \oplus \bbZ_2^5$, while $H_1(K_{3,3};\bbZ) = \bbZ^{25} \oplus \bbZ_2^4$.
 Also see Conjectures~\ref{conj.k5} and~\ref{conj.k33}.

\begin{proof}{[Theorem~\ref{thm.K5}]}
  We have the chain modules $C_0(K_5)\big|_{\S{2^2,1}} \cong \S{2^2,1}^{\oplus 5}$, $C_1(K_5)\big|_{\S{2^2,1}} \cong \S{2^2,1}^{\oplus 20}$, and $C_2(K_5)\big|_{\S{2^2,1}} \cong \S{2^2,1}^{\oplus 15}$.
  The differentials $d_2(K_5)$ and $d_1(K_5)$ are described by the following matrices:
  \[
    \setlength{\arraycolsep}{1pt}
    \renewcommand{\arraystretch}{1.0}
    \tiny
    d_2(K_5) =
    \kbordermatrix{
              & W_{18} & W_{19} & W_{1a} & W_{26} & W_{27} & W_{2a} & W_{35} & W_{37} & W_{39} & W_{45} & W_{46} & W_{48} & W_{5a} & W_{69} & W_{78} \\
      X_1^1   &     \n &      0 &      1 &      0 &      0 &      0 &      0 &      0 &      0 &      0 &      0 &      0 &      0 &      0 &      0 \\
      X_1^2   &      0 &     \n &      1 &      0 &      0 &      0 &      0 &      0 &      0 &      0 &      0 &      0 &      0 &      0 &      0 \\
      X_2^1   &      0 &      0 &      0 &     \n &      0 &      1 &      0 &      0 &      0 &      0 &      0 &      0 &      0 &      0 &      0 \\
      X_2^2   &      0 &      0 &      0 &      0 &     \n &      1 &      0 &      0 &      0 &      0 &      0 &      0 &      0 &      0 &      0 \\
      X_3^1   &      0 &      0 &      0 &      0 &      0 &      0 &     \n &      0 &      1 &      0 &      0 &      0 &      0 &      0 &      0 \\
      X_3^2   &      0 &      0 &      0 &      0 &      0 &      0 &      0 &     \n &      1 &      0 &      0 &      0 &      0 &      0 &      0 \\
      X_4^1   &      0 &      0 &      0 &      0 &      0 &      0 &      0 &      0 &      0 &     \n &      0 &      1 &      0 &      0 &      0 \\
      X_4^2   &      0 &      0 &      0 &      0 &      0 &      0 &      0 &      0 &      0 &      0 &     \n &      1 &      0 &      0 &      0 \\
      X_5^1   &      0 &      0 &      0 &      0 &      0 &      0 &      1 &      0 &      0 &      0 &      0 &      0 &      1 &      0 &      0 \\
      X_5^2   &      0 &      0 &      0 &      0 &      0 &      0 &      0 &      0 &      0 &      1 &      0 &      0 &      1 &      0 &      0 \\
      X_6^1   &      0 &      0 &      0 &      1 &      0 &      0 &      0 &      0 &      0 &      0 &      0 &      0 &      0 &      1 &      0 \\
      X_6^2   &      0 &      0 &      0 &      0 &      0 &      0 &      0 &      0 &      0 &      0 &      1 &      0 &      0 &      1 &      0 \\
      X_7^1   &      0 &      0 &      0 &      0 &      1 &      0 &      0 &      0 &      0 &      0 &      0 &      0 &      0 &      0 &      1 \\
      X_7^2   &      0 &      0 &      0 &      0 &      0 &      0 &      0 &      1 &      0 &      0 &      0 &      0 &      0 &      0 &      1 \\
      X_8^1   &      1 &      0 &      0 &      0 &      0 &      0 &      0 &      0 &      0 &      0 &      0 &      0 &      0 &      0 &     \n \\
      X_8^2   &      0 &      0 &      0 &      0 &      0 &      0 &      0 &      0 &      0 &      0 &      0 &      1 &      0 &      0 &     \n \\
      X_9^1   &      0 &      1 &      0 &      0 &      0 &      0 &      0 &      0 &      0 &      0 &      0 &      0 &      0 &     \n &      0 \\
      X_9^2   &      0 &      0 &      0 &      0 &      0 &      0 &      0 &      0 &      1 &      0 &      0 &      0 &      0 &     \n &      0 \\
      X_{10}^1&      0 &      0 &      1 &      0 &      0 &      0 &      0 &      0 &      0 &      0 &      0 &      0 &     \n &      0 &      0 \\
      X_{10}^2&      0 &      0 &      0 &      0 &      0 &      1 &      0 &      0 &      0 &      0 &      0 &      0 &     \n &      0 &      0 \\
    }
  \]

  \[
    \setlength{\arraycolsep}{2pt}
    \renewcommand{\arraystretch}{1.0}
    \tiny
    d_1(K_5) =
    \kbordermatrix{
          & X_1^1 & X_1^2 & X_2^1 & X_2^2 & X_3^1 & X_3^2 & X_4^1 & X_4^2 & X_5^1 & X_5^2 & X_6^1 & X_6^2 & X_7^1 & X_7^2 & X_8^1 & X_8^2 & X_9^1 & X_9^2 & X_{10}^1 & X_{10}^2 \\
      Y_1 &     1 &     0 &     0 &     0 &    \n &     0 &     0 &     1 &    \n &     0 &     0 &     1 &     0 &     0 &     1 &    \n &     0 &     1 &       \n &        0 \\
      Y_2 &     0 &     1 &     0 &     0 &     0 &     0 &    \n &     1 &     0 &    \n &     0 &     1 &     0 &     0 &     0 &     0 &     1 &     0 &       \n &        0 \\
      Y_3 &     0 &     0 &     1 &     0 &    \n &     0 &     0 &     0 &    \n &     0 &     1 &     0 &     0 &     0 &     0 &     0 &     0 &     1 &        0 &       \n \\
      Y_4 &     0 &     0 &     0 &     1 &     0 &     0 &    \n &     0 &     0 &    \n &     0 &     0 &     1 &     0 &     0 &     1 &     0 &     0 &        0 &       \n \\
      Y_5 &     0 &     0 &     0 &     0 &     0 &     1 &     0 &    \n &     0 &     0 &     0 &    \n &     0 &     1 &     0 &     1 &     0 &    \n &        0 &        0
    }
  \]

  It is easily verified that $d_1\circ d_2 =0$, $\dim \ker d_1(K_5) = 15$, and $\mathrm{rank}\,  d_2(K_5) = 15$.

  Let  $g = W_{18}+W_{19}+W_{1,10}+W_{26}-W_{27}-W_{2,10}+W_{35}+W_{37}+W_{39}+W_{45}+W_{46}+W_{48}-W_{5,10}-W_{69}+W_{78}\in C_2(K_5)$ and $h = X_2^1-X_7^2-X_9^1-X_9^2 + X_{10}^1 \in C_1(K_5)$.
  We note that $h\notin \im d_2$, $d_2(g) = 2h$ and $d_1(h) = 0$, so $h$ generates $\bbZ_2$-torsion in $H_{1}(K_5;\bbZ)$.
\end{proof}

\begin{conjecture}
  \label{conj.k5}
  The $q$-degree zero homology of $K_5$ over $\bbC$ is
  \[
    H_0(K_5;\bbC) = \S{1^5},
    \quad
    H_1(K_5;\bbC) = \S{21^3}^{\oplus6},
    \quad
    H_2(K_5;\bbC) = \S{31^2}^{\oplus 11} \oplus \S{32}^{\oplus 5},
    \quad
    H_3(K_5;\bbC) = \S{41}^{\oplus 6}.
  \]
\end{conjecture}

\begin{conjecture}
  \label{conj.k33}
  The $q$-degree zero homology of $K_{3,3}$ over $\bbC$ is
  \[
    H_0(K_{3,3};\bbC) = \S{1^6},
    \quad
    H_1(K_{3,3};\bbC) = \S{21^4}^{\oplus4}\oplus \S{2^3},
    \quad
    H_2(K_{3,3};\bbC) = \S{31^3}^{\oplus10} \oplus \S{321}^{\oplus4},
  \]
  \[
    H_3(K_{3,3};\bbC) = \S{3^2}\oplus \S{41^2}^{\oplus11}\oplus \S{42}^{\oplus6},
    \quad
    H_4(K_{3,3};\bbC) = \S{51}^{\oplus5}.
  \]
\end{conjecture}
\begin{remark}
  Our computations show that the homology of the nonplanar graph $G_6$ in Figure~\ref{fig.pairs_three} (and graph number 137 in Appendix~\ref{computations}) contains $\bbZ_2$-torsion, namely, $H_1(G_6;\bbZ) = \bbZ^{30} \oplus\bbZ_2^9$.
  It is generated by a linear combination of polytabloids of shape $(2^2,1^2)$.
  Also see remark~\ref{rem.G6}.
\end{remark}

\subsection{Torsion and Planarity}
\label{infinitefamily}
The pair of graphs in Figure~\ref{fig.pairs_three} appears in Chow's thesis~\cite[Section 6]{Cho95} as an example showing that the chromatic symmetric function does not detect planarity of graphs.
In Section~\ref{sectorsion} we showed that the chromatic symmetric homology of the nonplanar graphs $K_5$ and $K_{3,3}$ contains torsion.
Here, we conjecture that the appearance of $\bbZ_2$-torsion in bidegree $(1,0)$ chromatic symmetric homology detects nonplanarity of graphs.
Assuming that the following conjecture is true, then based on a theorem of Orellana and Scott~\cite[Theorem 4.2]{OS14}, we are able to construct an infinite family of pairs of graphs with equal chromatic symmetric function but distinct homology.

\begin{conjecture}
  \label{tor.conj}
  A graph $G$ is nonplanar if and only if the $q$-degree zero homology $H_{1}(G;\bbZ)$ contains $\bbZ_2$-torsion.
\end{conjecture}

This conjecture has been computationally verified on all $143$ connected graphs on up to $6$ vertices (see Section \ref{computations}).

\begin{remark}
  In the case of Khovanov homology, the question (2) from the introduction is quite obvious: $\bbZ_2$-torsion appears to arise in the Khovanov homology of any nontrivial knot.
  In 2011, Kronheimer and Mrowka \cite{Kron2011} used gauge theory to show that Khovanov homology detects whether a knot is trivial, a question which is still open for the Jones polynomial.
  The abundance of $\bbZ_2$-torsion in Khovanov homology led Shumakovitch to conjecture that in fact it is the presence of $\bbZ_2$-torsion which detects nontriviality \cite{S1}.
  Conjecture~\ref{tor.conj} can be seen as an analogue of Shumakovitch's conjecture in the case of chromatic symmetric homology.
\end{remark}

In what follows, it will be convenient to have a graph $G$ in mind which will undergo some ``local'' modifications.
We will use diagrams such as \usebox\graphUZVW \ and \usebox\graphUWUZVWWZ \ to denote graphs which are identical outside of the dotted region, but which differ as shown inside the dotted region.
In other words, we are fixing a graph and considering different possibilities for the induced subgraph on the indicated subset of its vertices.
In this notation, we restate a result of Orellana and Scott.

\begin{theorem}{\cite[Theorem 4.2]{OS14}}
  \label{orellana}
  Let $G$ be a graph, and let $u,v,w,z$ be vertices of $G$. Suppose there is an automorphism $f$ on the graph $G=\usebox\graphUZVWlabeled$ such that $f(\{u,w\})=\{v,z\}$ and $f(\{v,z\})=\{u,w\}$.
  Then the graphs \usebox\graphUWUZVWWZ \ and \usebox\graphUZVWVZWZ \ have the same chromatic symmetric function.
\end{theorem}

\begin{remark}
  Theorem \ref{orellana} involves four possible types of maps via the requirement that $\{u,w\}\xleftrightarrow[]{f}\{v,z\}$: horizontal reflection, vertical reflection, $90$ degree rotation clockwise, and $90$ degree rotation counterclockwise.
  The two rotations cannot possibly be automorphisms, so only the reflections need be considered.
  However, in the case of vertical reflection $(u\leftrightarrow v, z\leftrightarrow w)$, it is clear that the graphs \usebox\graphUWUZVWWZ \ and \usebox\graphUZVWVZWZ \ are isomorphic.
  Thus, of the four maps possible in Theorem \ref{orellana}, the only maps for which the statement is nontrivial and nonvacuous are the automorphisms via horizontal reflection, interchanging $u\leftrightarrow z$ and $v\leftrightarrow w$.
\end{remark}

Let us suppose for the moment that our Conjecture \ref{tor.conj} holds true.
Then to construct pairs of graphs with equal chromatic symmetric function, but distinct homology, it will suffice to find graphs \usebox\graphUZVWlabeled with an automorphism via horizontal reflection such that \usebox\graphUWUZVWWZ \ is planar but  \usebox\graphUZVWVZWZ \ is not.

The graph $G$ in Figure~\ref{fig.inf.family} has an automorphism via horizontal reflection. 
The graphs $G_7$ and $G_8$ are constructed by adding edges to $G$, where the circles labeled $P_1$ and $P_2$ indicate the attachment of horizontally symmetric planar graphs attached to the left and rightmost vertices of $G$.
By~\cite[Theorem 4.2]{OS14}, the graphs $G_7$ and $G_8$ have identical chromatic symmetric functions.
One can check that $G_7$ contains a copy of $K_{3,3}$, while $G_8$ is planar, hence conjecturally, they are distinguished by chromatic symmetric homology.

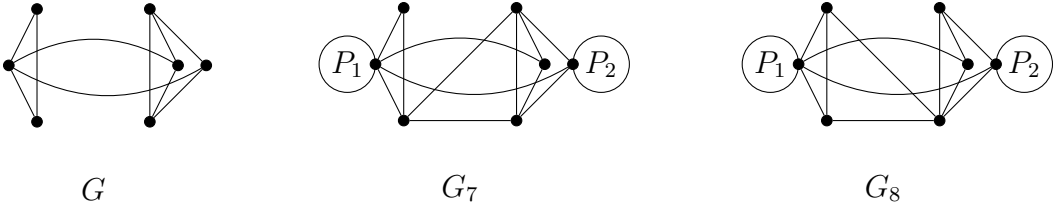
\begin{figure}[ht]
  \centering
  \begin{tikzpicture}[scale=.75]
    \node at (0,0){\usebox\infinitezero};
    \node at (6.5,0){\usebox\infiniteone};
    \node at (14,0){\usebox\infinitetwo};
  \end{tikzpicture}
  \caption{
    The construction of an infinite family of pairs of graphs which are conjecturally distinguished by chromatic symmetric homology, but are known to not be distinguished by the chromatic symmetric function.
    }
  \label{fig.inf.family}
\end{figure}
The following conjecture would follow if Conjecture~\ref{tor.conj} holds.

\begin{conjecture}
  \label{infinitefamilyconj}
  Each pair of graphs constructed as in Figure \ref{fig.inf.family} is not distinguished by the chromatic symmetric function but is distinguished by chromatic symmetric homology.
\end{conjecture}
We verified computationally Conjecture \ref{infinitefamilyconj} when $P_1$ and $P_2$ are trivial in which case, as $\bbZ$-modules, we get the following homology groups:

\begin{center}
  \begin{tabular}{|Sc|Sc|Sc|}
    \hline
    $i$ & $H_i(G_7;\bbZ)$ as a $\bbZ$-module & $H_i(G_8;\bbZ)$ as a $\bbZ$-module \\
    \hline
    \hline
    $0$ & $\bbZ$                             & $\bbZ$                             \\
    \hline
    $1$ & $\bbZ^{50} \oplus \bbZ_2^{14}$     & $\bbZ^{64}$                        \\
    \hline
    $2$ & $\bbZ^{787}$                       & $\bbZ^{857}$                       \\
    \hline
    $3$ & $\bbZ^{1514}$                      & $\bbZ^{1570}$                      \\
    \hline
    $4$ & $\bbZ^{496}$                       & $\bbZ^{496}$                       \\
    \hline
    $5$ & $\bbZ^{36}$                        & $\bbZ^{36}$                        \\
    \hline
  \end{tabular}
\end{center}

\subsection{Other types of torsion}
\label{z3torsion}

In the case of Khovanov homology, it was conjectured in \cite{Kh} that odd torsion does not exist.
It was not until 7 years later that Bar-Natan~\cite{Bar} found a counter-example: the $(5,6)$-torus knot with 24 crossings was shown to have $\bbZ_3$-torsion in homology.
For chromatic symmetric homology, we are able to computationally verify that the star graph on 7 vertices has $\ZZ_3$-torsion in bidegree $(1,0)$.

Below is a chart showing the computed homology over $\bbZ$ of the star graphs on $n=4,\ldots, 7$ vertices, and the conjectured homology over $\bbC$ as a $\bbC[\fS_n]$-module.
We use the same notation to list $H_*(G)$ as in Appendix~\ref{computations}.

\begin{center}
\begin{tabular}{|Sc|>{\centering\arraybackslash}m{5.7cm}|>{\centering\arraybackslash}m{5.7cm}|}
	\hline
  $G$ & $H_*(G;\bbZ)$ as a $\bbZ$-module &  $H_*(G;\bbC)$ as a $\bbC[\fS_n]$-module
	\\
  \hline\hline
  $\begin{tikzpicture}[scale=0.18, baseline={([yshift=-0.3em]current bounding box.center)}]
    \SetVertexSimple[MinSize=3pt, InnerSep=0pt]
    \definecolor{cv0}{rgb}{0.0,0.0,0.0}
    \definecolor{cfv0}{rgb}{1.0,1.0,1.0}
    \definecolor{clv0}{rgb}{0.0,0.0,0.0}
    \definecolor{cv1}{rgb}{0.0,0.0,0.0}
    \definecolor{cfv1}{rgb}{1.0,1.0,1.0}
    \definecolor{clv1}{rgb}{0.0,0.0,0.0}
    \definecolor{cv2}{rgb}{0.0,0.0,0.0}
    \definecolor{cfv2}{rgb}{1.0,1.0,1.0}
    \definecolor{clv2}{rgb}{0.0,0.0,0.0}
    \definecolor{cv3}{rgb}{0.0,0.0,0.0}
    \definecolor{cfv3}{rgb}{1.0,1.0,1.0}
    \definecolor{clv3}{rgb}{0.0,0.0,0.0}
    \definecolor{cv0v1}{rgb}{0.0,0.0,0.0}
    \definecolor{cv0v2}{rgb}{0.0,0.0,0.0}
    \definecolor{cv0v3}{rgb}{0.0,0.0,0.0}
    \Vertex[NoLabel,x=2.5cm,y=2.5cm]{v0}
    \Vertex[NoLabel,x=2.5cm,y=5.0cm]{v1}
    \Vertex[NoLabel,x=0.0cm,y=1.25cm]{v2}
    \Vertex[NoLabel,x=5.0cm,y=1.25cm]{v3}
    \Vertex[NoLabel,x=2.5cm,y=0cm,empty=true]{empty}
    \Edge[lw=0.1cm,style={color=cv0v1,},](v0)(v1)
    \Edge[lw=0.1cm,style={color=cv0v2,},](v0)(v2)
    \Edge[lw=0.1cm,style={color=cv0v3,},](v0)(v3)
  \end{tikzpicture}$ 
  & $\bbZ,\bbZ^2$
  & $\S{1^4},\S{2^2}$
  \\
  \hline
  $\begin{tikzpicture}[scale=0.18, baseline={([yshift=-0.3em]current bounding box.center)}]
    \SetVertexSimple[MinSize=3pt, InnerSep=0pt]
    \definecolor{cv0}{rgb}{0.0,0.0,0.0}
    \definecolor{cfv0}{rgb}{1.0,1.0,1.0}
    \definecolor{clv0}{rgb}{0.0,0.0,0.0}
    \definecolor{cv1}{rgb}{0.0,0.0,0.0}
    \definecolor{cfv1}{rgb}{1.0,1.0,1.0}
    \definecolor{clv1}{rgb}{0.0,0.0,0.0}
    \definecolor{cv2}{rgb}{0.0,0.0,0.0}
    \definecolor{cfv2}{rgb}{1.0,1.0,1.0}
    \definecolor{clv2}{rgb}{0.0,0.0,0.0}
    \definecolor{cv3}{rgb}{0.0,0.0,0.0}
    \definecolor{cfv3}{rgb}{1.0,1.0,1.0}
    \definecolor{clv3}{rgb}{0.0,0.0,0.0}
    \definecolor{cv4}{rgb}{0.0,0.0,0.0}
    \definecolor{cfv4}{rgb}{1.0,1.0,1.0}
    \definecolor{clv4}{rgb}{0.0,0.0,0.0}
    \definecolor{cv0v1}{rgb}{0.0,0.0,0.0}
    \definecolor{cv0v2}{rgb}{0.0,0.0,0.0}
    \definecolor{cv0v3}{rgb}{0.0,0.0,0.0}
    \definecolor{cv0v4}{rgb}{0.0,0.0,0.0}
    \Vertex[NoLabel,x=2.5cm,y=2.5cm]{v0}
    \Vertex[NoLabel,x=2.5cm,y=5.0cm]{v1}
    \Vertex[NoLabel,x=0.0cm,y=2.5cm]{v2}
    \Vertex[NoLabel,x=2.5cm,y=0.0cm]{v3}
    \Vertex[NoLabel,x=5.0cm,y=2.5cm]{v4}
    \Edge[lw=0.1cm,style={color=cv0v1,},](v0)(v1)
    \Edge[lw=0.1cm,style={color=cv0v2,},](v0)(v2)
    \Edge[lw=0.1cm,style={color=cv0v3,},](v0)(v3)
    \Edge[lw=0.1cm,style={color=cv0v4,},](v0)(v4)
  \end{tikzpicture}$ 
  & $\bbZ,\bbZ^{20}$
  & $\S{1^5}, \S{2^21}^{\oplus3}\oplus\S{32}$
  \\
  \hline
  $\begin{tikzpicture}[scale=0.18, baseline={([yshift=-0.3em]current bounding box.center)}]
    \SetVertexSimple[MinSize=3pt, InnerSep=0pt]
    \definecolor{cv0}{rgb}{0.0,0.0,0.0}
    \definecolor{cfv0}{rgb}{1.0,1.0,1.0}
    \definecolor{clv0}{rgb}{0.0,0.0,0.0}
    \definecolor{cv1}{rgb}{0.0,0.0,0.0}
    \definecolor{cfv1}{rgb}{1.0,1.0,1.0}
    \definecolor{clv1}{rgb}{0.0,0.0,0.0}
    \definecolor{cv2}{rgb}{0.0,0.0,0.0}
    \definecolor{cfv2}{rgb}{1.0,1.0,1.0}
    \definecolor{clv2}{rgb}{0.0,0.0,0.0}
    \definecolor{cv3}{rgb}{0.0,0.0,0.0}
    \definecolor{cfv3}{rgb}{1.0,1.0,1.0}
    \definecolor{clv3}{rgb}{0.0,0.0,0.0}
    \definecolor{cv4}{rgb}{0.0,0.0,0.0}
    \definecolor{cfv4}{rgb}{1.0,1.0,1.0}
    \definecolor{clv4}{rgb}{0.0,0.0,0.0}
    \definecolor{cv5}{rgb}{0.0,0.0,0.0}
    \definecolor{cfv5}{rgb}{1.0,1.0,1.0}
    \definecolor{clv5}{rgb}{0.0,0.0,0.0}
    \definecolor{cv0v1}{rgb}{0.0,0.0,0.0}
    \definecolor{cv0v2}{rgb}{0.0,0.0,0.0}
    \definecolor{cv0v3}{rgb}{0.0,0.0,0.0}
    \definecolor{cv0v4}{rgb}{0.0,0.0,0.0}
    \definecolor{cv0v5}{rgb}{0.0,0.0,0.0}
    \Vertex[NoLabel,x=2.5cm,y=2.2361cm]{v0}
    \Vertex[NoLabel,x=2.5cm,y=5.0cm]{v1}
    \Vertex[NoLabel,x=0.0cm,y=3.0902cm]{v2}
    \Vertex[NoLabel,x=0.9549cm,y=0.0cm]{v3}
    \Vertex[NoLabel,x=4.0451cm,y=0.0cm]{v4}
    \Vertex[NoLabel,x=5.0cm,y=3.0902cm]{v5}
    \Edge[lw=0.1cm,style={color=cv0v1,},](v0)(v1)
    \Edge[lw=0.1cm,style={color=cv0v2,},](v0)(v2)
    \Edge[lw=0.1cm,style={color=cv0v3,},](v0)(v3)
    \Edge[lw=0.1cm,style={color=cv0v4,},](v0)(v4)
    \Edge[lw=0.1cm,style={color=cv0v5,},](v0)(v5)
  \end{tikzpicture}$
  & $\bbZ,\bbZ^{152}$
  & $\S{1^6}, \S{2^21^2}^{\oplus6}\oplus\S{2^3}^{\oplus5}\oplus\S{321}^{\oplus4}\oplus\S{42}$
  \\
  \hline
  $\begin{tikzpicture}[scale=0.18, baseline={([yshift=-0.3em]current bounding box.center)}]
    \SetVertexSimple[MinSize=3pt, InnerSep=0pt]
    \definecolor{cv0}{rgb}{0.0,0.0,0.0}
    \definecolor{cfv0}{rgb}{1.0,1.0,1.0}
    \definecolor{clv0}{rgb}{0.0,0.0,0.0}
    \definecolor{cv1}{rgb}{0.0,0.0,0.0}
    \definecolor{cfv1}{rgb}{1.0,1.0,1.0}
    \definecolor{clv1}{rgb}{0.0,0.0,0.0}
    \definecolor{cv2}{rgb}{0.0,0.0,0.0}
    \definecolor{cfv2}{rgb}{1.0,1.0,1.0}
    \definecolor{clv2}{rgb}{0.0,0.0,0.0}
    \definecolor{cv3}{rgb}{0.0,0.0,0.0}
    \definecolor{cfv3}{rgb}{1.0,1.0,1.0}
    \definecolor{clv3}{rgb}{0.0,0.0,0.0}
    \definecolor{cv4}{rgb}{0.0,0.0,0.0}
    \definecolor{cfv4}{rgb}{1.0,1.0,1.0}
    \definecolor{clv4}{rgb}{0.0,0.0,0.0}
    \definecolor{cv5}{rgb}{0.0,0.0,0.0}
    \definecolor{cfv5}{rgb}{1.0,1.0,1.0}
    \definecolor{clv5}{rgb}{0.0,0.0,0.0}
    \definecolor{cv6}{rgb}{0.0,0.0,0.0}
    \definecolor{cfv6}{rgb}{1.0,1.0,1.0}
    \definecolor{clv6}{rgb}{0.0,0.0,0.0}
    \definecolor{cv0v1}{rgb}{0.0,0.0,0.0}
    \definecolor{cv0v2}{rgb}{0.0,0.0,0.0}
    \definecolor{cv0v3}{rgb}{0.0,0.0,0.0}
    \definecolor{cv0v4}{rgb}{0.0,0.0,0.0}
    \definecolor{cv0v5}{rgb}{0.0,0.0,0.0}
    \definecolor{cv0v6}{rgb}{0.0,0.0,0.0}
    \Vertex[NoLabel,x=2.5cm,y=2.5cm]{v0}
    \Vertex[NoLabel,x=2.5cm,y=5.0cm]{v1}
    \Vertex[NoLabel,x=0.0cm,y=3.75cm]{v2}
    \Vertex[NoLabel,x=0.0cm,y=1.25cm]{v3}
    \Vertex[NoLabel,x=2.5cm,y=0.0cm]{v4}
    \Vertex[NoLabel,x=5.0cm,y=1.25cm]{v5}
    \Vertex[NoLabel,x=5.0cm,y=3.75cm]{v6}
    \Edge[lw=0.1cm,style={color=cv0v1,},](v0)(v1)
    \Edge[lw=0.1cm,style={color=cv0v2,},](v0)(v2)
    \Edge[lw=0.1cm,style={color=cv0v3,},](v0)(v3)
    \Edge[lw=0.1cm,style={color=cv0v4,},](v0)(v4)
    \Edge[lw=0.1cm,style={color=cv0v5,},](v0)(v5)
    \Edge[lw=0.1cm,style={color=cv0v6,},](v0)(v6)
  \end{tikzpicture}$
  & $\bbZ,\bbZ^{1092}\oplus\bbZ_3$
  & $\S{1^7}, \S{2^21^3}^{\oplus10}\oplus\S{2^31}^{\oplus16}\oplus\S{321^2}^{\oplus10}\oplus\S{32^2}^{\oplus9}\oplus\S{421}^{\oplus5}\oplus\S{52}$
  \\
  \hline
\end{tabular}
\end{center}

Observe that the multiplicity of $\S{n-2,2}$ in $H_1(G;\bbC)$ is $1$, and the multiplicity of $\S{2^2,1^{n-4}}$ in $H_1(G;\bbC)$ is $\binom{n-2}{2}$, for $n=4,\ldots,7$. 



\appendix
\section{A simple example: the triangle}
\label{app.K3}
\label{sec-A}
In this example, we compute $H_1(K_3;\bbC)$ for the complete graph $K_3$, following the conventions laid out in Section~\ref{sec.restrict}. Consider $K_3$ with the following vertices and edges ordered lexicographically: $e_{12}, e_{13}, e_{23}$.
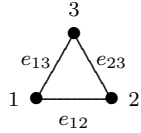
\[
  \beginpicture
  \setcoordinatesystem units <1cm,1cm>              
  \setplotarea x from 0 to 0, y from -0.5 to 1.5    
  \multiput{$\bullet$} at 0 0 1 0 0.5 0.867 /
  \plot 0 0 1 0 0.5 .867 0 0 /
  \put{\tiny$3$} at 0.5 1.2
  \put{\tiny$2$} at 1.3 0
  \put{\tiny$1$} at -0.3 0
  \put{\tiny$e_{23}$} at 1 .5
  \put{\tiny$e_{12}$} at .5 -.3
  \put{\tiny$e_{13}$} at 0 .5
  \endpicture
\]
The following numberings index cyclic generators for the permutation modules associated to $K_3$.
\[
  \ytableausetup{notabloids}
  U = \ytableaushort{123}, \qquad
  X_1 = \ytableaushort{12,3},\qquad
  X_2 = \ytableaushort{13,2},\qquad
  X_3 = \ytableaushort{23,1}, \qquad
  V = \ytableaushort{1,2,3}.
\]
For $i=1,2$, let $s_i$ denote the simple transposition that exchanges $i$ and $i+1$.
The symmetric group is $\fS_3 = \{e, s_1,s_2,s_1s_2,s_2s_1,s_1s_2s_1\}$, and given a numbering $T$, we have the permutation modules $\calM_{T} = \bbC[\fS_3]a_T$ in $q$-degree zero.
\[
  \begin{array}{llll}
    a_U\!\!\! &= e+s_1+s_2+s_1s_2+s_2s_1+s_1s_2s_1, &\quad \calM_U\!\!\! &\cong \M{3} \cong \S{3},\\
    a_{X_1}\!\!\!\!\! &= e+s_1, &\quad   \calM_{X_1}\!\!\! &\cong \M{2,1} \cong \S{3}\oplus\S{2,1},\\
    a_{X_2}\!\!\!\!\! &= e+s_1s_2s_1, &\quad \calM_{X_2}\!\!\! &\cong \M{2,1} \cong \S{3}\oplus\S{2,1},\\
    a_{X_3}\!\!\!\!\! &= e+s_2, &\quad \calM_{X_3}\!\!\! &\cong \M{2,1} \cong \S{3}\oplus\S{2,1},\\
    a_V\!\!\!\!\! &= e, &\quad \calM_V\!\!\! &\cong \M{1^3} \cong \S{3}\oplus\S{2,1}^{\oplus2}\oplus \S{1^3}.
  \end{array}
\]

The $\fS_3$-modules for the computation of the $q$-degree zero homology of $K_3$ are indicated in Figure~\ref{fig.k3}.
Each edge map $\pm d_{\varepsilon}$ is a (signed) inclusion.
\begin{figure}[ht!]
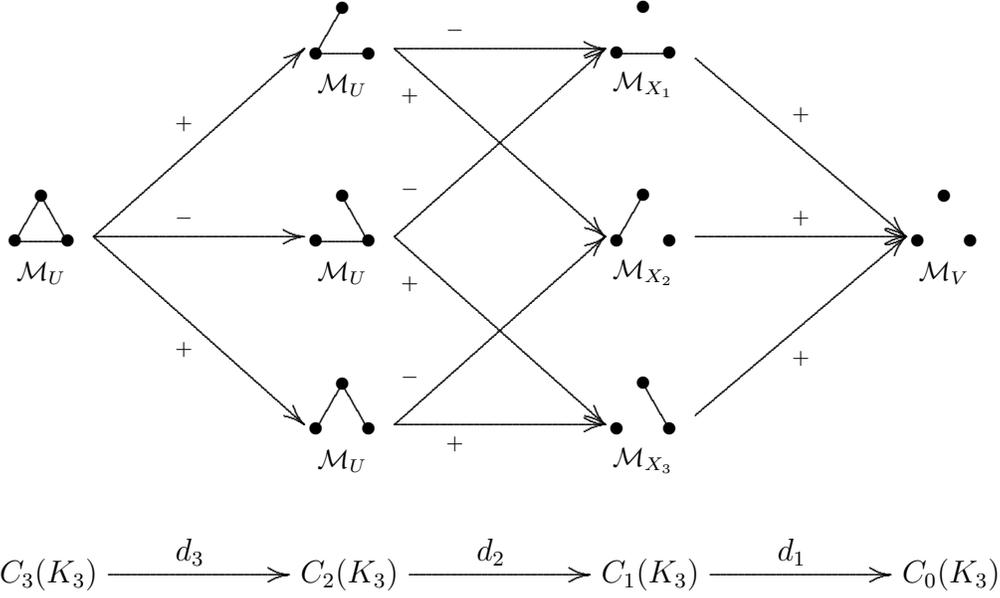

  \[
    \beginpicture
    \setcoordinatesystem units <4cm,2.5cm>         
    \setplotarea x from 0 to 0, y from -2 to 1.3    



    \arrow <8pt> [.2,.67] from 0.1 0 to .8 1   \put{\tiny$+$} at .4 .6
    \arrow <8pt> [.2,.67] from 0.1 0 to .8 0    \put{\tiny$-$} at .4 .1
    \arrow <8pt> [.2,.67] from 0.1 0 to .8 -1   \put{\tiny$+$} at .4 -.6

    \arrow <8pt> [.2,.67] from 1.1 1 to 1.8 1   \put{\tiny$-$}[c] at 1.3 1.1
    \arrow <8pt> [.2,.67] from 1.1 1 to 1.8 0   \put{\tiny$+$}[c] at 1.15 .75
    \arrow <8pt> [.2,.67] from 1.1 0 to 1.8 1   \put{\tiny$-$}[c] at 1.15 .25
    \arrow <8pt> [.2,.67] from 1.1 0 to 1.8 -1  \put{\tiny$+$}[c] at 1.15 -.25
    \arrow <8pt> [.2,.67] from 1.1 -1 to 1.8 0  \put{\tiny$-$}[c] at 1.15 -.75
    \arrow <8pt> [.2,.67] from 1.1 -1 to 1.8 -1 \put{\tiny$+$}[c] at 1.3 -1.1

    \arrow <8pt> [.2,.67] from 2.1 .95 to 2.8 0  \put{\tiny$+$}[c] at 2.45 .65
    \arrow <8pt> [.2,.67] from 2.1 0 to 2.8 0    \put{\tiny$+$}[c] at 2.45 .1
    \arrow <8pt> [.2,.67] from 2.1 -.95 to 2.8 0 \put{\tiny$+$}[c] at 2.45 -.65

    \put{
      $\beginpicture
      \setcoordinatesystem units <.7cm,.7cm>         
      \setplotarea x from 0 to 2, y from 0 to 0    
      \multiput{$\bullet$} at 0 0 1 0 0.5 0.867 /
      \plot 0 0 1 0 0.5 .867 0 0 /
      \put{\footnotesize$\calM_U$} at .5 -.6
      \endpicture$
    }[c] at 0 0

    \put{
      $\beginpicture
      \setcoordinatesystem units <.7cm,.7cm>         
      \setplotarea x from 0 to 2, y from 0 to 0    
      \multiput{$\bullet$} at 0 0 1 0 0.5 0.867 /
      \plot 1 0 0 0 0.5 .867 /
      \put{\footnotesize$\calM_{U}$} at .5 -.6
      \endpicture$
    }[c] at 1 1
    \put{
      $\beginpicture
      \setcoordinatesystem units <.7cm,.7cm>         
      \setplotarea x from 0 to 2, y from 0 to 0    
      \multiput{$\bullet$} at 0 0 1 0 0.5 0.867 /
      \plot 0.5 .867 1 0 0 0 /
      \put{\footnotesize$\calM_{U}$} at .5 -.6
      \endpicture$
    }[c] at 1 0
    \put{
      $\beginpicture
      \setcoordinatesystem units <.7cm,.7cm>         
      \setplotarea x from 0 to 2, y from 0 to 0    
      \multiput{$\bullet$} at 0 0 1 0 0.5 0.867 /
      \plot 0 0 0.5 .867 1 0 /
      \put{\footnotesize$\calM_{U}$} at .5 -.6
      \endpicture$
    }[c] at 1 -1

    \put{
      $\beginpicture
      \setcoordinatesystem units <.7cm,.7cm>         
      \setplotarea x from 0 to 2, y from 0 to 0    
      \multiput{$\bullet$} at 0 0 1 0 0.5 0.867 /
      \plot 0 0 1 0 /
      \put{\footnotesize$\calM_{X_1}$} at .5 -.6
      \endpicture$
    }[c] at 2 1
    \put{
      $\beginpicture
      \setcoordinatesystem units <.7cm,.7cm>         
      \setplotarea x from 0 to 2, y from 0 to 0    
      \multiput{$\bullet$} at 0 0 1 0 0.5 0.867 /
      \plot 0 0 0.5 .867 /
      \put{\footnotesize$\calM_{X_2}$} at .5 -.6
      \endpicture$
    }[c] at 2 0
    \put{
      $\beginpicture
      \setcoordinatesystem units <.7cm,.7cm>         
      \setplotarea x from 0 to 2, y from 0 to 0    
      \multiput{$\bullet$} at 0 0 1 0 0.5 0.867 /
      \plot 1 0 0.5 .867 /
      \put{\footnotesize$\calM_{X_3}$} at .5 -.6
      \endpicture$
    }[c] at 2 -1

    \put{
      $\beginpicture
      \setcoordinatesystem units <.7cm,.7cm>         
      \setplotarea x from 0 to 2, y from 0 to 0    
      \multiput{$\bullet$} at 0 0 1 0 0.5 0.867 /
      \put{\footnotesize$\calM_V$} at .5 -.6
      \endpicture$
    } at 3 0

    \put{$C_{3}(K_3)$}[c] at -.05 -1.8
    \put{$C_2(K_3)$}[c] at .95 -1.8
    \put{$C_{1}(K_3)$}[c] at 1.95 -1.8
    \put{$C_{0}(K_3)$}[c] at 2.95 -1.8
    \arrow <8pt> [.2,.67] from 0.15 -1.8 to 0.75 -1.8
    \put{$d_{3}$} at .42 -1.68
    \arrow <8pt> [.2,.67] from 1.15 -1.8 to 1.75 -1.8
    \put{$d_{2}$} at 1.42 -1.68
    \arrow <8pt> [.2,.67] from 2.15 -1.8 to 2.75 -1.8
    \put{$d_{1}$} at 2.42 -1.68
    \endpicture
  \]
  \caption{The Boolean lattice of subgraphs of $K_3$, showing the signed edges.}
  \label{fig.k3}
\end{figure}

\begin{figure}[th]
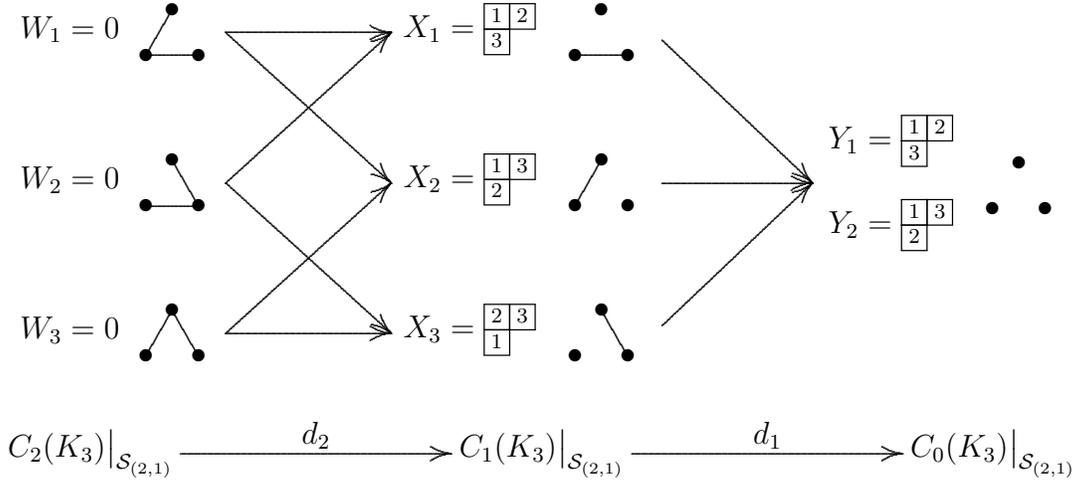

  \[
    \beginpicture
    \setcoordinatesystem units <4cm,2cm>         
    \setplotarea x from 0 to 0, y from -2 to 1.3    

    \arrow <8pt> [.2,.67] from 0.45 1 to 1 1
    \arrow <8pt> [.2,.67] from 0.45 1 to 1 0
    \arrow <8pt> [.2,.67] from 0.45 0 to 1 1
    \arrow <8pt> [.2,.67] from 0.45 0 to 1 -1
    \arrow <8pt> [.2,.67] from 0.45 -1 to 1 0
    \arrow <8pt> [.2,.67] from 0.45 -1 to 1 -1

    \arrow <8pt> [.2,.67] from 1.9 .95 to 2.4 0
    \arrow <8pt> [.2,.67] from 1.9 0 to 2.4 0
    \arrow <8pt> [.2,.67] from 1.9 -.95 to 2.4 0

    \put{
      $\beginpicture
      \setcoordinatesystem units <.7cm,.7cm>         
      \setplotarea x from 0 to 2, y from 0 to 0    
      \multiput{$\bullet$} at 0 0 1 0 0.5 0.867 /
      \plot 1 0 0 0 0.5 .867 /
      \put{$W_{1}=0$}[r] at -.5 .5
      \endpicture$
    }[c] at 0.15 1
    \put{
      $\beginpicture
      \setcoordinatesystem units <.7cm,.7cm>         
      \setplotarea x from 0 to 2, y from 0 to 0    
      \multiput{$\bullet$} at 0 0 1 0 0.5 0.867 /
      \plot 0.5 .867 1 0 0 0 /
      \put{$W_{2}=0$}[r] at -.5 .5
      \endpicture$
    }[c] at 0.15 0
    \put{
      $\beginpicture
      \setcoordinatesystem units <.7cm,.7cm>         
      \setplotarea x from 0 to 2, y from 0 to 0    
      \multiput{$\bullet$} at 0 0 1 0 0.5 0.867 /
      \plot 0 0 0.5 .867 1 0 /
      \put{$W_{3}=0$}[r] at -.5 .5
      \endpicture$
    }[c] at 0.15 -1

    \put{
      $\beginpicture
      \setcoordinatesystem units <.7cm,.7cm>         
      \setplotarea x from 0 to 2, y from 0 to 0    
      \multiput{$\bullet$} at 0 0 1 0 0.5 0.867 /
      \plot 0 0 1 0 /
      \put{$X_{1}=\ytableaushort{12,3}$}[r] at -.75 .5
      \endpicture$
    }[c] at 1.5 1
    \put{
      $\beginpicture
      \setcoordinatesystem units <.7cm,.7cm>         
      \setplotarea x from 0 to 2, y from 0 to 0    
      \multiput{$\bullet$} at 0 0 1 0 0.5 0.867 /
      \plot 0 0 0.5 .867 /
      \put{$X_{2}=\ytableaushort{13,2}$}[r] at -.75 .5
      \endpicture$
    }[c] at 1.5 0
    \put{
      $\beginpicture
      \setcoordinatesystem units <.7cm,.7cm>         
      \setplotarea x from 0 to 2, y from 0 to 0    
      \multiput{$\bullet$} at 0 0 1 0 0.5 0.867 /
      \plot 1 0 0.5 .867 /
      \put{$X_{3}=\ytableaushort{23,1}$}[r] at -.75 .5
      \endpicture$
    }[c] at 1.5 -1

    \put{
      $\beginpicture
      \setcoordinatesystem units <.7cm,.7cm>         
      \setplotarea x from 0 to 2, y from 0 to 0    
      \multiput{$\bullet$} at 0 0 1 0 0.5 0.867 /
      \put{$Y_1=\ytableaushort{12,3}$}[r] at -.75 1.3
      \put{$Y_2=\ytableaushort{13,2}$}[r] at -.75 -.3
      \endpicture$
    } at 2.9 0

    \put{$C_2(K_3)\big|_{\S{2,1}} $}[c] at 0 -1.8
    \put{$C_{1}(K_3)\big|_{\S{2,1}} $}[c] at 1.5 -1.8
    \put{$C_{0}(K_3)\big|_{\S{2,1}} $}[c] at 3 -1.8
    \arrow <8pt> [.2,.67] from 0.3 -1.8 to 1.2 -1.8
    \put{$d_{2}$} at .75 -1.68
    \arrow <8pt> [.2,.67] from 1.8 -1.8 to 2.7 -1.8
    \put{$d_{1}$} at 2.25 -1.68
    \endpicture
  \]
  \caption{Computation of $H_1(K_3;\bbC)$ restricted to the Specht module $\S{2,1}$.}
  \label{fig.k3prime}
\end{figure}

Next we compute the homology restricted to the Specht modules isomorphic to $\S{2,1}$.
First, it is clear that $C_3(K_3)\big|_{\S{2,1}} = C_2(K_3)\big|_{\S{2,1}}=0$.
For $C_1(K_3)\big|_{\S{2,1}} \cong \S{2,1}^{\oplus3}$, we choose the generators
\begin{align*}
  v_{X_1}^{Y_1} &= c_{X_1} \in \calM_{X_1},\\
  v_{X_2}^{Y_1} &= \sigma_{X_2,Y_1}c_{X_2} = s_2 c_{X_2}= c_{X_1}s_2 \in \calM_{X_2},\\
  v_{X_3}^{Y_1} &= \sigma_{X_3,Y_1}c_{X_3} = s_2s_1c_{X_3} = c_{X_1}s_2s_1 \in \calM_{X_3},
\end{align*}
so that $C_1(K_3)\big|_{\S{2,1}} = \bigoplus_{i=1}^3 \bbC[\fS_3]\cdot v_{X_i}^{Y_1}$.
Finally for $C_0(K_3)\big|_{\S{2,1}}\cong \S{2,1}^{\oplus2}$, we choose the generators $v_{Y_1}^{Y_1}, v_{Y_2}^{Y_1}$ so that $C_0(K_3)\big|_{\S{2,1}} = \bigoplus_{i=1}^2 \bbC[\fS_3] \cdot v_{Y_i}^{Y_1}$.
See Example~\ref{eg.forappA} for a related computation.

Applying the straightening laws to $v_{X_3}^{Y_1}$, we get
\[
  \ytableausetup{smalltableaux}
  v_{\,\scalebox{.75}{\ytableaushort{23,1}}}^{Y_1}
  =
  -v_{\,\scalebox{.75}{\ytableaushort{13,2}}}^{Y_1} - v_{\,\scalebox{.75}{\ytableaushort{21,3}}}^{Y_1}
  = -v_{\,\scalebox{.75}{\ytableaushort{13,2}}}^{Y_1} - v_{\,\scalebox{.75}{\ytableaushort{12,3}}}^{Y_1},
\]
so the only non-zero differential defined on our generators is
\[
  \setlength{\arraycolsep}{3pt}
  d_{1}\big|_{\S{2,1}} =
  \kbordermatrix{
        & X_1 & X_2 & X_3 \\
    Y_1 &   1 &   0 &  \n \\
    Y_2 &   0 &   1 &  \n \\
  }.
\]

It follows that $\ker d_1$ is generated by $v_{X_1}^{Y_1}+v_{X_2}^{Y_1}+v_{X_3}^{Y_1}$ and $H_1(K_3;\bbC) \cong \S{2,1}$.
It can also be shown that $H_0(K_3;\bbC) \cong \S{1^3}$.

\section{Proofs}
\label{app.bigmatrices}
\label{sec-B}
This section contains the proofs of Lemmas~\ref{lemma.G3},~\ref{lemma.G4},~\ref{lemma.G5},~\ref{lemma.G6}, and Theorem~\ref{thm.K33}.

\noindent{\bf Lemma~\ref{lemma.G3}.} The multiplicity of the $\fS_6$-module $\S{2^3}$ in $H_{1}(G_3;\bbC)$ is $1$.
\begin{proof}
  The chain modules are $C_0(G_3)\big|_{\S{2^3}} \cong \S{2^3}^{\oplus 5}$, $C_1(G_3)\big|_{\S{2^3}}\cong \S{2^3}^{\oplus 23}$, and $C_2(G_3)\big|_{\S{2^3}}\cong \S{2^3}^{\oplus 22}$.
  The differentials $d_2(G_3)$ and $d_1(G_3)$ are described by the following matrices (here, $a=10$ and $b=11$):
  \[
    \setlength{\arraycolsep}{.75pt}
    \renewcommand{\arraystretch}{1.0}
    \tiny
    d_2(G_3) =
    \kbordermatrix{
               & W_{16} & W_{17} & W_{18} & W_{19} & W_{1a} & W_{1b} & W_{25} & W_{29} & W_{2a} & W_{2b} & W_{34} & W_{37} & W_{38} & W_{3b} & W_{49} & W_{4a} & W_{4b} & W_{57} & W_{58} & W_{5b} & W_{6b} & W_{7a} & W_{89} \\
      X_1^1    &     \n &      0 &      1 &      1 &      0 &     \n &      0 &      0 &      0 &      0 &      0 &      0 &      0 &      0 &      0 &      0 &      0 &      0 &      0 &      0 &      0 &      0 &      0 \\
      X_1^2    &      0 &     \n &      1 &      1 &     \n &      0 &      0 &      0 &      0 &      0 &      0 &      0 &      0 &      0 &      0 &      0 &      0 &      0 &      0 &      0 &      0 &      0 &      0 \\
      X_2^1    &      0 &      0 &      0 &      0 &      0 &      0 &     \n &      1 &      0 &     \n &      0 &      0 &      0 &      0 &      0 &      0 &      0 &      0 &      0 &      0 &      0 &      0 &      0 \\
      X_2^2    &      0 &      0 &      0 &      0 &      0 &      0 &      0 &      1 &     \n &      0 &      0 &      0 &      0 &      0 &      0 &      0 &      0 &      0 &      0 &      0 &      0 &      0 &      0 \\
      X_3^1    &      0 &      0 &      0 &      0 &      0 &      0 &      0 &      0 &      0 &      0 &     \n &      1 &      0 &     \n &      0 &      0 &      0 &      0 &      0 &      0 &      0 &      0 &      0 \\
      X_3^2    &      0 &      0 &      0 &      0 &      0 &      0 &      0 &      0 &      0 &      0 &      0 &      1 &     \n &      0 &      0 &      0 &      0 &      0 &      0 &      0 &      0 &      0 &      0 \\
      X_4^1    &      0 &      0 &      0 &      0 &      0 &      0 &      0 &      0 &      0 &      0 &      1 &      0 &      0 &      0 &      1 &      0 &     \n &      0 &      0 &      0 &      0 &      0 &      0 \\
      X_4^2    &      0 &      0 &      0 &      0 &      0 &      0 &      0 &      0 &      0 &      0 &      0 &      0 &      0 &      0 &      1 &     \n &      0 &      0 &      0 &      0 &      0 &      0 &      0 \\
      X_5^1    &      0 &      0 &      0 &      0 &      0 &      0 &      1 &      0 &      0 &      0 &      0 &      0 &      0 &      0 &      0 &      0 &      0 &      1 &      0 &     \n &      0 &      0 &      0 \\
      X_5^2    &      0 &      0 &      0 &      0 &      0 &      0 &      0 &      0 &      0 &      0 &      0 &      0 &      0 &      0 &      0 &      0 &      0 &      1 &     \n &      0 &      0 &      0 &      0 \\
      X_6^1    &      1 &      0 &      0 &      0 &      0 &      0 &      0 &      0 &      0 &      0 &      0 &      0 &      0 &      0 &      0 &      0 &      0 &      0 &      0 &      0 &     \n &      0 &      0 \\
      X_6^2    &      0 &      0 &      0 &      0 &      0 &      0 &      0 &      0 &      0 &      0 &      0 &      0 &      0 &      0 &      0 &      0 &      0 &      0 &      0 &      0 &      0 &      0 &      0 \\
      X_7^1    &      0 &      1 &      0 &      0 &      0 &      0 &      0 &      0 &      0 &      0 &      0 &      0 &      0 &      0 &      0 &      0 &      0 &     \n &      0 &      0 &      0 &     \n &      0 \\
      X_7^2    &      0 &      0 &      0 &      0 &      0 &      0 &      0 &      0 &      0 &      0 &      0 &      1 &      0 &      0 &      0 &      0 &      0 &     \n &      0 &      0 &      0 &      0 &      0 \\
      X_8^1    &      0 &      0 &      1 &      0 &      0 &      0 &      0 &      0 &      0 &      0 &      0 &      0 &      0 &      0 &      0 &      0 &      0 &      0 &     \n &      0 &      0 &      0 &     \n \\
      X_8^2    &      0 &      0 &      0 &      0 &      0 &      0 &      0 &      0 &      0 &      0 &      0 &      0 &      1 &      0 &      0 &      0 &      0 &      0 &     \n &      0 &      0 &      0 &      0 \\
      X_9^1    &      0 &      0 &      0 &      1 &      0 &      0 &      0 &      0 &      0 &      0 &      0 &      0 &      0 &      0 &     \n &      0 &      0 &      0 &      0 &      0 &      0 &      0 &      1 \\
      X_9^2    &      0 &      0 &      0 &      0 &      0 &      0 &      0 &      1 &      0 &      0 &      0 &      0 &      0 &      0 &     \n &      0 &      0 &      0 &      0 &      0 &      0 &      0 &      0 \\
      X_{10}^1 &      0 &      0 &      0 &      0 &      1 &      0 &      0 &      0 &      0 &      0 &      0 &      0 &      0 &      0 &      0 &     \n &      0 &      0 &      0 &      0 &      0 &      1 &      0 \\
      X_{10}^2 &      0 &      0 &      0 &      0 &      0 &      0 &      0 &      0 &      1 &      0 &      0 &      0 &      0 &      0 &      0 &     \n &      0 &      0 &      0 &      0 &      0 &      0 &      0 \\
      X_{11}^1 &      0 &      0 &      0 &      0 &      0 &      1 &      0 &      0 &      0 &      0 &      0 &      0 &      0 &     \n &      0 &      0 &     \n &      0 &      0 &      0 &      1 &      0 &      0 \\
      X_{11}^2 &      0 &      0 &      0 &      0 &      0 &      0 &      0 &      0 &      0 &      1 &      0 &      0 &      0 &     \n &      0 &      0 &     \n &      0 &      0 &      1 &      0 &      0 &      0 \\
    }
  \]
  \[
    \setlength{\arraycolsep}{2pt}
    \renewcommand{\arraystretch}{1.0}
    \tiny
    d_1(G_3) =
    \kbordermatrix{
          & X_1^1 & X_1^2 & X_2^1 & X_2^2 & X_3^1 & X_3^2 & X_4^1 & X_4^2 & X_5^1 & X_5^2 & X_6^1 & X_6^2 & X_7^1 & X_7^2 & X_8^1 & X_8^2 & X_9^1 & X_9^2 & X_{10}^1 & X_{10}^2 & X_{11}^1 & X_{11}^2 \\
      Y_1 &     1 &     0 &     0 &     0 &    \n &     0 &    \n &     0 &     0 &     1 &     1 &    \n &     0 &     1 &    \n &     0 &    \n &     0 &        0 &        0 &        1 &        0 \\
      Y_2 &     0 &     0 &     1 &     0 &    \n &     0 &    \n &     0 &     1 &     0 &     0 &     0 &     0 &     1 &     0 &     0 &     0 &    \n &        0 &        0 &        0 &        1 \\
      Y_3 &     0 &     1 &     0 &     0 &     0 &     0 &     0 &    \n &     0 &     1 &     0 &     0 &     1 &     0 &    \n &     0 &    \n &     0 &        1 &        0 &        0 &        0 \\
      Y_4 &     0 &     0 &     0 &     1 &     0 &     0 &     0 &    \n &     0 &     0 &     0 &     1 &     0 &     0 &     0 &     0 &     0 &    \n &        0 &        1 &        0 &        0 \\
      Y_5 &     0 &     0 &     0 &     0 &     0 &     1 &     0 &     0 &     0 &    \n &     0 &     1 &     0 &    \n &     0 &     1 &     0 &     0 &        0 &        0 &        0 &        0 \\
    }
  \]

  It is easily verified that $d_1\circ d_2 =0$, $\dim \ker d_1(G_3) = 17$, and $\mathrm{rank}\, d_2(G_3) = 16$, so the multiplicity of the Specht module $\S{2^3}$ in $H_{1}(G_3;\bbC)$ is one.
  We note that a generator of the homology $H_{1}(G_3;\bbC)\big|_{\S{2^3}}$ is $X_6^2 -X_8^2-X_{10}^2 +X_{11}^1$.
\end{proof}

\noindent{\bf Lemma~\ref{lemma.G4}.} The multiplicity of the $\fS_6$-module $\S{2^3}$ in $H_{1}(G_4;\bbC)$ is $0$.
\begin{proof}
  The chain modules are $C_0(G_4)\big|_{\S{2^3}}\cong \S{2^3}^{\oplus 5}$, $C_1(G_4)\big|_{\S{2^3}}\cong \S{2^3}^{\oplus 23}$, and $C_2(G_4)\big|_{\S{2^3}}\cong \S{2^3}^{\oplus 22}$.
  The differentials $d_2(G_4)$ and $d_1(G_4)$ are described by the following matrices (here, $a=10$ and $b=11$):
  \[
    \setlength{\arraycolsep}{.75pt}
    \renewcommand{\arraystretch}{1.0}
    \tiny
    d_2(G_4) =
    \kbordermatrix{
               & W_{16} & W_{17} & W_{18} & W_{19} & W_{1a} & W_{1b} & W_{23} & W_{25} & W_{27} & W_{28} & W_{2b} & W_{39} & W_{3a} & W_{3b} & W_{47} & W_{48} & W_{4b} & W_{56} & W_{58} & W_{5a} & W_{6b} & W_{7a} & W_{89} \\
      X_1^1    &     \n &      0 &      1 &      1 &      0 &     \n &      0 &      0 &      0 &      0 &      0 &      0 &      0 &      0 &      0 &      0 &      0 &      0 &      0 &      0 &      0 &      0 &      0 \\
      X_1^2    &      0 &     \n &      1 &      1 &     \n &      0 &      0 &      0 &      0 &      0 &      0 &      0 &      0 &      0 &      0 &      0 &      0 &      0 &      0 &      0 &      0 &      0 &      0 \\
      X_2^1    &      0 &      0 &      0 &      0 &      0 &      0 &     \n &      0 &      1 &      0 &     \n &      0 &      0 &      0 &      0 &      0 &      0 &      0 &      0 &      0 &      0 &      0 &      0 \\
      X_2^2    &      0 &      0 &      0 &      0 &      0 &      0 &      0 &     \n &      1 &     \n &      0 &      0 &      0 &      0 &      0 &      0 &      0 &      0 &      0 &      0 &      0 &      0 &      0 \\
      X_3^1    &      0 &      0 &      0 &      0 &      0 &      0 &      1 &      0 &      0 &      0 &      0 &      1 &      0 &     \n &      0 &      0 &      0 &      0 &      0 &      0 &      0 &      0 &      0 \\
      X_3^2    &      0 &      0 &      0 &      0 &      0 &      0 &      0 &      0 &      0 &      0 &      0 &      1 &     \n &      0 &      0 &      0 &      0 &      0 &      0 &      0 &      0 &      0 &      0 \\
      X_4^1    &      0 &      0 &      0 &      0 &      0 &      0 &      0 &      0 &      0 &      0 &      0 &      0 &      0 &      0 &      1 &      0 &     \n &      0 &      0 &      0 &      0 &      0 &      0 \\
      X_4^2    &      0 &      0 &      0 &      0 &      0 &      0 &      0 &      0 &      0 &      0 &      0 &      0 &      0 &      0 &      1 &     \n &      0 &      0 &      0 &      0 &      0 &      0 &      0 \\
      X_5^1    &      0 &      0 &      0 &      0 &      0 &      0 &      0 &      1 &      0 &      0 &      0 &      0 &      0 &      0 &      0 &      0 &      0 &      1 &     \n &      0 &      0 &      0 &      0 \\
      X_5^2    &      0 &      0 &      0 &      0 &      0 &      0 &      0 &      0 &      0 &      0 &      0 &      0 &      0 &      0 &      0 &      0 &      0 &      1 &      0 &     \n &      0 &      0 &      0 \\
      X_6^1    &      1 &      0 &      0 &      0 &      0 &      0 &      0 &      0 &      0 &      0 &      0 &      0 &      0 &      0 &      0 &      0 &      0 &     \n &      0 &      0 &     \n &      0 &      0 \\
      X_6^2    &      0 &      0 &      0 &      0 &      0 &      0 &      0 &      0 &      0 &      0 &      0 &      0 &      0 &      0 &      0 &      0 &      0 &     \n &      0 &      0 &      0 &      0 &      0 \\
      X_7^1    &      0 &      1 &      0 &      0 &      0 &      0 &      0 &      0 &      0 &      0 &      0 &      0 &      0 &      0 &     \n &      0 &      0 &      0 &      0 &      0 &      0 &     \n &      0 \\
      X_7^2    &      0 &      0 &      0 &      0 &      0 &      0 &      0 &      0 &      1 &      0 &      0 &      0 &      0 &      0 &     \n &      0 &      0 &      0 &      0 &      0 &      0 &      0 &      0 \\
      X_8^1    &      0 &      0 &      1 &      0 &      0 &      0 &      0 &      0 &      0 &      0 &      0 &      0 &      0 &      0 &      0 &     \n &      0 &      0 &      0 &      0 &      0 &      0 &     \n \\
      X_8^2    &      0 &      0 &      0 &      0 &      0 &      0 &      0 &      0 &      0 &      1 &      0 &      0 &      0 &      0 &      0 &     \n &      0 &      0 &      1 &      0 &      0 &      0 &      0 \\
      X_9^1    &      0 &      0 &      0 &      1 &      0 &      0 &      0 &      0 &      0 &      0 &      0 &     \n &      0 &      0 &      0 &      0 &      0 &      0 &      0 &      0 &      0 &      0 &      1 \\
      X_9^2    &      0 &      0 &      0 &      0 &      0 &      0 &      0 &      0 &      0 &      0 &      0 &     \n &      0 &      0 &      0 &      0 &      0 &      0 &      0 &      0 &      0 &      0 &      0 \\
      X_{10}^1 &      0 &      0 &      0 &      0 &      1 &      0 &      0 &      0 &      0 &      0 &      0 &      0 &     \n &      0 &      0 &      0 &      0 &      0 &      0 &      0 &      0 &      1 &      0 \\
      X_{10}^2 &      0 &      0 &      0 &      0 &      0 &      0 &      0 &      0 &      0 &      0 &      0 &      0 &     \n &      0 &      0 &      0 &      0 &      0 &      0 &      1 &      0 &      0 &      0 \\
      X_{11}^1 &      0 &      0 &      0 &      0 &      0 &      1 &      0 &      0 &      0 &      0 &     \n &      0 &      0 &     \n &      0 &      0 &      0 &      0 &      0 &      0 &      1 &      0 &      0 \\
      X_{11}^2 &      0 &      0 &      0 &      0 &      0 &      0 &      0 &      0 &      0 &      0 &     \n &      0 &      0 &     \n &      0 &      0 &      1 &      0 &      0 &      0 &      0 &      0 &      0 \\
    }
  \]
  \[
    \setlength{\arraycolsep}{2pt}
    \renewcommand{\arraystretch}{1.0}
    \tiny
    d_1(G_4) =
    \kbordermatrix{
          & X_1^1 & X_1^2 & X_2^1 & X_2^2 & X_3^1 & X_3^2 & X_4^1 & X_4^2 & X_5^1 & X_5^2 & X_6^1 & X_6^2 & X_7^1 & X_7^2 & X_8^1 & X_8^2 & X_9^1 & X_9^2 & X_{10}^1 & X_{10}^2 & X_{11}^1 & X_{11}^2 \\
      Y_1 &     1 &     0 &    \n &     0 &    \n &     0 &     0 &     1 &     0 &     0 &     1 &    \n &     0 &     1 &    \n &     0 &    \n &     0 &        0 &        0 &        1 &        0 \\
      Y_2 &     0 &     0 &    \n &     0 &    \n &     0 &     1 &     0 &     0 &     0 &     0 &     0 &     0 &     1 &     0 &     0 &     0 &    \n &        0 &        0 &        0 &        1 \\
      Y_3 &     0 &     1 &     0 &     0 &     0 &    \n &     0 &     1 &     0 &     0 &     0 &     0 &     1 &     0 &    \n &     0 &    \n &     0 &        1 &        0 &        0 &        0 \\
      Y_4 &     0 &     0 &     0 &     0 &     0 &    \n &     0 &     0 &     0 &     1 &     0 &     1 &     0 &     0 &     0 &     0 &     0 &    \n &        0 &        1 &        0 &        0 \\
      Y_5 &     0 &     0 &     0 &     1 &     0 &     0 &     0 &    \n &     1 &     0 &     0 &     1 &     0 &    \n &     0 &     1 &     0 &     0 &        0 &        0 &        0 &        0 \\
    }
  \]

  It is easily verified that $d_1\circ d_2 =0$, $\dim \ker d_1(G_4) = 17$, and $\mathrm{rank}\, d_2(G_4) = 17$, so the multiplicity of the Specht module $\S{2^3}$ in $H_{1}(G_4;\bbC)$ is zero.
\end{proof}

\noindent{\bf Lemma~\ref{lemma.G5}.} The multiplicity of the $\fS_6$-module $\S{2^2,1^2}$ in $H_{1}(G_5,\bbC)$ is $1$.
\begin{proof}
  We have the chain modules $C_0(G_5)\big|_{\S{2^2,1^2}}\cong \S{2^2,1^2}^{\oplus 9}$, $C_1(G_5)\big|_{\S{2^2,1^2}}\cong \S{2^2,1^2}^{\oplus 33}$, and $C_2(G_5)\big|_{\S{2^2,1^2}}\cong \S{2^2,1^2}^{\oplus 24}$.
  The differentials $d_2(G_5)$ and $d_1(G_5)$ are described by the following matrices (here, $a=10$ and $b=11$):
  \[
    \setlength{\arraycolsep}{.75pt}
    \renewcommand{\arraystretch}{1.0}
    \tiny
    d_2(G_5) =\!
    \kbordermatrix{
               & W_{17} & W_{18} & W_{19} & W_{1a} & W_{1b} & W_{24} & W_{25} & W_{26} & W_{27} & W_{28} & W_{2a} & W_{34} & W_{35} & W_{37} & W_{39} & W_{49} & W_{4a} & W_{4b} & W_{58} & W_{5b} & W_{67} & W_{69} & W_{7b} & W_{89} \\
      X_1^1    &     \n &      0 &      1 &     \n &      0 &      0 &      0 &      0 &      0 &      0 &      0 &      0 &      0 &      0 &      0 &      0 &      0 &      0 &      0 &      0 &      0 &      0 &      0 &      0 \\
      X_1^2    &      0 &      0 &      1 &      0 &     \n &      0 &      0 &      0 &      0 &      0 &      0 &      0 &      0 &      0 &      0 &      0 &      0 &      0 &      0 &      0 &      0 &      0 &      0 &      0 \\
      X_1^3    &      0 &     \n &      0 &      1 &     \n &      0 &      0 &      0 &      0 &      0 &      0 &      0 &      0 &      0 &      0 &      0 &      0 &      0 &      0 &      0 &      0 &      0 &      0 &      0 \\
      X_2^1    &      0 &      0 &      0 &      0 &      0 &     \n &      0 &      0 &      1 &     \n &      0 &      0 &      0 &      0 &      0 &      0 &      0 &      0 &      0 &      0 &      0 &      0 &      0 &      0 \\
      X_2^2    &      0 &      0 &      0 &      0 &      0 &      0 &     \n &      0 &      1 &      0 &     \n &      0 &      0 &      0 &      0 &      0 &      0 &      0 &      0 &      0 &      0 &      0 &      0 &      0 \\
      X_2^3    &      0 &      0 &      0 &      0 &      0 &      0 &      0 &     \n &      0 &      1 &     \n &      0 &      0 &      0 &      0 &      0 &      0 &      0 &      0 &      0 &      0 &      0 &      0 &      0 \\
      X_3^1    &      0 &      0 &      0 &      0 &      0 &      0 &      0 &      0 &      0 &      0 &      0 &     \n &      0 &      1 &      0 &      0 &      0 &      0 &      0 &      0 &      0 &      0 &      0 &      0 \\
      X_3^2    &      0 &      0 &      0 &      0 &      0 &      0 &      0 &      0 &      0 &      0 &      0 &      0 &     \n &      1 &     \n &      0 &      0 &      0 &      0 &      0 &      0 &      0 &      0 &      0 \\
      X_3^3    &      0 &      0 &      0 &      0 &      0 &      0 &      0 &      0 &      0 &      0 &      0 &      0 &      0 &      0 &     \n &      0 &      0 &      0 &      0 &      0 &      0 &      0 &      0 &      0 \\
      X_4^1    &      0 &      0 &      0 &      0 &      0 &      0 &      0 &      0 &      0 &      0 &      0 &      0 &      0 &      0 &      0 &      1 &     \n &      0 &      0 &      0 &      0 &      0 &      0 &      0 \\
      X_4^2    &      0 &      0 &      0 &      0 &      0 &      1 &      0 &      0 &      0 &      0 &      0 &      0 &      0 &      0 &      0 &      1 &      0 &     \n &      0 &      0 &      0 &      0 &      0 &      0 \\
      X_4^3    &      0 &      0 &      0 &      0 &      0 &      0 &      0 &      0 &      0 &      0 &      0 &      1 &      0 &      0 &      0 &      0 &      1 &     \n &      0 &      0 &      0 &      0 &      0 &      0 \\
      X_5^1    &      0 &      0 &      0 &      0 &      0 &      0 &      0 &      0 &      0 &      0 &      0 &      0 &      0 &      0 &      0 &      0 &      0 &      0 &     \n &      0 &      0 &      0 &      0 &      0 \\
      X_5^2    &      0 &      0 &      0 &      0 &      0 &      0 &      1 &      0 &      0 &      0 &      0 &      0 &      0 &      0 &      0 &      0 &      0 &      0 &      0 &     \n &      0 &      0 &      0 &      0 \\
      X_5^3    &      0 &      0 &      0 &      0 &      0 &      0 &      0 &      0 &      0 &      0 &      0 &      0 &      1 &      0 &      0 &      0 &      0 &      0 &      1 &     \n &      0 &      0 &      0 &      0 \\
      X_6^1    &      0 &      0 &      0 &      0 &      0 &      0 &      0 &      0 &      0 &      0 &      0 &      0 &      0 &      0 &      0 &      0 &      0 &      0 &      0 &      0 &      1 &      0 &      0 &      0 \\
      X_6^2    &      0 &      0 &      0 &      0 &      0 &      0 &      0 &      0 &      0 &      0 &      0 &      0 &      0 &      0 &      0 &      0 &      0 &      0 &      0 &      0 &      1 &     \n &      0 &      0 \\
      X_6^3    &      0 &      0 &      0 &      0 &      0 &      0 &      0 &      1 &      0 &      0 &      0 &      0 &      0 &      0 &      0 &      0 &      0 &      0 &      0 &      0 &      0 &     \n &      0 &      0 \\
      X_7^1    &      1 &      0 &      0 &      0 &      0 &      0 &      0 &      0 &      0 &      0 &      0 &      0 &      0 &      0 &      0 &      0 &      0 &      0 &      0 &      0 &      1 &      0 &      0 &      0 \\
      X_7^2    &      0 &      0 &      0 &      0 &      0 &      0 &      0 &      0 &      1 &      0 &      0 &      0 &      0 &      0 &      0 &      0 &      0 &      0 &      0 &      0 &      0 &      0 &     \n &      0 \\
      X_7^3    &      0 &      0 &      0 &      0 &      0 &      0 &      0 &      0 &      0 &      0 &      0 &      0 &      0 &      1 &      0 &      0 &      0 &      0 &      0 &      0 &     \n &      0 &     \n &      0 \\
      X_8^1    &      0 &      1 &      0 &      0 &      0 &      0 &      0 &      0 &      0 &      0 &      0 &      0 &      0 &      0 &      0 &      0 &      0 &      0 &     \n &      0 &      0 &      0 &      0 &      0 \\
      X_8^2    &      0 &      0 &      0 &      0 &      0 &      0 &      0 &      0 &      0 &      0 &      0 &      0 &      0 &      0 &      0 &      0 &      0 &      0 &     \n &      0 &      0 &      0 &      0 &     \n \\
      X_8^3    &      0 &      0 &      0 &      0 &      0 &      0 &      0 &      0 &      0 &      1 &      0 &      0 &      0 &      0 &      0 &      0 &      0 &      0 &      0 &      0 &      0 &      0 &      0 &     \n \\
      X_9^1    &      0 &      0 &      1 &      0 &      0 &      0 &      0 &      0 &      0 &      0 &      0 &      0 &      0 &      0 &      0 &     \n &      0 &      0 &      0 &      0 &      0 &      1 &      0 &      0 \\
      X_9^2    &      0 &      0 &      0 &      0 &      0 &      0 &      0 &      0 &      0 &      0 &      0 &      0 &      0 &      0 &      0 &     \n &      0 &      0 &      0 &      0 &      0 &      0 &      0 &      1 \\
      X_9^3    &      0 &      0 &      0 &      0 &      0 &      0 &      0 &      0 &      0 &      0 &      0 &      0 &      0 &      0 &      1 &      0 &      0 &      0 &      0 &      0 &      0 &     \n &      0 &      1 \\
      X_{10}^1 &      0 &      0 &      0 &      1 &      0 &      0 &      0 &      0 &      0 &      0 &      0 &      0 &      0 &      0 &      0 &      0 &     \n &      0 &      0 &      0 &      0 &      0 &      0 &      0 \\
      X_{10}^2 &      0 &      0 &      0 &      0 &      0 &      0 &      0 &      0 &      0 &      0 &      0 &      0 &      0 &      0 &      0 &      0 &     \n &      0 &      0 &      0 &      0 &      0 &      0 &      0 \\
      X_{10}^3 &      0 &      0 &      0 &      0 &      0 &      0 &      0 &      0 &      0 &      0 &      1 &      0 &      0 &      0 &      0 &      0 &      0 &      0 &      0 &      0 &      0 &      0 &      0 &      0 \\
      X_{11}^1 &      0 &      0 &      0 &      0 &      1 &      0 &      0 &      0 &      0 &      0 &      0 &      0 &      0 &      0 &      0 &      0 &      0 &     \n &      0 &      1 &      0 &      0 &      0 &      0 \\
      X_{11}^2 &      0 &      0 &      0 &      0 &      0 &      0 &      0 &      0 &      0 &      0 &      0 &      0 &      0 &      0 &      0 &      0 &      0 &     \n &      0 &      0 &      0 &      0 &      1 &      0 \\
      X_{11}^3 &      0 &      0 &      0 &      0 &      0 &      0 &      0 &      0 &      0 &      0 &      0 &      0 &      0 &      0 &      0 &      0 &      0 &      0 &      0 &     \n &      0 &      0 &      1 &      0
    }
  \]
  \[
    \setlength{\arraycolsep}{1pt}
    \renewcommand{\arraystretch}{1.0}
    \tiny
    d_1\!\!=\!\!
    \kbordermatrix{
          & X_1^1 & X_1^2 & X_1^3 & X_2^1 & X_2^2 & X_2^3 & X_3^1 & X_3^2 & X_3^3 & X_4^1 & X_4^2 & X_4^3 & X_5^1 & X_5^2 & X_5^3 & X_6^1 & X_6^2 & X_6^3 & X_7^1 & X_7^2 & X_7^3 & X_8^1 & X_8^2 & X_8^3 & X_9^1 & X_9^2 & X_9^3 & X_{10}^1 & X_{10}^2 & X_{10}^3 & X_{11}^1 & X_{11}^2 & X_{11}^3 \\
      Y_1 &     1 &     0 &     0 &     0 &     1 &     0 &     0 &    \n &     0 &    \n &     0 &     0 &     0 &     1 &    \n &     0 &     0 &     0 &     1 &    \n &     1 &     0 &    \n &     0 &    \n &     0 &    \n &        1 &        0 &        1 &        0 &        0 &        0 \\
      Y_2 &     0 &     0 &     0 &     0 &     0 &     0 &     0 &     0 &     0 &    \n &     0 &     0 &     1 &     0 &     0 &     0 &     0 &     0 &     0 &     0 &     0 &     0 &    \n &     0 &     0 &    \n &     0 &        0 &        1 &        0 &        0 &        0 &        0 \\
      Y_3 &     0 &     1 &     0 &    \n &     1 &     0 &     0 &     0 &    \n &     0 &    \n &     0 &     0 &     1 &     0 &     0 &     0 &     0 &     0 &     0 &     0 &     0 &     0 &    \n &    \n &     0 &    \n &        0 &        0 &        1 &        1 &        0 &        0 \\
      Y_4 &     0 &     0 &     0 &    \n &     0 &     0 &     0 &     0 &     0 &     0 &    \n &     0 &     0 &     0 &     0 &     0 &     0 &     0 &     0 &     1 &     0 &     0 &     0 &    \n &     0 &    \n &     0 &        0 &        0 &        0 &        0 &        1 &        0 \\
      Y_5 &     0 &     0 &     0 &     0 &    \n &     0 &     0 &     0 &     0 &     0 &     0 &     0 &     0 &    \n &     0 &     0 &     0 &     0 &     0 &     1 &     0 &     0 &     0 &     0 &     0 &     0 &     0 &        0 &        0 &       \n &        0 &        0 &        1 \\
      Y_6 &     0 &     0 &     1 &     0 &     0 &     0 &    \n &     1 &    \n &     0 &     0 &    \n &     0 &     0 &     1 &     0 &     0 &     0 &     0 &     0 &     0 &     1 &     0 &     0 &     0 &     0 &     0 &       \n &        0 &        0 &        1 &        0 &        0 \\
      Y_7 &     0 &     0 &     0 &     0 &     0 &     0 &    \n &     0 &     0 &     0 &     0 &    \n &     0 &     0 &     0 &     1 &     0 &     0 &     0 &     0 &     1 &     0 &     0 &     0 &     0 &     0 &     0 &        0 &       \n &        0 &        0 &        1 &        0 \\
      Y_8 &     0 &     0 &     0 &     0 &     0 &     0 &     0 &    \n &     0 &     0 &     0 &     0 &     0 &     0 &    \n &     0 &     1 &     0 &     0 &     0 &     1 &     0 &    \n &     0 &     0 &     0 &    \n &        0 &        0 &        0 &        0 &        0 &        1 \\
      Y_9 &     0 &     0 &     0 &     0 &     0 &     1 &     0 &     0 &    \n &     0 &     0 &     0 &     0 &     0 &     0 &     0 &     0 &     1 &     0 &     0 &     0 &     0 &     0 &    \n &     0 &     0 &    \n &        0 &        0 &        1 &        0 &        0 &        0
    }
  \]

  It is easily verified that $d_1\circ d_2 =0$, $\dim \ker d_1(G_5) = 24$, and $\mathrm{rank}\, d_2(G_5) = 23$, so the multiplicity of the Specht module $\S{2^2,1^2}$ in $H_{1}(G_5;\bbC)$ is one.
  We note that a generator of the homology $H_1(G_5;\bbC)|_{\S{2^2,1^2}}$ is $X_9^3+X_{10}^3+X_{11}^3$.
\end{proof}

\noindent{\bf Lemma~\ref{lemma.G6}.} The multiplicity of the $\fS_6$-module $\S{2^2,1^2}$ in $H_{1}(G_6,\bbC)$ is $0$.
\begin{proof}
  We have the chain modules $C_0(G_6)\big|_{\S{2^2,1^2}}\cong \S{2^2,1^2}^{\oplus 9}$, $C_1(G_6)\big|_{\S{2^2,1^2}}\cong \S{2^2,1^2}^{\oplus 33}$, and $C_2(G_6)\big|_{\S{2^2,1^2}}\cong \S{2^2,1^2}^{\oplus 24}$.
  The differentials $d_2(G_6)$ and $d_1(G_6)$ are described by the following matrices (here, $a=10$ and $b=11$):
  \[
    \setlength{\arraycolsep}{.75pt}
    \renewcommand{\arraystretch}{1.0}
    d_2(G_6) =\!
    \tiny
    \kbordermatrix{
               & W_{18} & W_{19} & W_{1a} & W_{1b} & W_{26} & W_{27} & W_{2a} & W_{2b} & W_{35} & W_{37} & W_{39} & W_{3b} & W_{45} & W_{46} & W_{47} & W_{48} & W_{49} & W_{4a} & W_{5a} & W_{5b} & W_{69} & W_{6b} & W_{78} & W_{8b} \\
      X_1^1    &     \n &      0 &      1 &      0 &      0 &      0 &      0 &      0 &      0 &      0 &      0 &      0 &      0 &      0 &      0 &      0 &      0 &      0 &      0 &      0 &      0 &      0 &      0 &      0 \\
      X_1^2    &      0 &     \n &      1 &     \n &      0 &      0 &      0 &      0 &      0 &      0 &      0 &      0 &      0 &      0 &      0 &      0 &      0 &      0 &      0 &      0 &      0 &      0 &      0 &      0 \\
      X_1^3    &      0 &      0 &      0 &     \n &      0 &      0 &      0 &      0 &      0 &      0 &      0 &      0 &      0 &      0 &      0 &      0 &      0 &      0 &      0 &      0 &      0 &      0 &      0 &      0 \\
      X_2^1    &      0 &      0 &      0 &      0 &     \n &      0 &      1 &      0 &      0 &      0 &      0 &      0 &      0 &      0 &      0 &      0 &      0 &      0 &      0 &      0 &      0 &      0 &      0 &      0 \\
      X_2^1    &      0 &      0 &      0 &      0 &      0 &     \n &      1 &     \n &      0 &      0 &      0 &      0 &      0 &      0 &      0 &      0 &      0 &      0 &      0 &      0 &      0 &      0 &      0 &      0 \\
      X_2^3    &      0 &      0 &      0 &      0 &      0 &      0 &      0 &     \n &      0 &      0 &      0 &      0 &      0 &      0 &      0 &      0 &      0 &      0 &      0 &      0 &      0 &      0 &      0 &      0 \\
      X_3^1    &      0 &      0 &      0 &      0 &      0 &      0 &      0 &      0 &     \n &      0 &      1 &      0 &      0 &      0 &      0 &      0 &      0 &      0 &      0 &      0 &      0 &      0 &      0 &      0 \\
      X_3^2    &      0 &      0 &      0 &      0 &      0 &      0 &      0 &      0 &      0 &     \n &      1 &     \n &      0 &      0 &      0 &      0 &      0 &      0 &      0 &      0 &      0 &      0 &      0 &      0 \\
      X_3^3    &      0 &      0 &      0 &      0 &      0 &      0 &      0 &      0 &      0 &      0 &      0 &     \n &      0 &      0 &      0 &      0 &      0 &      0 &      0 &      0 &      0 &      0 &      0 &      0 \\
      X_4^1    &      0 &      0 &      0 &      0 &      0 &      0 &      0 &      0 &      0 &      0 &      0 &      0 &     \n &      0 &      0 &      1 &     \n &      0 &      0 &      0 &      0 &      0 &      0 &      0 \\
      X_4^2    &      0 &      0 &      0 &      0 &      0 &      0 &      0 &      0 &      0 &      0 &      0 &      0 &      0 &     \n &      0 &      1 &      0 &     \n &      0 &      0 &      0 &      0 &      0 &      0 \\
      X_4^3    &      0 &      0 &      0 &      0 &      0 &      0 &      0 &      0 &      0 &      0 &      0 &      0 &      0 &      0 &     \n &      0 &      1 &     \n &      0 &      0 &      0 &      0 &      0 &      0 \\
      X_5^1    &      0 &      0 &      0 &      0 &      0 &      0 &      0 &      0 &      1 &      0 &      0 &      0 &      0 &      0 &      0 &      0 &      0 &      0 &      1 &      0 &      0 &      0 &      0 &      0 \\
      X_5^2    &      0 &      0 &      0 &      0 &      0 &      0 &      0 &      0 &      0 &      0 &      0 &      0 &      0 &      0 &      0 &      0 &      0 &      0 &      1 &     \n &      0 &      0 &      0 &      0 \\
      X_5^3    &      0 &      0 &      0 &      0 &      0 &      0 &      0 &      0 &      0 &      0 &      0 &      0 &      1 &      0 &      0 &      0 &      0 &      0 &      0 &     \n &      0 &      0 &      0 &      0 \\
      X_6^1    &      0 &      0 &      0 &      0 &      1 &      0 &      0 &      0 &      0 &      0 &      0 &      0 &      0 &      0 &      0 &      0 &      0 &      0 &      0 &      0 &      1 &      0 &      0 &      0 \\
      X_6^2    &      0 &      0 &      0 &      0 &      0 &      0 &      0 &      0 &      0 &      0 &      0 &      0 &      0 &      0 &      0 &      0 &      0 &      0 &      0 &      0 &      1 &     \n &      0 &      0 \\
      X_6^3    &      0 &      0 &      0 &      0 &      0 &      0 &      0 &      0 &      0 &      0 &      0 &      0 &      0 &      1 &      0 &      0 &      0 &      0 &      0 &      0 &      0 &     \n &      0 &      0 \\
      X_7^1    &      0 &      0 &      0 &      0 &      0 &      1 &      0 &      0 &      0 &      0 &      0 &      0 &      0 &      0 &      0 &      0 &      0 &      0 &      0 &      0 &      0 &      0 &      1 &      0 \\
      X_7^2    &      0 &      0 &      0 &      0 &      0 &      0 &      0 &      0 &      0 &      1 &      0 &      0 &      0 &      0 &      0 &      0 &      0 &      0 &      0 &      0 &      0 &      0 &      1 &      0 \\
      X_7^3    &      0 &      0 &      0 &      0 &      0 &      0 &      0 &      0 &      0 &      0 &      0 &      0 &      0 &      0 &      1 &      0 &      0 &      0 &      0 &      0 &      0 &      0 &      0 &      0 \\
      X_8^1    &      1 &      0 &      0 &      0 &      0 &      0 &      0 &      0 &      0 &      0 &      0 &      0 &      0 &      0 &      0 &      0 &      0 &      0 &      0 &      0 &      0 &      0 &     \n &      0 \\
      X_8^2    &      0 &      0 &      0 &      0 &      0 &      0 &      0 &      0 &      0 &      0 &      0 &      0 &      0 &      0 &      0 &      0 &      0 &      0 &      0 &      0 &      0 &      0 &     \n &     \n \\
      X_8^3    &      0 &      0 &      0 &      0 &      0 &      0 &      0 &      0 &      0 &      0 &      0 &      0 &      0 &      0 &      0 &      1 &      0 &      0 &      0 &      0 &      0 &      0 &      0 &     \n \\
      X_9^1    &      0 &      1 &      0 &      0 &      0 &      0 &      0 &      0 &      0 &      0 &      0 &      0 &      0 &      0 &      0 &      0 &      0 &      0 &      0 &      0 &     \n &      0 &      0 &      0 \\
      X_9^2    &      0 &      0 &      0 &      0 &      0 &      0 &      0 &      0 &      0 &      0 &      1 &      0 &      0 &      0 &      0 &      0 &      0 &      0 &      0 &      0 &     \n &      0 &      0 &      0 \\
      X_9^3    &      0 &      0 &      0 &      0 &      0 &      0 &      0 &      0 &      0 &      0 &      0 &      0 &      0 &      0 &      0 &      0 &      1 &      0 &      0 &      0 &      0 &      0 &      0 &      0 \\
      X_{10}^1 &      0 &      0 &      1 &      0 &      0 &      0 &      0 &      0 &      0 &      0 &      0 &      0 &      0 &      0 &      0 &      0 &      0 &      0 &     \n &      0 &      0 &      0 &      0 &      0 \\
      X_{10}^2 &      0 &      0 &      0 &      0 &      0 &      0 &      1 &      0 &      0 &      0 &      0 &      0 &      0 &      0 &      0 &      0 &      0 &      0 &     \n &      0 &      0 &      0 &      0 &      0 \\
      X_{10}^3 &      0 &      0 &      0 &      0 &      0 &      0 &      0 &      0 &      0 &      0 &      0 &      0 &      0 &      0 &      0 &      0 &      0 &      1 &      0 &      0 &      0 &      0 &      0 &      0 \\
      X_{11}^1 &      0 &      0 &      0 &      1 &      0 &      0 &      0 &      0 &      0 &      0 &      0 &      0 &      0 &      0 &      0 &      0 &      0 &      0 &      0 &     \n &      0 &      1 &      0 &      0 \\
      X_{11}^2 &      0 &      0 &      0 &      0 &      0 &      0 &      0 &      1 &      0 &      0 &      0 &      0 &      0 &      0 &      0 &      0 &      0 &      0 &      0 &     \n &      0 &      0 &      0 &      1 \\
      X_{11}^3 &      0 &      0 &      0 &      0 &      0 &      0 &      0 &      0 &      0 &      0 &      0 &      1 &      0 &      0 &      0 &      0 &      0 &      0 &      0 &      0 &      0 &     \n &      0 &      1
    }
  \]
  \[
    \setlength{\arraycolsep}{1pt}
    \renewcommand{\arraystretch}{1.0}
    \tiny
    d_1\!\!=\!\!
    \kbordermatrix{
          & X_1^1 & X_1^2 & X_1^3 & X_2^1 & X_2^2 & X_2^3 & X_3^1 & X_3^2 & X_3^3 & X_4^1 & X_4^2 & X_4^3 & X_5^1 & X_5^2 & X_5^3 & X_6^1 & X_6^2 & X_6^3 & X_7^1 & X_7^2 & X_7^3 & X_8^1 & X_8^2 & X_8^3 & X_9^1 & X_9^2 & X_9^3 & X_{10}^1 & X_{10}^2 & X_{10}^3 & X_{11}^1 & X_{11}^2 & X_{11}^3 \\
      Y_1 &     1 &     0 &     0 &     0 &     0 &     0 &    \n &     0 &     0 &     0 &    \n &     0 &    \n &     0 &     0 &     0 &     1 &    \n &     0 &     0 &     0 &     1 &    \n &     1 &     0 &     1 &     0 &       \n &        0 &       \n &        0 &        0 &        0 \\
      Y_2 &     0 &     0 &     0 &     1 &     0 &     0 &    \n &     0 &     0 &     0 &     0 &     0 &    \n &     0 &     0 &     1 &     0 &     0 &     0 &     0 &     0 &     0 &     0 &     0 &     0 &     1 &     0 &        0 &       \n &        0 &        0 &        0 &        0 \\
      Y_3 &     0 &     1 &     0 &     0 &     0 &     0 &     0 &     0 &     0 &     0 &     0 &    \n &     0 &    \n &     0 &     0 &     1 &     0 &     0 &     0 &    \n &     0 &     0 &     0 &     1 &     0 &     1 &       \n &        0 &       \n &        1 &        0 &        0 \\
      Y_4 &     0 &     0 &     0 &     0 &     1 &     0 &     0 &     0 &     0 &     0 &     0 &     0 &     0 &    \n &     0 &     0 &     0 &     0 &     1 &     0 &     0 &     0 &     1 &     0 &     0 &     0 &     0 &        0 &       \n &        0 &        0 &        1 &        0 \\
      Y_5 &     0 &     0 &     0 &     0 &     0 &     0 &     0 &     1 &     0 &     0 &     0 &     0 &     0 &     0 &     0 &     0 &    \n &     0 &     0 &     1 &     0 &     0 &     1 &     0 &     0 &    \n &     0 &        0 &        0 &        0 &        0 &        0 &        1 \\
      Y_6 &     0 &     0 &     1 &     0 &     0 &     0 &     0 &     0 &     0 &    \n &     1 &    \n &     0 &     0 &    \n &     0 &     0 &     1 &     0 &     0 &    \n &     0 &     0 &     0 &     0 &     0 &     0 &        0 &        0 &        0 &        1 &        0 &        0 \\
      Y_7 &     0 &     0 &     0 &     0 &     0 &     1 &     0 &     0 &     0 &    \n &     0 &     0 &     0 &     0 &    \n &     0 &     0 &     0 &     0 &     0 &     0 &     0 &     0 &     1 &     0 &     0 &    \n &        0 &        0 &        0 &        0 &        1 &        0 \\
      Y_8 &     0 &     0 &     0 &     0 &     0 &     0 &     0 &     0 &     1 &     0 &    \n &     0 &     0 &     0 &     0 &     0 &     0 &    \n &     0 &     0 &     0 &     0 &     0 &     1 &     0 &     0 &     0 &        0 &        0 &       \n &        0 &        0 &        1 \\
      Y_9 &     0 &     0 &     0 &     0 &     0 &     0 &     0 &     0 &     0 &     0 &     0 &    \n &     0 &     0 &     0 &     0 &     0 &     0 &     0 &     0 &    \n &     0 &     0 &     0 &     0 &     0 &     1 &        0 &        0 &       \n &        0 &        0 &        0
    }
  \]

  It is easily verified that $d_1\circ d_2 =0$, $\dim \ker d_1(G_6) = 24$, and $\mathrm{rank}\, d_2(G_6) = 24$, so the multiplicity of the Specht module $\S{2^2,1^2}$ in $H_{1}(G_6;\bbC)$ is zero.
\end{proof}

\begin{remark}
  \label{rem.G6}
  Let $g = -W_{18}-W_{19}-W_{1a}+W_{26}+W_{27}+W_{2a}+W_{35}+W_{37}+W_{39}-W_{45}-W_{46}+W_{48}+2W_{49}+2W_{4a}-W_{5a}-W_{5b}-W_{69}-W_{6b}-W_{78}+W_{8b} \in C_2(G_6)$ and $h = X_9^2+X_9^3+X_{10}^2+X_{10}^3+X_{11}^2+X_{11}^3 \in C_1(G_6)$.
  We note that over $\bbZ$, $h\notin \im d_2$, $d_2(g)=2h$ and $d_1(h)=0$, so $h$ generates $\bbZ_2$-torsion in $H_{1}(G_6;\bbZ)$.
\end{remark}

\noindent{\bf Theorem~\ref{thm.K33}.} The chromatic symmetric homology $H_1(K_{3,3};\bbZ)$ contains $\bbZ_2$-torsion.
\begin{proof}
  We have $C_0(K_{3,3})\big|_{\S{2^2,1^2}} \cong \S{2^2,1^2}^{\oplus 9}$, $C_1(K_{3,3})\big|_{\S{2^2,1^2}} \cong \S{2^2,1^2}^{\oplus 27}$, and $C_2(K_{3,3})\big|_{\S{2^2,1^2}} \cong \S{2^2,1^2}^{\oplus 18}$.
  The differentials $d_2(K_{3,3})$ and $d_1(K_{3,3})$ are described by the following matrices:
  \[
    \setlength{\arraycolsep}{1pt}
    \renewcommand{\arraystretch}{1.0}
    \tiny
    d_2(K_{3,3}) =
    \kbordermatrix{
            & W_{16} & W_{17} & W_{18} & W_{19} & W_{24} & W_{25} & W_{27} & W_{29} & W_{34} & W_{35} & W_{36} & W_{38} & W_{48} & W_{49} & W_{56} & W_{57} & W_{69} & W_{78} \\
      X_1^1 &     \n &      0 &      1 &      0 &      0 &      0 &      0 &      0 &      0 &      0 &      0 &      0 &      0 &      0 &      0 &      0 &      0 &      0 \\
      X_1^2 &      0 &      0 &      1 &     \n &      0 &      0 &      0 &      0 &      0 &      0 &      0 &      0 &      0 &      0 &      0 &      0 &      0 &      0 \\
      X_1^3 &      0 &     \n &      0 &     \n &      0 &      0 &      0 &      0 &      0 &      0 &      0 &      0 &      0 &      0 &      0 &      0 &      0 &      0 \\
      X_2^1 &      0 &      0 &      0 &      0 &     \n &      0 &     \n &      0 &      0 &      0 &      0 &      0 &      0 &      0 &      0 &      0 &      0 &      0 \\
      X_2^2 &      0 &      0 &      0 &      0 &      0 &     \n &      0 &     \n &      0 &      0 &      0 &      0 &      0 &      0 &      0 &      0 &      0 &      0 \\
      X_2^3 &      0 &      0 &      0 &      0 &      0 &      0 &      1 &     \n &      0 &      0 &      0 &      0 &      0 &      0 &      0 &      0 &      0 &      0 \\
      X_3^1 &      0 &      0 &      0 &      0 &      0 &      0 &      0 &      0 &     \n &      0 &      1 &      0 &      0 &      0 &      0 &      0 &      0 &      0 \\
      X_3^2 &      0 &      0 &      0 &      0 &      0 &      0 &      0 &      0 &      0 &      0 &      1 &     \n &      0 &      0 &      0 &      0 &      0 &      0 \\
      X_3^3 &      0 &      0 &      0 &      0 &      0 &      0 &      0 &      0 &      0 &     \n &      0 &     \n &      0 &      0 &      0 &      0 &      0 &      0 \\
      X_4^1 &      0 &      0 &      0 &      0 &      1 &      0 &      0 &      0 &      0 &      0 &      0 &      0 &      1 &      0 &      0 &      0 &      0 &      0 \\
      X_4^2 &      0 &      0 &      0 &      0 &      0 &      0 &      0 &      0 &      0 &      0 &      0 &      0 &      1 &     \n &      0 &      0 &      0 &      0 \\
      X_4^3 &      0 &      0 &      0 &      0 &      0 &      0 &      0 &      0 &      1 &      0 &      0 &      0 &      0 &     \n &      0 &      0 &      0 &      0 \\
      X_5^1 &      0 &      0 &      0 &      0 &      0 &      0 &      0 &      0 &      0 &      0 &      0 &      0 &      0 &      0 &      1 &     \n &      0 &      0 \\
      X_5^2 &      0 &      0 &      0 &      0 &      0 &      1 &      0 &      0 &      0 &      0 &      0 &      0 &      0 &      0 &      1 &      0 &      0 &      0 \\
      X_5^3 &      0 &      0 &      0 &      0 &      0 &      0 &      0 &      0 &      0 &      1 &      0 &      0 &      0 &      0 &      0 &      1 &      0 &      0 \\
      X_6^1 &      1 &      0 &      0 &      0 &      0 &      0 &      0 &      0 &      0 &      0 &      0 &      0 &      0 &      0 &     \n &      0 &      0 &      0 \\
      X_6^2 &      0 &      0 &      0 &      0 &      0 &      0 &      0 &      0 &      0 &      0 &      0 &      0 &      0 &      0 &     \n &      0 &     \n &      0 \\
      X_6^3 &      0 &      0 &      0 &      0 &      0 &      0 &      0 &      0 &      0 &      0 &      1 &      0 &      0 &      0 &      0 &      0 &     \n &      0 \\
      X_7^1 &      0 &      1 &      0 &      0 &      0 &      0 &      0 &      0 &      0 &      0 &      0 &      0 &      0 &      0 &      0 &      1 &      0 &      0 \\
      X_7^2 &      0 &      0 &      0 &      0 &      0 &      0 &      1 &      0 &      0 &      0 &      0 &      0 &      0 &      0 &      0 &      0 &      0 &     \n \\
      X_7^3 &      0 &      0 &      0 &      0 &      0 &      0 &      0 &      0 &      0 &      0 &      0 &      0 &      0 &      0 &      0 &     \n &      0 &     \n \\
      X_8^1 &      0 &      0 &      1 &      0 &      0 &      0 &      0 &      0 &      0 &      0 &      0 &      0 &     \n &      0 &      0 &      0 &      0 &      0 \\
      X_8^2 &      0 &      0 &      0 &      0 &      0 &      0 &      0 &      0 &      0 &      0 &      0 &      0 &     \n &      0 &      0 &      0 &      0 &      1 \\
      X_8^3 &      0 &      0 &      0 &      0 &      0 &      0 &      0 &      0 &      0 &      0 &      0 &      1 &      0 &      0 &      0 &      0 &      0 &      1 \\
      X_9^1 &      0 &      0 &      0 &      1 &      0 &      0 &      0 &      0 &      0 &      0 &      0 &      0 &      0 &     \n &      0 &      0 &      0 &      0 \\
      X_9^2 &      0 &      0 &      0 &      0 &      0 &      0 &      0 &      0 &      0 &      0 &      0 &      0 &      0 &     \n &      0 &      0 &      1 &      0 \\
      X_9^3 &      0 &      0 &      0 &      0 &      0 &      0 &      0 &      1 &      0 &      0 &      0 &      0 &      0 &      0 &      0 &      0 &      1 &      0 \\
    }
  \]

  \[
    \setlength{\arraycolsep}{1pt}
    \renewcommand{\arraystretch}{1.0}
    \tiny
    d_1(K_{3,3}) =
    \kbordermatrix{
          & X_1^1 & X_1^2 & X_1^3 & X_2^1 & X_2^2 & X_2^3 & X_3^1 & X_3^2 & X_3^3 & X_4^1 & X_4^2 & X_4^3 & X_5^1 & X_5^2 & X_5^3 & X_6^1 & X_6^2 & X_6^3 & X_7^1 & X_7^2 & X_7^3 & X_8^1 & X_8^2 & X_8^3 & X_9^1 & X_9^2 & X_9^3 \\
      Y_1 &     1 &     0 &     0 &    \n &     0 &     0 &     0 &    \n &     0 &    \n &     0 &     0 &     0 &     0 &     0 &     1 &    \n &     1 &     0 &    \n &     0 &    \n &     0 &    \n &     0 &     0 &     0 \\
      Y_2 &     0 &     1 &     0 &     0 &     0 &     0 &     0 &     0 &    \n &     0 &    \n &     0 &     0 &     0 &    \n &     0 &     0 &     0 &     0 &     0 &    \n &    \n &     0 &    \n &     1 &     0 &     0 \\
      Y_3 &     0 &     0 &     1 &     0 &     0 &     0 &    \n &     1 &    \n &     0 &     0 &    \n &     0 &     0 &    \n &     0 &     0 &     0 &     1 &     0 &     0 &     0 &     0 &     0 &     1 &     0 &     0 \\
      Y_4 &     0 &     0 &     0 &    \n &     0 &     0 &     0 &     0 &     0 &    \n &     0 &     0 &     0 &     0 &     0 &     0 &     0 &     0 &     0 &    \n &     0 &     0 &    \n &     0 &     0 &     0 &     0 \\
      Y_5 &     0 &     0 &     0 &     0 &     0 &     0 &     0 &     0 &     0 &     0 &    \n &     0 &     1 &     0 &     0 &     0 &     1 &     0 &     0 &     0 &    \n &     0 &    \n &     0 &     0 &     1 &     0 \\
      Y_6 &     0 &     0 &     0 &     0 &     0 &     0 &    \n &     0 &     0 &     0 &     0 &    \n &     0 &     0 &     0 &     0 &     0 &     1 &     0 &     0 &     0 &     0 &     0 &     0 &     0 &     1 &     0 \\
      Y_7 &     0 &     0 &     0 &     0 &     1 &     0 &     0 &     0 &     0 &     0 &     0 &     0 &     0 &     1 &     0 &     0 &     1 &     0 &     0 &     0 &     0 &     0 &     0 &     0 &     0 &     0 &     1 \\
      Y_8 &     0 &     0 &     0 &     0 &     0 &     1 &     0 &    \n &     0 &     0 &     0 &     0 &     0 &     0 &     0 &     0 &     0 &     1 &     0 &    \n &     0 &     0 &     0 &    \n &     0 &     0 &     1 \\
      Y_9 &     0 &     0 &     0 &     0 &     0 &     0 &     0 &     0 &    \n &     0 &     0 &     0 &     0 &     0 &    \n &     0 &     0 &     0 &     0 &     0 &    \n &     0 &     0 &    \n &     0 &     0 &     0 \\
    }
  \]
  It is easily verified that $d_1\circ d_2 =0$, $\dim \ker d_1(K_{3,3}) = 18$, and $\mathrm{rank}\, d_2(K_{3,3}) = 18$.
  Let $g = W_{16}-W_{17}+W_{18}+W_{19} -W_{24}-W_{25}+W_{27}+W_{29} +W_{34}-W_{35}+W_{36}+W_{38} +W_{48}+W_{49} +W_{56}+W_{57} -W_{69}+W_{78}\in C_2(K_{3,3})$ and $h = X_6^3-X_7^3+X_8^3-X_9^2 \in C_1(K_{3,3})$.
  We note that $h\notin \im d_2$, $d_2(g)=2h$ and $d_1(h)=0$, so $h$ generates $\bbZ_2$-torsion in $H_{1}(K_{3,3};\bbZ)$.
\end{proof}

\section{Computational data and conjectures}
\label{computations}
In this section, we provide the results of homology computations over $\bbZ$ for all connected graphs up to 6 vertices.
The program we used to obtain these is available at \cite{code}.
Given a graph $G$, the homology $H_*(G;\ZZ)$ is listed using the notation $H_0(G;\ZZ), H_1(G;\ZZ), \dots, H_k(G;\ZZ)$ where $k$ is maximal among indices with nonzero homology.

\begin{longtable}{|c|Sc|>{\centering\arraybackslash}m{5.61cm}||c|Sc|>{\centering\arraybackslash}m{5.61cm}|}
\hline
{\scriptsize no.} & $G$ & $H_*(G,\ZZ)$ as a $\bbZ$-module & {\scriptsize no.} &$G$ & $H_*(G,\ZZ)$ as a $\bbZ$-module \\ \hline\hline
\endhead
{\scriptsize 1} & $\begin{tikzpicture}[scale=0.18, baseline={([yshift=-0.3em]current bounding box.center)}]
\SetVertexSimple[MinSize=3pt, InnerSep=0pt]
\Vertex[NoLabel,x=2.5cm,y=2.5cm]{v0}
\end{tikzpicture}$ & $ \ZZ $ & & & \\ \hline
\hline
{\scriptsize 2} & $\begin{tikzpicture}[scale=0.18, baseline={([yshift=-0.3em]current bounding box.center)}]
\SetVertexSimple[MinSize=3pt, InnerSep=0pt]
\Vertex[NoLabel,x=5.0cm,y=2.5cm]{v0}
\Vertex[NoLabel,x=0.0cm,y=2.5cm]{v1}
\Edge[](v0)(v1)
\end{tikzpicture}$ & $ \ZZ$ && & \\ \hline
\hline
{\scriptsize 3} & $\begin{tikzpicture}[scale=0.18, baseline={([yshift=-0.3em]current bounding box.center)}]
\SetVertexSimple[MinSize=3pt, InnerSep=0pt]
\Vertex[NoLabel,x=2.5cm,y=5.0cm]{v0}
\Vertex[NoLabel,x=0.0cm,y=0.0cm]{v1}
\Vertex[NoLabel,x=5.0cm,y=0.0cm]{v2}
\Edge[](v0)(v2)
\Edge[](v1)(v2)
\end{tikzpicture}$ & $ \ZZ$ & {\scriptsize 4} & $\begin{tikzpicture}[scale=0.18, baseline={([yshift=-0.3em]current bounding box.center)}]
\SetVertexSimple[MinSize=3pt, InnerSep=0pt]
\Vertex[NoLabel,x=2.5cm,y=5.0cm]{v0}
\Vertex[NoLabel,x=0.0cm,y=0.0cm]{v1}
\Vertex[NoLabel,x=5.0cm,y=0.0cm]{v2}
\Edge[](v0)(v1)
\Edge[](v0)(v2)
\Edge[](v1)(v2)
\end{tikzpicture}$ & $ \ZZ,  \ZZ^{2}$ \\
\hline
\end{longtable}
\subsection{Some conjectures on \texorpdfstring{$q$}{q}-degree zero homology}

Based on observations made from the table in Appendix \ref{computations}, along with some computations done by hand, we make the following conjectures.

\begin{conjecture}
  \label{conj.one}
  The $q$-degree zero homology of the graphs $G_1$ and $G_2$ are
  \begin{center}
  \begin{tabular}{|Sc|Sc|Sc|}
    \hline
    $i$ & $H_i(G_1;\bbC)$ as a $\bbC[\fS_6]$-module & $H_i(G_2;\bbC)$ as a $\bbC[\fS_6]$-module \\
    \hline
    \hline
    $0$ & $\calS_{(1^6)}$                    & $\calS_{(1^6)}$ 
    \\
    \hline
    $1$ & $\calS_{(21^4)}^{\oplus4}\oplus
            \calS_{(2^21^2)}^{\oplus4} \oplus
            \calS_{(2^3)}^{\oplus2}$     
        & $\calS_{(21^4)}^{\oplus4}\oplus
          \calS_{(2^21^2)}^{\oplus3} \oplus
          \calS_{(2^3)}$                        
    \\
    \hline
    $2$ & $\calS_{(2^21^2)}\oplus
          \calS_{(2^3)}\oplus
          \calS_{(31^3)}^{\oplus5}\oplus
          \calS_{(321)}^{\oplus6}\oplus
          \calS_{(3^2)}$
        & $\calS_{(31^3)}^{\oplus5} \oplus
          \calS_{(321)}^{\oplus6} \oplus
          \calS_{(3^2)}$                       
    \\
    \hline
    $3$ & $\calS_{(41^2)}^{\oplus2}\oplus\calS_{(42)}^{\oplus2}$                   & $\calS_{(41^2)}^{\oplus2}\oplus\calS_{(42)}^{\oplus2}$               
    \\
    \hline
  \end{tabular}
  \end{center}
\end{conjecture}

\begin{conjecture}
  The $q$-degree zero homology of the graphs $G_3$ and $G_4$ are
  \begin{center}
  \begin{tabular}{|Sc|Sc|Sc|}
    \hline
    $i$ & $H_i(G_3;\bbC)$ as a $\bbC[\fS_6]$-module & $H_i(G_4;\bbC)$ as a $\bbC[\fS_6]$-module \\
    \hline
    \hline
    $0$ & $\calS_{(1^6)}$                    & $\calS_{(1^6)}$ 
    \\
    \hline
    $1$ & $\calS_{(21^4)}^{\oplus6} \oplus
          \calS_{(2^21^2)}^{\oplus3} \oplus
          \calS_{(2^3)}$  
        & $\calS_{(21^4)}^{\oplus6} \oplus \calS_{(2^21^2)}^{\oplus2}$                        
    \\
    \hline
    $2$ & $\calS_{(2^21^2)}^{\oplus2} \oplus
          \calS_{(2^3)}^{\oplus2} \oplus
          \calS_{(31^3)}^{\oplus13} \oplus
          \calS_{(321)}^{\oplus10} \oplus
          \calS_{(3^2)}$
        & $\calS_{(2^21^2)} \oplus \calS_{(2^3)}\oplus
           \calS_{(31^3)}^{\oplus13} \oplus \calS_{(321)}^{\oplus10}\oplus
          \calS_{(3^2)}$                       
    \\
    \hline
    $3$ & $\calS_{(41^2)}^{\oplus12}\oplus \calS_{(42)}^{\oplus8}$                  
        & $\calS_{(41^2)}^{\oplus12}\oplus \calS_{(42)}^{\oplus8}$               
    \\
    \hline
    $4$ & $\calS_{(51)}^{\oplus4}$                  
        & $\calS_{(51)}^{\oplus4}$              
    \\
    \hline
  \end{tabular}
  \end{center}
\end{conjecture}

\begin{conjecture}
  The $q$-degree zero homology of the graphs $G_5$ and $G_6$ are
  \begin{center}
  \begin{tabular}{|Sc|Sc|Sc|}
    \hline
    $i$ & $H_i(G_5;\bbC)$ as a $\bbC[\fS_6]$-module & $H_i(G_6;\bbC)$ as a $\bbC[\fS_6]$-module \\
    \hline
    \hline
    $0$ & $\calS_{(1^6)}$                    & $\calS_{(1^6)}$ 
    \\
    \hline
    $1$ & $\calS_{(21^4)}^{\oplus6} \oplus
          \calS_{(2^21^2)}$  
        & $\calS_{(21^4)}^{\oplus6}$
    \\
    \hline
    $2$ & $\calS_{(2^21^2)} \oplus
          \calS_{(2^3)} \oplus
          \calS_{(31^3)}^{\oplus14} \oplus
          \calS_{(321)}^{\oplus8}$
        & $\calS_{(2^3)} \oplus
          \calS_{(31^3)}^{\oplus14} \oplus
          \calS_{(321)}^{\oplus8}$                       
    \\
    \hline
    $3$ & $\calS_{(41^2)}^{\oplus15}\oplus \calS_{(42)}^{\oplus9}$                  
        & $\calS_{(41^2)}^{\oplus15}\oplus \calS_{(42)}^{\oplus9}$               
    \\
    \hline
    $4$ & $\calS_{(51)}^{\oplus6}$                  
        & $\calS_{(51)}^{\oplus6}$              
    \\
    \hline
  \end{tabular}
\end{center}
\end{conjecture}

\begin{conjecture}
  Let $C_n$ be the cycle graph with $n$ vertices.
  Then $H_{i}(C_n;\ZZ)$ is a free $\ZZ$-module of rank ${\binom{n-1}{i}}$ for $0\leq i\leq n-1$.
\end{conjecture}

Given a graph $G$, let $\textnormal{span}_0(G)$ denote the homological span of the degree 0 chromatic symmetric homology of $G$. Since the minimal grading with nonzero homology is always $0$, we have $\textnormal{span}_0(G)=k+1$ where $k$ is maximal among indices such that $H_k(G;\ZZ)\neq 0$.

\begin{conjecture}
  Given any graph $G$, chromatic symmetric homology groups $H_\ell(G;\bbC)$ are non-trivial for all $0\leq \ell\leq \textnormal{span}_0(G)-1$.
\end{conjecture}

A {\em cut vertex} in a graph $G$ is a vertex $v$ such that $G-v$ has more connected components than $G$.
A {\em block} of a graph $G$ is a maximal connected subgraph of $G$ with no cut-vertex.
\begin{conjecture}\label{spanconj}
  Let $G$ be a graph with $n$ vertices and $m$ edges, and let $b$ denote the number of blocks of $G$.
  Then $n-b\leq \textnormal{span}_0(G)\leq n-1.$
\end{conjecture}

\begin{remark}
  Conjecture \ref{spanconj} holds in the case of chromatic homology by \cite[Theorem 44]{sazdanovic2018patterns} and \cite[Theorem 4]{HGRP}. The upper bound follows from the deletion-contraction long exact sequence for chromatic homology, which does not exist for chromatic symmetric homology.
\end{remark}


\begin{conjecture}
  Let $G$ be a graph and suppose that, for each $i$, the free part of $H_{i}(G;\ZZ)$ as a $\ZZ$-module has rank $n_i$.
  Then there exists $0\leq\ell\leq k$ such that 
$
    n_0\leq n_1\leq\dots\leq n_\ell\geq n_{\ell+1}\geq n_{\ell+2}\geq\dots\geq n_k.
$
  In other words, ranks of degree 0 chromatic symmetric homology are unimodal.

\end{conjecture}

\end{document}

%% file: savebox.tex
\newsavebox\graphUZVW
\begin{lrbox}{\graphUZVW}
\begin{tikzpicture}[scale=.3, baseline=-2.7pt]
\draw[dotted,thick] (0,0) circle (1cm);
\draw[fill=black] (.7,.7) circle (.1cm);
\draw[fill=black] (.7,-.7) circle (.1cm);
\draw[fill=black] (-.7,.7) circle (.1cm);
\draw[fill=black] (-.7,-.7) circle (.1cm);
\coordinate (v) at (.7,.7);
\coordinate (w) at (.7,-.7);
\coordinate (u) at (-.7,.7);
\coordinate (z) at (-.7,-.7);
\draw[thick] (u)--(z);
\draw[thick] (w)--(v);
\end{tikzpicture}
\end{lrbox}

\newsavebox\graphUZVWlabeled
\begin{lrbox}{\graphUZVWlabeled}
\begin{tikzpicture}[scale=.3, baseline=-2.7pt]
\draw[dotted,thick] (0,0) circle (1cm);
\draw[fill=black]  (.7,.7) circle (.1cm);
\draw[fill=black] (.7,-.7) circle (.1cm);
\draw[fill=black] (-.7,.7) circle (.1cm);
\draw[fill=black] (-.7,-.7) circle (.1cm);
\node at (1.4,.7){$v$};
\node at (-1.4,.7){$u$};
\node at (1.4,-.7){$w$};
\node at (-1.4,-.7){$z$};
\coordinate (v) at (.7,.7);
\coordinate (w) at (.7,-.7);
\coordinate (u) at (-.7,.7);
\coordinate (z) at (-.7,-.7);
\draw[thick] (u)--(z);
\draw[thick] (w)--(v);
\end{tikzpicture}
\end{lrbox}

\newsavebox\graphUWUZVW
\begin{lrbox}{\graphUWUZVW}
\begin{tikzpicture}[scale=.3, baseline=-2.7pt]
\draw[dotted,thick] (0,0) circle (1cm);
\draw[fill=black] (.7,.7) circle (.1cm);
\draw[fill=black] (.7,-.7) circle (.1cm);
\draw[fill=black] (-.7,.7) circle (.1cm);
\draw[fill=black] (-.7,-.7) circle (.1cm);
\coordinate (v) at (.7,.7);
\coordinate (w) at (.7,-.7);
\coordinate (u) at (-.7,.7);
\coordinate (z) at (-.7,-.7);
\draw[thick] (u)--(z);
\draw[thick] (u)--(w);
\draw[thick] (w)--(v);
\end{tikzpicture}
\end{lrbox}

\newsavebox\graphUZVWVZ
\begin{lrbox}{\graphUZVWVZ}
\begin{tikzpicture}[scale=.3, baseline=-2.7pt]
\draw[dotted,thick] (0,0) circle (1cm);
\draw[fill=black] (.7,.7) circle (.1cm);
\draw[fill=black] (.7,-.7) circle (.1cm);
\draw[fill=black] (-.7,.7) circle (.1cm);
\draw[fill=black] (-.7,-.7) circle (.1cm);
\coordinate (v) at (.7,.7);
\coordinate (w) at (.7,-.7);
\coordinate (u) at (-.7,.7);
\coordinate (z) at (-.7,-.7);
\draw[thick] (u)--(z);
\draw[thick] (v)--(z);
\draw[thick] (w)--(v);
\end{tikzpicture}
\end{lrbox}

\newsavebox\graphUZVWWZ
\begin{lrbox}{\graphUZVWWZ}
\begin{tikzpicture}[scale=.3, baseline=-2.7pt]
\draw[dotted,thick] (0,0) circle (1cm);
\draw[fill=black] (.7,.7) circle (.1cm);
\draw[fill=black] (.7,-.7) circle (.1cm);
\draw[fill=black] (-.7,.7) circle (.1cm);
\draw[fill=black] (-.7,-.7) circle (.1cm);
\coordinate (v) at (.7,.7);
\coordinate (w) at (.7,-.7);
\coordinate (u) at (-.7,.7);
\coordinate (z) at (-.7,-.7);
\draw[thick] (u)--(z);
\draw[thick] (z)--(w);
\draw[thick] (w)--(v);
\end{tikzpicture}
\end{lrbox}

\newsavebox\graphUWUZVWWZ
\begin{lrbox}{\graphUWUZVWWZ}
\begin{tikzpicture}[scale=.3, baseline=-2.7pt]
\draw[dotted,thick] (0,0) circle (1cm);
\draw[fill=black] (.7,.7) circle (.1cm);
\draw[fill=black] (.7,-.7) circle (.1cm);
\draw[fill=black] (-.7,.7) circle (.1cm);
\draw[fill=black] (-.7,-.7) circle (.1cm);
\coordinate (v) at (.7,.7);
\coordinate (w) at (.7,-.7);
\coordinate (u) at (-.7,.7);
\coordinate (z) at (-.7,-.7);
\draw[thick] (u)--(w);
\draw[thick] (u)--(z);
\draw[thick] (z)--(w);
\draw[thick] (w)--(v);
\end{tikzpicture}
\end{lrbox}

\newsavebox\graphUZVWVZWZ
\begin{lrbox}{\graphUZVWVZWZ}
\begin{tikzpicture}[scale=.3, baseline=-2.7pt]
\draw[dotted,thick] (0,0) circle (1cm);
\draw[fill=black] (.7,.7) circle (.1cm);
\draw[fill=black] (.7,-.7) circle (.1cm);
\draw[fill=black] (-.7,.7) circle (.1cm);
\draw[fill=black] (-.7,-.7) circle (.1cm);
\coordinate (v) at (.7,.7);
\coordinate (w) at (.7,-.7);
\coordinate (u) at (-.7,.7);
\coordinate (z) at (-.7,-.7);
\draw[thick] (v)--(z);
\draw[thick] (u)--(z);
\draw[thick] (z)--(w);
\draw[thick] (w)--(v);
\end{tikzpicture}
\end{lrbox}

\newsavebox\infinitezero
\begin{lrbox}{\infinitezero}
\begin{tikzpicture}[scale=.75]
\node at (1,-1.2){$G$};   
	\vertex[fill](v1) at (0,0){};
    \vertex[fill](v2) at (2,0){};
    \vertex[fill](v3) at (0,2){};
    \vertex[fill](v4) at (2,2){};
    \vertex[fill](v5) at (-0.5,1){};
    \vertex[fill](v6) at (2.5,1){};
    \vertex[fill](v7) at (3,1){};    
    \draw (v1)--(v5)--(v3)--(v1);
    \draw (v4)--(v6);
    \draw (v4)--(v7);
    \draw (v2)--(v4);
    \draw (v2)--(v6);
    \draw (v2)--(v7);    
    \draw (v6) to [out=150,in=30] (v5);
    \draw (v7) to [out=210,in=-30] (v5);
\end{tikzpicture}
\end{lrbox}

\newsavebox\infiniteone
\begin{lrbox}{\infiniteone}
\begin{tikzpicture}[scale=.75]
\node at (1,-1.2){$G_7$};   
	\vertex[fill](v1) at (0,0){};
    \vertex[fill](v2) at (2,0){};
    \vertex[fill](v3) at (0,2){};
    \vertex[fill](v4) at (2,2){};
    \vertex[fill](v5) at (-0.5,1){};
    \vertex[fill](v6) at (2.5,1){};
    \vertex[fill](v7) at (3,1){};
    \draw (v1)--(v5)--(v3)--(v1);
    \draw (v4)--(v2);
    \draw (v1)--(v4);
    \draw (v1)--(v2);
    \draw (v4)--(v6)--(v2);
    \draw (v4)--(v7)--(v2);
    \draw (v6) to [out=150,in=30] (v5);
    \draw (v7) to [out=210,in=-30] (v5);
    \draw (-1,1) circle (.5cm);
    \node at (-1,1) {$P_1$};
    \draw (3.5,1) circle (.5cm);
    \node at (3.5,1) {$P_2$};
\end{tikzpicture}
\end{lrbox}

\newsavebox\infinitetwo
\begin{lrbox}{\infinitetwo}
\begin{tikzpicture}[scale=.75]
\node at (1,-1.2){$G_8$};   
	\vertex[fill](v1) at (0,0){};
    \vertex[fill](v2) at (2,0){};
    \vertex[fill](v3) at (0,2){};
    \vertex[fill](v4) at (2,2){};
    \vertex[fill](v5) at (-0.5,1){};
    \vertex[fill](v6) at (2.5,1){};
    \vertex[fill](v7) at (3,1){};
    \draw (v1)--(v5)--(v3)--(v1);
    \draw (v4)--(v2);
    \draw (v2)--(v3);
    \draw (v1)--(v2);
    \draw (v4)--(v6)--(v2);
    \draw (v4)--(v7)--(v2);
    \draw (v6) to [out=150,in=30] (v5);
    \draw (v7) to [out=210,in=-30] (v5);
    \draw (-1,1) circle (.5cm);
    \node at (-1,1) {$P_1$};
    \draw (3.5,1) circle (.5cm);
    \node at (3.5,1) {$P_2$};
\end{tikzpicture}
\end{lrbox}